\documentclass[a4paper,9pt]{article}

\usepackage[dvips]{color}

\usepackage{latexsym}
\usepackage{amsmath,verbatim}
\usepackage{graphics}

\usepackage[all]{xy}

\usepackage[vcentermath,enableskew]{youngtab}

\usepackage{amsthm}
\usepackage{amssymb}
\usepackage{makeidx}
\usepackage{multicol}

\usepackage[all]{xy}

\newfont{\Fr}{eufm10}
\newfont{\Sc}{eusm10}
\newfont{\Bb}{msbm10}
\newfont{\Am}{msam10}
\newfont{\am}{msam7}

\numberwithin{equation}{section}
\newtheorem{theorem}{Theorem}[section]
\newtheorem{proposition}[theorem]{Proposition}
\newtheorem{lemma}[theorem]{Lemma}
\newtheorem{corollary}[theorem]{Corollary}
\newtheorem{claim}{Claim}{\bf}{\it}

\newtheorem{ftheorem}{Theorem}{\bf}{\it}
{\bf}{\it}
\theoremstyle{definition}
\newtheorem{definition}[theorem]{Definition}

\newtheorem{convention}[theorem]{Convention}
\newtheorem{construction}[theorem]{Construction}
{\bf}{\rm}

\theoremstyle{remark}
\newtheorem{example}[theorem]{Example}
\newtheorem{remark}[theorem]{Remark}
\newtheorem{fremark}[ftheorem]{Remark}
\newtheorem{algorithm}[theorem]{Algorithm}

\newtheorem{definition and corollary}[theorem]{Definition and Corollary}

\newtheorem{fexample}[ftheorem]{Example}{\it}{\rm}

\newcommand{\Ind}{\mbox{\rm Ind}}

\newcommand{\MID}{\! \! \mid}

\newcommand{\h}{\mathfrac{\h}}

\newlength{\tabwidth}
\newlength{\tabheight}
\newlength{\tabrule}
\newlength{\tabwidthx}
\newlength{\tabheightx}

\def\gentabbox#1#2#3#4{\vbox to \tabheight{\setlength{\tabrule}{#3}%
  \setlength{\tabwidthx}{#1\tabwidth}\addtolength{\tabwidthx}{\tabrule}%

\setlength{\tabheightx}{#2\tabheight}\addtolength{\tabheightx}{-\tabheight}%
  \hbox to #1\tabwidth{%
    \hspace{-0.5\tabrule}\rule{\tabrule}{#2\tabheight}\hspace{-\tabrule}%
    \vbox to #2\tabheight{\hsize=\tabwidthx%
      \vspace{-0.5\tabrule}\hrule width\tabwidthx height\tabrule%
      \vspace{-0.5\tabrule}\vfil%
      \hbox to \tabwidthx{\hss#4\hss}%
        \vfil\vspace{-0.5\tabrule}%
      \hrule width\tabwidthx height\tabrule\vspace{-0.5\tabrule}}%
    \hspace{-\tabrule}\rule{\tabrule}{#2\tabheight}\hspace{-0.5\tabrule}}%
  \vspace{-\tabheightx}}}
\def\genblankbox#1#2{\vbox to \tabheight{\vfil\hbox to
#1\tabwidth{\hfil}}}
\def\tabbox#1#2#3{\gentabbox{#1}{#2}{0.4pt}{\strut #3}}

\catcode`\:=13 \catcode`\.=13 \catcode`\;=13 
\catcode`\>=13 \catcode`\^=13
\def:#1\\{\hbox{$#1$}}
\def.#1{\tabbox{1}{1}{$#1$}}
\def>#1{\tabbox{2}{1}{$#1$}}
\def^#1{\tabbox{1}{2}{$#1$}}
\def;{\genblankbox{1}{1}\relax}
\catcode`\:=12 \catcode`\.=12 \catcode`\;=12 
\catcode`\>=12 \catcode`\^=7

\newenvironment{tableau}{\bgroup\catcode`\:=13 \catcode`\.=13
  \catcode`\;=13 \catcode`\>=13 \catcode`\^=13
  \setlength{\tabheight}{3ex}\setlength{\tabwidth}{3ex}%
  \def\b##1##2##3{\gentabbox{##1}{##2}{1.2pt}{\vbox{##3}}}%
  \def\n##1##2##3{\gentabbox{##1}{##2}{0.4pt}{\vbox{##3}}}%
  \vbox\bgroup\offinterlineskip}{\egroup\egroup}

\makeindex

\begin{document}

\title{Tempered modules in exotic Deligne-Langlands correspondence}
\author{Dan \textsc{Ciubotaru}\footnote{Department of Mathematics, University of Utah, Salt Lake City, UT 84112, USA, \texttt{ciubo@math.utah.edu.}}
\ and\  Syu \textsc{Kato}\footnote{Research Institute for Mathematical Sciences, Kyoto
  University, Oiwake Kita-Shirakawa Sakyo Kyoto 606-8502, Japan, \texttt{kato@kurims.kyoto-u.ac.jp}.}}

\maketitle

\begin{abstract}
The main purpose of this paper is to produce a geometric realization for tempered modules of
the affine Hecke algebra of type $C_n$ with arbitrary, non-root
of unity, unequal parameters, using the exotic Deligne-Langlands correspondence (\cite{K}). Our classification has several applications to the structure of the tempered Hecke algebra
modules. In particular, we provide a geometric and a combinatorial
classification of discrete series which contain the $\mathsf{sgn}$
representation of the Weyl group, equivalently, via the Iwahori-Matsumoto involution, of spherical cuspidal modules. This last combinatorial classification was expected from \cite{HO}, and determines the $L^2$-solutions for the Lieb-McGuire system.
\end{abstract}


\section*{Introduction} The main purpose of this paper is to study tempered modules 
of the affine Hecke algebra $\mathbb H$ of type $C_n$ (Definition
\ref{Checke}) using the framework in \cite{K}. Here $\mathbb H$ is an algebra over $\mathcal A=\mathbb C[\mathbf{q}_0^{\pm
  1},\mathbf{q}_1^{\pm 1}, \mathbf{q}_2^{\pm 1}],$ where
$\mathbf{q}_0,\mathbf{q}_1,\mathbf{q}_2$ are three indeterminate
parameters.
\medskip

The main case we deal with in this paper is the affine Hecke algebra $\mathbb H_{n,m}$ of type {$C_n$}
with unequal parameters (see \S \ref{remHecke}). The algebra $\mathbb H$ specializes to $\mathbb H _{n,m}$ by the specialization $(\mathbf{q}_0,\mathbf{q}_1,\mathbf{q}_2 ) \mapsto (-1,q^{m},q)$ with $q \in \mathbb C ^{\times}$ and $m \in \mathbb R$. These algebras, for special values of $m$, appear as convolution algebras in the theory of
$p$-adic groups. For example, if $m=1$ or $m=1/2$ (and an appropriate
value of $q$), they correspond to
Iwahori-Hecke algebras for split $p$-adic $SO(2n+1)$ or $PSp(2n)$,
respectively. More generally, when $m \in\mathbb Z_{\ge 0} +\epsilon$,
where $\epsilon\in\{0,1/2,1/4\}$, (graded versions of) these algebras appear from
representations of $p$-adic groups with unipotent cuspidal support, in the sense of \cite{L5}.

Set $G=Sp(2n,\mathbb C)$, let $T$ denote its diagonal
torus, and let $W = N _G ( T ) / T$. The algebra $\mathbb H$ has a large abelian subalgebra
$\mathcal A\otimes R(T)$, and the
tempered and discrete series $\mathbb H$-modules are defined by the
Casselman criterion for the generalized $R(T)$-weights (Definition
\ref{d:temp}). By a result of Bernstein and Lusztig, the center of
$\mathbb H$ is $\mathcal A\otimes R(T)^W.$ Therefore, the central
characters of irreducible $\mathbb H$-modules are in one-to-one
correspondence with $W$-conjugacy classes of semisimple elements of the form
$a=(s,q_0,q_1,q_2)$, where $s\in T.$ The $\mathbb H$-action on a module with central character $a$ factors through a finite-dimensional algebra $\mathbb H _a$. We say that the central character
$a$ is positive real if $-q_0,q_1 \in \mathbb R _{>0} ,q_2\in\mathbb R_{>1},$ and $s$ is
hyperbolic (see Definition \ref{d:real}). The algebra $\mathbb H_a$ contains a copy of $\mathbb C[W]$, for positive real $a$ (see \cite{L2}).

We are interested in the identification of tempered $\mathbb H$-modules for non-root of unity parameters. 
Via Lusztig's reduction (\cite{L2}), this can essentially be reduced
to the determination of 
the tempered modules and discrete series with positive real central
character for the case of $\mathbb
H_{n,m}$. 

For the algebra $\mathbb H_{n,m}$, we say that the parameter $m \in \mathbb R$ is {\it generic} if
$m\notin\frac 12\mathbb Z.$ By \cite{O,OS} (also
\cite{L6} for $m\in\mathbb Z+1/4$), we know that the 
central characters which allow discrete series in the positive real generic range
are in one-to-one correspondence with partitions $\sigma$ of $n$ (see
Definition \ref{d:dist}). 

\begin{fexample}[$n=2$ case]
There are two central characters which afford discrete series corresponding to $\sigma_1=(2)$ and $\sigma_2=(1^2)$. The first discrete series is always the Steinberg module (corresponding to the $\mathsf{sgn}$ representation
of the Weyl group), but the dimension and the $W$-module structure of
the second discrete series depend on $m.$
\end{fexample}

We call a $s_\sigma \in T$ which affords discrete series distinguished.

Our main result is the description of parameters corresponding to discrete series $\mathbb H _{n,m}$-modules within the framework of the exotic Deligne-Langlands correspondence (eDL for short). To be more precise, let us recall it briefly (c.f. \S \ref{sec:irrH} or \cite{K}). Let $V_1=\mathbb C^{2n}$ be the vector representation of $G$, and $V_2=\wedge^2 V_1.$
Set $\mathbb V=V_1^{\oplus 2}\oplus V_2.$ This is a representation of
$\mathcal G=G\times (\mathbb C^\times)^3$, where $G$ acts diagonally
on $(V_1,V_1,V_2)$ and $(c_0,c_1,c_2)\in (\mathbb C^\times)^3$ acts by
multiplication by $(c_0^{-1},c_1^{-1},c_2^{-1}).$ The eDL correspondence is stated as a one-to-one correspondence
\begin{equation}
G(a)\backslash\mathbb V^a\longleftrightarrow \mathsf{Irrep} \mathbb
H_a, \quad X\mapsto L_{(a,X)},
\end{equation}
where $G(a)$ is the centralizer of $s$ in $G$ and $\mathbb V ^a$ is the set of $a$-fixed points in $\mathbb V$.

There is a combinatorial parameterization of the left hand side in terms of marked
partitions (see \S \ref{sec:1.3}). The right
hand side is the set of irreducible $\mathbb H$-modules with central
character $a.$ The irreducible module $L_{(a,X)}$ is obtained as a quotient of a standard geometric module $M_{(a,X)}.$ We call $(a, X)$ the eDL parameter of $L _{(a,X)}$. There are some
remarks to be made at this point:

\begin{enumerate}
\item There are no local systems in this picture: the isotropy
group of every $X$ is connected;
\item There exists an exotic Springer correspondence
  (\cite{K}). Moreover, the homology groups of  (classical)
  Springer fibers of $\mathop{Sp} ( 2n )$ and $\mathop{SO} ( 2n + 1 )$
  can be realized via the homology of a suitable exotic Springer fiber
  (c.f. Corollary \ref{c:springer});

\item For the Hecke algebras which are known to appear in the representation theory
  of $p$-adic groups, the Deligne-Langlands-Lusztig correspondence (DLL for
  short) was established in \cite{KL2,L2,L3,L6}. The connection
between the eDL and DLL correspondences is non-trivial. In particular, the ``lowest $W$-types" of a fixed irreducible module differ between the eDL/DLL correspondences;
\item In the DLL correspondence, we can specify a discrete series by its ``lowest $W$-type". However, our ``eDL lowest $W$-type" does not single out discrete series. Hence, we need the full eDL parameter in order to specify discrete series.
\end{enumerate}

The third phenomenon gives a restriction on the $W$-character of discrete series which seems far from trivial (c.f. \S \ref{clrelation}).

\smallskip

In \S \ref{sec:alg}, we give an effective, simple algorithm which
produces from a distinguished central character $s_\sigma$, an element
$X_{\mathsf{out}(\sigma)}$ in $\mathbb V^a$.

\begin{ftheorem}[=Theorem \ref{t:main}]\label{fmain} Let $\sigma$ be a partition of $n$. Let $a=(s_\sigma, \vec
  q)$ be the positive real generic central character, where
  {$\vec q=(-1,q^{m},q)$}, and 
   $s_\sigma$ is the distinguished semisimple element
  corresponding to $\sigma$. Then $( a, X_{\mathsf{out}(\sigma)} )$ is the eDL parameter of a unique discrete series $\mathbb H$-module with central character $a$.
\end{ftheorem}

We also give a geometric characterization of $\mathsf{out} ( \sigma )$ in \S \ref{sec:3.6}, as the minimal orbit with respect to certain conditions.

More generally, one can describe the tempered spectrum of $\mathbb
H_{n,m}$ with real positive generic parameter as follows.

\begin{ftheorem}[=Theorem \ref{indADt} and Corollary \ref{c:exh}]\label{findADt} If $n_1+n_2=n$, we set $\mathbb
  H^S_m :=\mathbb H_{n_1,m}^{\mathsf A} \otimes \mathbb
  H_{n_{2},m} \subset \mathbb H _{n,m}$, where $\mathbb
  H_{n_1,m}^{\mathsf A}$ is an affine Hecke algebra of $\mathop{GL} (
  n _1 )$. 
\begin{enumerate} 
\item[(i)] {Let $L _1 ^{\mathsf A}$ and $L _2$ be a tempered $\mathbb
    H_{n_1,m}^{\mathsf A}$-module with real central character and a discrete series $\mathbb H
    _{n_2, m}$-module with positive real generic parameter, respectively. Then,
 $$L := \mathsf{Ind} ^{\mathbb H _{n,m}} _{\mathbb H ^S _m} (L _1 ^{\mathsf A} \boxtimes L _2)$$
 is an irreducible tempered $\mathbb H _{n, m}$-module.}
\item[(ii)] Every irreducible tempered $\mathbb H_{n,m}$-module with 
positive real generic central character can be realized in a
  unique way as in (i).  
\end{enumerate}
\end{ftheorem}

The classification of tempered modules for the Hecke algebras of type
$\mathsf A$ is well-known from \cite{Z} and \cite{KL2}. Theorem
  \ref{findADt} was established before, by different methods, in the
  work of Opdam \cite{O}, Delorme-Opdam \cite{DO}, and Slooten \cite{Sl2}.

\begin{fexample}[Table of discrete series for $n=2$]
Let $\{\epsilon_1-\epsilon_2,\epsilon_1,\epsilon_2\}$ denote the positive roots {(of $\mathop{SO} (5, \mathbb C )$)}. Let $v _{\beta}$ denote a $T$-eigenvector of weight $\beta$ in $\mathbb V$. Assume the notation for $W$-representations in Remark \ref{r:bipartitions}. We have:

\begin{center}
\begin{tabular}{|c|c|c|c|}
\hline
$m$ &$(0,1/2)$ &$(1/2,1)$ &$(1,\infty)$\\
\hline
$X_{\mathsf{out}(\sigma_1)}$ &$v_{\epsilon_1-\epsilon_2}+v_{\epsilon_2}$ &$v_{\epsilon_1-\epsilon_2}+v_{\epsilon_2}$ &$v_{\epsilon_1-\epsilon_2}+v_{\epsilon_2}$\\
\hline
$X_{\mathsf{out}(\sigma_2)}$ &$v_{\epsilon_1-\epsilon_2}$
&$v_{\epsilon_1-\epsilon_2}+v_{\epsilon_1}$ &$v_{\epsilon_2}$\\
\hline
$\mathsf{ds}(\mathsf{out}(\sigma_2))|_W$ &$\{(1^2)(0)\}$
&$\{(1)(1)\}+\{(0)(1^2)\}$ &$\{(0)(2)\}$\\
\hline
\end{tabular}
\end{center}
\end{fexample}

To transfer our description from generic parameters to special parameters, we prove a continuity result of tempered modules, which is an algebraic analogue of a result of \cite{OS}.

\begin{ftheorem}[=Corollary \ref{spec}]\label{fspec}
Let $a ^t = a \exp ( \gamma t )$ be a one-parameter family of positive
real central characters depending on $t \in \mathbb R$ by
$$\gamma \in \mathfrak t \oplus \{ 0\} \oplus \mathbb R ^2 _{\ge 0} \subset \mathrm{Lie} ( T \times ( \mathbb C ^{\times} ) ^3 ).$$
Let $X_t\in\mathbb V^{a^t}$ be a family of exotic nilpotent elements
corresponding to the same marked partition $\tau.$ We assume that $a^t$ is generic for all $t\in
(-\epsilon,\epsilon)\setminus\{0\}$ for some small $\epsilon>0.$

\begin{enumerate}
\item[(i)] The module $L _{(a^0,X_0)}$ is an irreducible
  quotient of both of the two limit modules $\lim_{t\to 0^\pm} L_{(a^t, X_t)}$.
\item[(ii)] The module $L _{(a^0, X_0)}$ is tempered if $L _{(a^t,
    X_t)}$ is a tempered module in at least one of the regions
$0<t<\epsilon$ or $-\epsilon<t<0.$
\item[(iii)] The module $L_{(a^0, X_0)}$ is a discrete series if $L _{(a^t,
    X_t)}$ are discrete series for $t\in(-\epsilon,\epsilon)\setminus\{0\}.$
\end{enumerate}
\end{ftheorem}

\begin{fremark}\label{ex:deform} In general, the limit tempered module
  $\lim_{t\to 0} L_{(a^t,X_t)}$ is reducible. For
    example, let us consider $\mathbb H_{n,m}$ with $n=2,$ and $1/2<m<1$. There is one tempered $\mathbb H_{2,m}$-module $L _m$ with its $W_2$-structure $\{(0)(1^2)\}+\{(0)(2)\}+\{(1)(1)\}.$ We have
    $$\lim_{m\to 1^-} L _m \cong U _1 \oplus U _2, \hskip 3mm U _1 |_{W} \cong \{(0)(1^2)\}+\{(1)(1)\}, \text{ and }U _2 |_{W} \cong \{(0)(2)\},$$
    where $U_1,U_2$ are tempered modules of $\mathbb H_{2,1}$ (the
    affine Hecke algebra with equal parameters of type $B_2$). In the usual
    DLL correspondence, the tempered modules $U_1$
    and $U_2$ are parameterized by the same nilpotent adjoint orbit in
    $\mathfrak{so}(5),$ i.e., they are in the same L-packet.
\end{fremark}

The result of \cite{OS} guarantees that every tempered module arises
as $L _{(a^0,X_0)}$ via Theorem \ref{fspec}. This completes the
description of tempered modules in the eDL correspondence. It may be
worth mentioning that the classification of \cite{OS} in the case of
$\mathbb H_{n,m}$ is basically the same as a conjecture of Slooten
\cite{Sl2} 4.2.1 (i). It also shares certain aspects with the
  classification of the tempered and discrete series $\mathbb H
  _{n,m}$-modules by \cite{KL2,L2,L3,L6}. We provide a combinatorial identification
  of lowest $W$-types algorithms between \cite{Sl} and (the $\mathop{Spin} ( \ell
  )$-case of) \cite{L2} in \S \ref{sec:cuspidal} for the sake of
  completeness.

In addition, we include certain applications of this ''exotic''
classification for the $W$-structure of tempered modules in \S
4. Among these applications, we give a combinatorial
  classification of discrete series containing the $W$-type
  $\mathsf{sgn}$ (Theorem \ref{c:sgn}). This description follows from
  Theorem \ref{fmain} in conjunction with: 

\begin{ftheorem}[=Theorem \ref{sgnchar}]\label{fsgn}
Let $\sigma$ be a partition of $n$. The discrete series of $\mathbb H$
with central character $a = ( s _{\sigma}, \vec{q} )$ contains the
$\mathsf{sgn}$ $W$-type if and only if $G ( s _{\sigma} ) X _{\mathsf{out} ( \sigma )} \subset \mathbb V ^a$ is open dense.
\end{ftheorem}

\begin{fremark}Taking into account Theorem \ref{c:deform} and Heckman-Opdam
(see \cite{HO} Theorem 1.7 and the discussion around it), this
completes the classification of the $L^2$-solutions of the Lieb-McGuire system with attractive coupling parameters. Theorem \ref{c:sgn} itself
{fills out a missing piece in \cite{HO} (the unequal parameter case of Theorem 1.7 when the root system is of type $B _n$) and confirms a special part of a conjecture of Slooten \cite{Sl2} 4.2.1 (iii).}
 Recall that the affine Hecke algebra has the Iwahori-Matsumoto
  involution which interchanges modules containing the $\mathsf{sgn}$
  $W$-type and spherical modules, i.e., modules containing the $\mathsf{triv}$
  $W$-type. The images of discrete series containing the $\mathsf{sgn}$
  $W$-type under the Iwahori-Matsumoto involution are called spherical cuspidal in \cite{HO}. Therefore one
  may view Theorem \ref{c:sgn} as a classification of spherical
  cuspidal modules for the affine Hecke algebra $\mathbb H_{n,m}$ of type {$B_n/C_n$} with
  arbitrary parameter $m.$ When the Hecke algebra comes from a
  $p$-adic group, these should be examples of Arthur representations
  (in the sense of \cite{Ar}).
\end{fremark}

The organization of this paper is as follows.
In \S 1, we fix notation and recall the basic results. Some of the
material (like Corollary \ref{c:springer}) is new in the sense that it was
not included in \cite{K}. In \S 2, we present various technical lemmas
needed in the sequel. In \S 3, we formulate and prove our main
result, Theorem \ref{fmain}. Namely, after recalling some preliminary
facts in \S 3.1, we present our main algorithm $\sigma \to
\mathsf{out} ( \sigma )$ and state Theorem \ref{fmain} in \S 3.2. We
analyze the weight distribution of certain special discrete series in
\S 3.3. Then, we use the induction theorem repeatedly to prove that
the module $L _{(a _{\sigma}, X _{\mathsf{out} ( \sigma)})}$ must be a
discrete series for all $\sigma$. We also give an
alternate characterization of $\mathsf{out} ( \sigma )$ in \S
3.5. The last section \S 4 has various applications: we
characterize those discrete series $L _{(a _{\sigma}, X _{\mathsf{out}
    ( \sigma)})}$ which contain $\mathsf{sgn}$, and analyze their
deformations in \S 4.1. We prove Theorem \ref{findADt} in \S
4.2. In \S 4.3, we prove that for generic parameter $m$, the
$R(T)$-characters of irreducible $\mathbb H_{n,m}$-modules are linearly independent. We
explain a relation between the view-points of Lusztig and
Slooten-Opdam-Solleveld in \S 4.4. We deduce the $W_n$-independence of
tempered modules in \S 4.5, and characterize the tempered $\mathbb
H_{n,m}$-modules which are irreducible as $W$-modules. We finish the
paper by presenting several constraints on the $W$-structure of
tempered $\mathbb H_{n,m}$-module coming from the comparison of the
eDL correspondence with the DLL correspondence (c.f. \cite{L6}).  

\smallskip

\noindent{\bf Acknowledgments.} This work grew out of discussions with
Peter Trapa during the second author's visit at Utah (February
2008). We thank Peter Trapa for his kind invitation and support. We
are indebted to Eric Opdam, who drew our attention to the paper
\cite{HO}. We thank George Lusztig, Klaas Slooten, Maarten Solleveld,
Toshiaki Shoji for their interest and comments. D.C. was supported by
the NSF grant DMS-0554278. S.K was partly supported by the
postdoctoral fellowship at MSRI special semester program
``Combinatorial Representation Theory", and by the Grant-in-Aid for
Young Scientists (B) 20-74011. 

\section{Preliminaries}

\subsection{Affine Hecke algebra $\mathbb H$}\label{sec:1.1}

In this paper, $G$ will denote the group $Sp(2n,\mathbb C).$ Fix a
Borel subgroup $B$, and a maximal torus $T$ in $B$, and let
$W=N_G(T)/T$ be the Weyl group. We denote the character lattice of $T$
by $X^*(T).$ Fix a root system $R$ of $(G,T)$ with positive roots
$R^+$ given by $B,$ and simple roots $\Pi.$ In coordinates, the roots are
\begin{equation}
R^+=\{\epsilon_i\pm\epsilon_j\}_{i<j}\cup\{2\epsilon_i\}\subset \{\pm\epsilon_i\pm\epsilon_j\}\cup\{\pm2\epsilon_i\}=R,
\end{equation}
and $\Pi = \{ \alpha _i \} _{i=1} ^n$ with $\alpha _i = \epsilon _i -
\epsilon _{i+1}$ $(i\neq n)$, $2 \epsilon _n$ $(i=n)$. For every
$S\subset \Pi,$ we let $P_S=L_SU_S$ denote the minimal parabolic
subgroup containing $B$ and the one-parameter unipotent subgroups
corresponding to $(-S)$. {We sometimes add a subscript $n'$ in order
  to indicate that the corresponding objects are obtained by replacing
  $n$ with $n'$ (e.g., $W _{n'}$ represents the Weyl group of type $B_{n'}$).}

\smallskip

We set $\mathcal A _{\mathbb Z} := \mathbb Z [ \mathbf q _0 ^{\pm 1}, \mathbf q _1 ^{\pm 1}, \mathbf q _2 ^{\pm 1} ]$ and $\mathcal A := \mathbb C \otimes _{\mathbb Z} \mathcal A _{\mathbb Z} = \mathbb C [ \mathbf q _0 ^{\pm 1}, \mathbf q _1 ^{\pm 1}, \mathbf q _2 ^{\pm 1} ]$.

\begin{definition}[Affine Hecke algebras of type $C _n$]\label{Checke}
An affine Hecke algebra of type $C _n$ with three parameters is an associative algebra $\mathbb H _n = \mathbb H$ over $\mathcal A$ generated by $\{ T _i \} _{i = 1} ^n$ and $\{ e ^{\lambda} \} _{\lambda \in X ^* ( T ) }$ subject to the following relations:
\item {\bf (Toric relations)} For each $\lambda, \mu \in X ^* ( T )$, we have $e ^{\lambda} \cdot e ^{\mu} = e ^{\lambda + \mu}$ (and $e ^0 = 1$);
\item {\bf (The Hecke relations)} We have
\begin{equation}( T _i + 1 ) ( T _i - \mathbf q _2 ) = 0 \text{ ($1
    \le i < n$)  and  } ( T _n + 1 ) ( T _n + \mathbf q _0 \mathbf q
  _1 ) = 0;
\end{equation}
\item {\bf (The braid relations)} We have
\begin{align}\notag
&T _i T _j = T _j T _i \text{ (if $\left| i - j \right| > 1$)},\\
&( T _n T _{n - 1} ) ^2 = ( T _{n - 1} T _n ) ^2,\\\notag
&T _i T _{i + 1} T _i = T _{i + 1} T _i T _{i + 1}\text{ (if $1 \le i < n - 1$)};
\end{align}
\item {\bf (The Bernstein-Lusztig relations)} For each $\lambda \in X ^{*} ( T )$, we have
\begin{equation}
T _i e ^{\lambda} - e ^{s _i \lambda} T _i = \begin{cases} ( 1 -
  \mathbf q _2 ) \frac{e ^{\lambda} - e ^{s _i \lambda}}{e ^{\alpha
      _i} - 1} & \text{ ($i \neq n$)}\\ \frac{( 1 + \mathbf q _0
    \mathbf q _1 ) - ( \mathbf q _0 + \mathbf q _1 ) e ^{\epsilon
      _n}}{e ^{\alpha _n} - 1} ( e ^{\lambda} - e ^{s _n \lambda} ) &
  \text{ ($i = n$)} \end{cases}.
\end{equation}
\end{definition}

\begin{definition}[Parabolic subalgebras of $\mathbb H$]
Let $S \subset \Pi$ be a subset. We define $\mathbb H ^S$ to be the $\mathcal A$-subalgebra of $\mathbb H$ generated by $\{ T _i ; \alpha _i \in S \}$ and $\{ e ^{\lambda} ; \lambda \in X ^* ( T ) \}$. If $S = \Pi - \{ n \}$, then we may denote $\mathbb H ^S = \mathbb H _n ^S$ by $\mathbb H ^{\mathsf A} = \mathbb H _n ^{\mathsf{A}}$.
\end{definition}

\begin{remark}\label{remHecke}
{\bf 1)} The standard choice of parameters $( t _0, t _1, t _n )$ is: $t _1 ^2 = \mathbf q _2$, $t _n ^2 = - \mathbf q _0 \mathbf q _1$, and $t _n ( t _0 - t _0 ^{- 1} ) = ( \mathbf q _0 + \mathbf q _1 )$. This yields
$$T _n e ^{\lambda} - e ^{s _n \lambda} T _n =  \frac{1 - t _n ^2 - t _n ( t _0 - t _0 ^{- 1} ) e ^{\epsilon _n}}{e ^{2 \epsilon _n} - 1} ( e ^{\lambda} - e ^{s _n \lambda} );$$
{\bf 2)} An equal parameter extended {affine} Hecke algebra of type $B
_n$ is obtained by requiring $\mathbf q _0 + \mathbf q _1 = 0$ and
$\mathbf q _1 ^2 = \mathbf q _2$. An equal parameter
affine Hecke algebra of type $C _n$ is obtained by requiring $\mathbf q _2 = - \mathbf q _0 \mathbf q _1$ and $( 1 + \mathbf q _0 ) ( 1 + \mathbf q _1 ) = 0$;\\
{\bf 3)} The extended {affine} Hecke algebra with unequal parameters of type $B _n$ with the parameters
$$
\xymatrix@R=4pt{
& &  & *{\circ} \ar@{-}[d] & *{q} &  & *{q} & *{q} & *{q ^m}\\
& & *{\circ} \ar@{-}[r] & *{\circ} \ar@{-}[r] & *{\circ} \ar@{-}[rr]|{\cdots\cdots} &  & *{\circ} \ar@{-}[r] & *{\circ} \ar@{=}[r]|{>} & *{\circ}
}\\
$$
is obtained as
$$\mathbb H' _{n,m} := \mathbb H / ( \mathbf q _0 + q^{m/2}, \mathbf q
_1 - q ^{m/2}, \mathbf q _2 - q ).$$
 The {affine} Hecke algebra with unequal parameters of type $C _n$ with the parameters
$$
\xymatrix@R=4pt{
  & *{q^m}  & *{q} &  & *{q} & *{q} & *{q ^m}\\
 & *{\circ} \ar@{=}[r]|{>} & *{\circ} \ar@{-}[rr]|{\cdots\cdots} &  & *{\circ} \ar@{-}[r] & *{\circ} \ar@{=}[r]|{<} & *{\circ}
}\\
$$
is obtained as
$$\mathbb H _{n,m} := \mathbb H / ( \mathbf q _0 + 1, \mathbf q
_1 - q ^{m}, \mathbf q _2 - q ).$$
\end{remark}

Set $\mathcal G=G\times (\mathbb C^\times)^3.$
We recall the following well-known description of the center of $\mathbb H.$

\begin{theorem}[Bernstein-Lusztig \cite{L2}]\label{BL}
The center of the Hecke algebra $\mathbb H$ is
\begin{equation}Z ( \mathbb H ) \cong \mathcal A [ e ^{\lambda} ;
  \lambda \in X ^* ( T ) ] ^W \cong \mathcal A \otimes _{R ( \mathbb C
    ^{\times} ) ^3} R ( \mathcal G ).
\end{equation}

\end{theorem}

A semisimple element
$a\in \mathcal G$ is a pair
$a=(s,\vec q)=(s,q_0,q_1,q_2)$, for some $s\in G$ semisimple, and
$\vec q\in(\mathbb C^\times)^3.$
In view of the previous theorem, each semisimple $a \in \mathcal G$
defines a character $\mathbb C_a$ of $Z ( \mathbb H )$ by taking traces of elements of $R ( \mathcal G )$ at $a$.
We can define therefore the specialized Hecke algebras
\begin{equation}\mathbb H _a := \mathbb C _a \otimes _{Z ( \mathbb H
    )} \mathbb H \text{ and } \mathbb H _a ^S := \mathbb C _a \otimes _{Z ( \mathbb H
    )} \mathbb H ^S, \label{sH}
\end{equation}
where $S \subset \Pi$.

A $\mathbb H _a$-module $M$ is said to be a $\mathbb H$-module with central character $a$ or $s$. By Theorem \ref{BL}, having central character $s$ is equivalent to having central character $v \cdot s$ for each $v \in W$. (Here $v \cdot s$ is the action induced by the adjoint action of $N_G ( T )$ on $T$ inside $G$.)

We recall next the geometric construction of irreducible $\mathbb H_a$-modules
from \cite{K}, under the assumption in the following definition.

\begin{definition}\label{d:real}
A semisimple element $a=(s,\vec q)=(s,q_0,q_1,q)$ in $\mathcal G$ is
called positive real if $-q_0,q_1 \in \mathbb R _{>0},q \in \mathbb R _{>1}$ and $s$ has
only positive real eigenvalues on $V_1$. An element $\vec
  q=(q_0,q_1,q) \in (\mathbb C^\times)^3$ is called generic if it satisfies:
\begin{eqnarray}
q _1 ^2 \neq q ^{m'}, \hskip 2mm \text{for all }m'\in\mathbb Z.\label{generic q}
\end{eqnarray}
A positive real element $a = (s,\vec q)$ is called generic if $\vec{q}$ is generic. The set of semisimple positive real generic elements of $\mathcal G$ will be denoted
by $\mathcal G_0.$\\
Whenever $\vec q$ is generic, we fix a maximal subset $\mathbb S \subset \mathbb R _{>0}$ with the following properties:
\begin{itemize}
\item We have $q _0 ^{\pm 1} \not\in \mathbb S$ and $q _1 \in \mathbb S$;
\item $\mathbb S$ is stable by $q _2 ^{\mathbb Z}$-action;
\item We have $\{ \xi, \xi ^{-1} \} \not\subset \mathbb S$ for all $\xi \in \mathbb R _{>0}$.
\end{itemize}
We call $a=(s,\vec q) \in \mathcal G_0$ to be $\mathbb S$-positive if we have $s \in T$ and $\left<\epsilon _i,s \right> \in \mathbb S$ for each $i = 1, \ldots, n$. The set of $\mathbb S$-positive elements of $\mathcal G _0$ is denoted by {$\mathcal T _0$}.

\end{definition}

\begin{example} In the case of the Hecke algebra $\mathbb H_{n,m}$ of
type $C_n$ (as is defined in Example \ref{remHecke} {\bf 3)}), if we assume $q \in \mathbb R_{>1}$ and $m \in \mathbb R$, then the condition (\ref{generic q}) turns into
$m \not\in \frac 12\mathbb Z.$
\end{example}

Here we remark that the $G$-conjugates of $\mathcal T _0$ exhaust $\mathcal G _0$.

We have $R(T)\subset \mathbb H$. Thus, we can consider the set of $R(T)$-weights of $V$ for a finite dimensional
$\mathbb H_a$-module $V$. We denote it by $\Psi(V)\subset T.$ It is well known $\Psi(V)\subset W\cdot s$ whenever $a = ( s, \vec{q})$. 

\begin{definition}\label{d:temp} A $\mathbb H_a$-module $V$ is called tempered, if for
  all $\chi\in \Psi(V),$ one has{
\begin{equation}\label{eq:temp}
\langle \varpi_j, \chi \rangle\le 1,\text{ for all }
1\le j\le n,
\end{equation}
where $\varpi _j = \epsilon _1 + \cdots + \epsilon _j$.} The module $V$ is called a discrete series, if the inequalities in
(\ref{eq:temp}) are strict. 
\end{definition}

\subsection{Reduction to positive real character}\label{realred} 

We briefly recall Lusztig's reduction to
positive real central character. The original reference is \cite{L2}
(see also \cite{BM2} \S 6, \cite{L6} \S 4, and \cite{L7} \S 3), and there is a
complete recent exposition relative to
the tempered spectrum in \cite{OS}. In this subsection alone, we denote
the Hecke algebra defined in \S 1.1 by $\mathbb H_{R,X}^{\lambda_0,\lambda_1,\lambda_n}$, where, as before, $X=X^*(T),$ $T$ is
the torus for $Sp(2n,\mathbb C)$, $R$ denotes the roots of type $C_n$
(with coroots $\check R$ of type $B_n$), and 
$t_i=q^{\lambda_i}$ (as in Remark \ref{remHecke} {\bf 1)}), where $q$ acts by
some element in $\mathbb R_{>1}.$ Let us also denote by $\mathsf{Irrep}_{W\cdot s,
  \nu_0}\mathbb H_{R,X}^{\lambda_0,\lambda_1,\lambda_n}$, the set of isomorphism classes of
irreducible modules on which $q$ acts by $\nu_0,$ and the central
character is $s\in T.$ (The emphasis on $W$ in the notation will be
justified by the reduction procedure.)

Every $s\in T$ has a unique decomposition $s=s_e\cdot s_h,$ into an
elliptic part, and a hyperbolic part: $s_e\in S^1\otimes_{\mathbb
  Z}X_*(T),$ $s_h\in \mathbb R_{>0}\otimes_{\mathbb Z} X_*(T).$ Note
that {$(s_h,\vec{q}) \in \mathcal G _0$}.

Fix a central character $s=s_e\cdot s_h.$ Define
\begin{equation}
R_{s_e}=\left\{\alpha\in R: \begin{matrix}\check\alpha(s_e)=1, &\text{ if
    }\check\alpha\notin\{\pm\epsilon_n\}\\\check\alpha(s_e)=\pm 1, &\text{ if }\check\alpha\in\{\pm\epsilon_n\} \end{matrix}\right\}.
\end{equation}
Then $R_{s_e}\subset R$ is a root subsystem. Let $W_{s_e}\subset W$
denote the subgroup generated by the reflections in the roots of
$R_{s_e}.$ 

\begin{definition}\label{d:graded} Let $R'\subset 
  E',$ $\check R' \subset \check E'$ be root (and coroot)
  systems in the usual sense, and denote by $E'_\mathbb C$ and $
 \check E'_\mathbb C$ the complexifications. Let $\Pi'\subset R'$ be a
  set of simple roots, and $W'$ be the Coxeter group. Let $\mu$ be a
  $W'$-invariant function on $\Pi'.$ Define the graded (or degenerate) affine Hecke
  algebra $\overline{\mathbb H}_{R', E'}^\mu$ (\cite{L2}) to be the associative
  $\mathbb C[\mathbf r]$-algebra with unity generated by $\{t_w:~w\in W'\}$,
  {$\omega\in  \check E'_\mathbb C$} subject to the relations:
\begin{align*}
&t_wt_w'=t_{ww'}, \text{ for all }w,w'\in W';\\
&\omega\omega'=\omega'\omega, \text{ for all }\omega,\omega'\in
 \check E'_{\mathbb C};\\
&{\omega t_{s_\alpha}-t_{s_\alpha} s_\alpha(\omega)= \mathbf r
\mu(\alpha)\langle\alpha,\omega\rangle,\text{ for all }\alpha\in
\Pi',\omega\in  \check E'_{\mathbb C}}.
\end{align*}
\end{definition}

\begin{remark}The center of $\overline{\mathbb H}_{R', E'}^\mu$
is $\mathbb C[\mathbf r]\otimes S( \check E'_\mathbb C)^{W'},$ where $S(~)$
denotes the symmetric algebra (\cite{L2}). Therefore, the central characters of
irreducible modules, on which $\mathbf r$ acts by a certain scalar, are given by $W'$-conjugacy classes of elements in
$E'_\mathbb C$. Denote by $\mathsf{Irrep}_{W'\cdot
  \underline s,r_0}\overline{\mathbb H}_{R',
  E'}^\mu$ the class of irreducible modules with central character
$\underline s\in E'_\mathbb C$, and on which $\mathbf r$
acts by $r_0.$ 
\end{remark}

Fix $r_0$ such that $e^{r_0}=\nu_0.$ Recall that every hyperbolic
element $s_h\in T$ has a unique logarithm $\log s_h\in \mathfrak t.$ 

\begin{theorem}[\cite{L2}]\label{t:Lred} There are natural one-to-one correspondences:
$$\mathsf{Irrep}_{W\cdot s,\nu_0}\mathbb
H_{R,X}^{\lambda_0,\lambda_1,\lambda_n}\cong
\mathsf{Irrep}_{W_{s_e}\cdot s,\nu_0}\mathbb
H_{R_{s_e},X}^{\lambda_0,\lambda_1,\lambda_n}\cong
\mathsf{Irrep}_{W_{s_e}\cdot\log s_h,r_0}\overline{\mathbb H}_{R_{s_e},{\mathfrak t}^{*}}^{\mu_1,\mu_n},
$$
where $\mu_1=\lambda_1$, $\mu_n=\lambda_n+\langle\epsilon_n,s_e\rangle\lambda_0.$
\end{theorem}

These correspondences follow from an isomorphism of certain
completions of these Hecke algebras. Applying Theorem \ref{t:Lred} for
$\mathsf{Irrep}_{W_{s_e}\cdot s_h,\nu_0}{\mathbb
H_{R_{s_e},X}^{\mu_n/2,\mu_1,\mu_n/2}},$ we see
that:

\begin{corollary}\label{c:Lred}There is a natural one-to-one correspondence
$$\mathsf{Irrep}_{W\cdot s,\nu_0}\mathbb
H_{R,X}^{\lambda_0,\lambda_1,\lambda_n}\cong\mathsf{Irrep}_{W_{s_e}\cdot
  s_h,\nu_0}{\mathbb
H_{R_{s_e},X}^{\mu_n/2,\mu_1,\mu_n/2}},$$
where $\mu_1=\lambda_1$, $\mu_n=\lambda_n+\langle\epsilon_n,s_e\rangle\lambda_0.$
\end{corollary}

We should mention that, for a general affine Hecke algebra (not
necessarily of type $C_n$), a similar correspondence as in Theorem \ref{t:Lred}
holds, but the affine graded algebra $\overline {\mathbb H}$ needs to be replaced with an
extension of it by a group of diagram automorphisms.

\begin{corollary}In the correspondence of Corollary \ref{c:Lred},
  tempered modules and discrete series modules correspond, respectively.
\end{corollary}

This is easily seen from the isomorphism of algebra completions we
alluded to above (see for example \cite{BM2} \S 6, \cite{L7}
   \S 3, or \cite{OS}\S 2).

\begin{definition}\label{gradedC} We define $\overline{\mathbb H}_{n,m}$ to be the affine graded
  Hecke algebra corresponding to $R'=C_n$,
  $\mu(\epsilon_i-\epsilon_{i+1})=1$, $\mu(2\epsilon_n)=m$ (notation
  as in Definition \ref{d:graded}).
\end{definition}

Finally, recall that, again with the notation as in Definition \ref{d:graded}, there is an isomorphism between
  the affine graded Hecke algebra of type $R'=B_n$ with
  parameters $\mu(\epsilon_i-\epsilon_{i+1})=\mu_1$,
  $\mu(\epsilon_n)=2\mu_2$ and the affine graded Hecke algebra of type
  $R'=C_n$ with parameters $\mu(\epsilon_i-\epsilon_{i+1})=\mu_1$,
  $\mu(2\epsilon_n)=\mu_2$.  Therefore, when we use the affine graded
  Hecke algebra later in the paper, in particular in \S 4, we will
  consider (as we may) only the affine graded Hecke algebra of type
  $C_n$ with unequal parameters as in Definition \ref{gradedC}.

\subsection{Irreducible $\mathbb H$-modules}\label{sec:irrH}

We use the notation of \S \ref{sec:1.1}. In addition, we
introduce the following notation. For {an algebraic} group $H$, an element $h \in
H$, and {an algebraic} $H$-variety $\mathcal X$ we denote by $\mathcal X ^h$ and
$\mathcal X^H$, the
subvariety of $h$-fixed and $H$-fixed points in $\mathcal X$, respectively. For $x \in \mathcal X$,
we define $\mathsf{Stab} _H x := \{ h \in H ; h x = x \}$. We use German letters to denote Lie algebras (e.g. $\mathfrak h = \mathrm{Lie} H$).

\smallskip

Let $V_1=\mathbb C^{2n}$ denote the vector representation of $G$. Set
$V_2=\wedge^2 V_1$ and $\mathbb V=V_1^{\oplus 2}\oplus V_2.$ Then
  $\mathbb V$ is a representation of $\mathcal G$ as follows: $G$ acts
  diagonally, and an element $(c_0,c_1,c_2)\in (\mathbb C^\times)^3$ acts on
  $(V_1,V_1,V_2)$ via multiplication by $(c_0^{-1},c_1^{-1},c_2^{-1}).$

For every nonzero weight $\beta \in X ^* ( T )$ of the
  $G$-module $V _1 \oplus V _2$, we fix a non-zero $T$-eigenvector
  $v_\beta$. In coordinates, these nonzero weights are $\{\pm \epsilon_i:1\le
  i\le n\}\cup\{\pm(\epsilon_i\pm\epsilon_j): 1\le i<j\le n\}.$ The
  corresponding weight spaces are one-dimensional, so $v_\beta$ is
  unique up to scalar.

We denote by $\mathbb V^+$ the sum of $B$-positive $T$-weight spaces in
$\mathbb V.$ For each $S\subset \Pi,$ we will denote by $\mathbb V_S$
the sum of $T$-weight spaces for the weights in the $\mathbb
Q$-span of $S.$ 
We define the collapsing map (an analogue of the moment map)
\begin{equation}\mu : F := G \times ^B \mathbb V ^+ \longrightarrow
  \mathbb V, \quad \mu(g,v^+)=g\cdot v^+,\ g\in G,~ v^+\in \mathbb V^+,
\end{equation}
and denote the image of $\mu$ by $\mathfrak N$. For each {positive real} $a=(s,\vec q),$ we denote by $\mu^a$ the restriction
of $\mu$ to the $a$-fixed points of $F.$ We denote by $G(s)$ or 
$G(a)$ the centralizer $Z_G(s).$ This is a connected (reductive)
subgroup, since $G$ is simply connected. Its action on $\mathfrak N^a$
has finitely many orbits. We will describe this in more detail in
\S \ref{sec:1.3}.

Let $\mathsf{pr}_B:F\to G/B$ be the projection
$\mathsf{pr}_B(g,v^+)=gB.$ We define
\begin{equation}
\mathcal E_X^a:=\mathsf{pr}_B(\mu ^{-1} ( X ) ^a) \subset G / B,
\end{equation}
and call it an exotic Springer fiber. 

\begin{definition}[Standard module]
Let $a \in \mathcal G$ be a positive real element and let $X \in \mathfrak N ^a$. The total
Borel-Moore homology space
\begin{equation}M _{(a, X)} := \bigoplus _{m \ge 0} H _m ( \mathcal E _X ^a, \mathbb C )
\end{equation}
admits a structure of finite dimensional $\mathbb H _a$-module. We call this module a standard module. Fix $S \subset \Pi$. If $a \in L _S \times ( \mathbb C ^{\times} ) ^3$ and $X \in \mathbb V _S$, then
\begin{equation}M _{(a, X)} ^S:= \bigoplus _{m \ge 0} H _m ( \mathcal E _X ^a \cap P _S / B, \mathbb C )
\end{equation}
admits a $\mathbb H ^S$-module structure.\\
If $S = \Pi - \{ n \}$, then we may denote $M _{(a, X)} ^S$ by $M _{(a, X)} ^{\mathsf A}$.
\end{definition}

Let $\mathbb V ^S$ be the unique $T$-equivariant splitting of the map
$\mathbb V ^+ \longrightarrow \mathbb V ^+ / ( \mathbb V ^+ \cap
\mathbb V _S )$. If $X \in \mathbb V _S$, then we have necessarily
$\mathfrak u _S X \subset \mathbb V ^S$. The induction theorem is the following:

\begin{theorem}[\cite{K} Theorem 7.4]\label{indt}
Let $S \subset \Pi$ be given. Let {$a = (s,\vec{q})\in \mathcal G$ be a positive real element} and let $X \in \mathfrak N ^{a}$. Assume $s \in L_S$ and $X \in ( \mathfrak N \cap \mathbb V _S )$. If we have
\begin{eqnarray}
( \mathbb V ^S ) ^a \subset \mathfrak u _S X,\label{denom}
\end{eqnarray}
then we have an isomorphism
$$\mathsf{Ind} ^{\mathbb H} _{\mathbb H ^{S}} M _{(a,X)} ^{S} \cong M _{(a,X)}$$
as $\mathbb H$-modules. \hfill $\Box$
\end{theorem}

Fix the
semisimple element $a_0=(1,-1,1,1)\in \mathcal G.$ The following
result is an exotic version of the well-known Springer correspondence.

\begin{theorem}[\cite{K} Theorem 8.3]\label{t:1.15}
Let $X \in \mathfrak N ^{a _0}$ be given. Then, the space
\begin{equation}L _X := H _{2 d _X} ( \mathcal E _X ^{a_0},
  \mathbb C ), \text{ where }d _X := \dim \mathcal E _X ^{a_0}
\end{equation}
admits a structure of irreducible $W$-module.

Moreover, the map $X\mapsto L_X$ defines a one-to-one correspondence
between the set of orbits $G\backslash\mathfrak N^{a_0}$ and
$\mathsf{Irrep}~W.$ 
\end{theorem}
In this correspondence, if $X$ is in the open dense $G ( a _0 )$-orbit of $\mathfrak N ^{a _0}$ then $L_{X}$ is the $\mathsf{sgn}$ $W$-representation. If $X=0,$
then $L_{0}$ is the $\mathsf{triv}$ $W$-representation. 

In the following proposition,  $Z=F\times_\mathfrak N F$ denotes the
exotic Steinberg variety, and $H_\bullet^A( \bullet )$ is equivariant (Borel-Moore)
homology with respect to the group $A.$

\begin{proposition}[\cite{K} Theorem 9.2]\label{inclW}
Let $a = ( s, q _0, q _1, q)\in \mathcal G$ be a positive real
element. We set $\underline{a} := ( \log s, r _1, r ) \in \mathfrak
t \oplus \mathbb R^2$, where $r_1=\log q_1, r=\log q$. Let $A$ be a connected subtorus of $\mathcal G$ which contains
$(s,1,q_1,q)$. If $X\in
\mathfrak N^A$, then $H _{\bullet} ^A ( Z ^{a_0} )$ acquires a structure of a $\mathbb C [\mathfrak a]$-algebra which we denote by $\mathbb H_a^{+}$. We have:
\begin{enumerate}
\item The quotient of $\mathbb H_a^{+}$ by the ideal generated by functions of $\mathbb C [ \mathfrak a ]$ which are zero along $\underline{a}$ is isomorphic to $\mathbb H _a$;
\item We have a natural inclusion $\mathbb C [ W ] \hookrightarrow \mathbb H_a^{+}$.
\end{enumerate}
Moreover, we have
$$\mathbb C [ \mathfrak a ] \otimes H _{\bullet} ( \mathcal E _X ) \cong H _{\bullet} ^A ( \mathcal E _X ) \text{ for each }X \in \mathfrak N ^A$$
as a compatible $( \mathbb C [ W ], \mathbb C [ \mathfrak a ] )$-module, where $W$ acts on $\mathbb C [\mathfrak a]$ trivially. \hfill $\Box$
\end{proposition}

\begin{corollary}\label{Green}
Keep the setting of Proposition \ref{inclW}. We have
\begin{equation}\label{stdWstr}
M _{(a, X)} \cong \bigoplus _{m \ge 0} H _m ( \mathcal E _X ^{a _0}, \mathbb C ).
\end{equation}
as $\mathbb C [ W ]$-modules.\hfill $\Box$
\end{corollary}

In (\ref{stdWstr}), the irreducible $W$-module $L_X$ appears exactly
once in the decomposition of $M_{(a,X)}.$ There is a unique
irreducible quotient of $M_{(a,X)}$, denoted $L_{(a,X)}$, 
  with the property that $L_{(a,X)}$ contains $L_X.$

\begin{theorem}[\cite{K} Theorem 10.2]\label{eDL}
Let $a = ( s, \vec{q} )\in \mathcal G$ be a positive real element. We have $\mathfrak N ^a
\subset \mathfrak N ^{a _0}$. Then, we have a one-to-one correspondence
\begin{equation} G ( a ) \backslash \mathbb V ^a\leftrightarrow\mathsf{Irrep}~ \mathbb H _a ,\qquad X\mapsto L_{(a,X)}.
\end{equation}
The module $L _{(a,X)}$ is a $\mathbb H _a$-module quotient of $M_{(a,X)}$. Moreover, if $L_{(a,Y)}$ appears in $M_{(a,X)}$, then we have $X \in \overline{G(a)Y}$.
\end{theorem}

We need to emphasize that, unlike the case of \cite{KL2}, there are no nontrivial local systems appearing in the
parameterization of $\mathsf{Irrep}~ \mathbb H _a$ in Theorem \ref{eDL} for {positive real} $a\in
\mathcal G$ (see also Corollary \ref{stab}).

\begin{corollary}
If $a \in \mathcal G$ is a positive real element, then the set of $G ( a )$-orbits of $\mathfrak N
^a$ is finite. In particular, there exists a unique dense open $G ( a )$-orbit $\mathcal O _0 ^a$ of $\mathfrak N ^a$.
\end{corollary}

It is useful to remark that $L_{X_0}\cong \mathsf{sgn}$ ($X _0 \in \mathcal O_0^{a_0}$)
appears with multiplicity one in every standard module $M_{(a,X)}.$

From Theorem \ref{eDL}, together with the DLL correspondence of type $\mathsf{A}$, we deduce:
\begin{theorem}
Keep the same setting as Theorem \ref{eDL}. Let $S \subset \Pi$ and assume $s \in L _S$. Then, we have a one-to-one correspondence
\begin{equation} L_S ( s ) \backslash \mathbb V _S ^a \leftrightarrow\mathsf{Irrep}~ \mathbb H _a ^S ,\qquad X \mapsto L_{(a,X)} ^S.
\end{equation}
The module $L _{(a,X)} ^S$ is a $\mathbb H _a ^S$-module quotient of $M_{(a,X)} ^S$. Moreover, if $L_{(a,Y)} ^S$ appears in $M_{(a,X)} ^S$, then we have $X \in \overline{L_S (a)Y}$.
\end{theorem}

\begin{convention}\label{conv:1.21}
Assume that $a \in T \times ( \mathbb C ^{\times} ) ^3$. If $S = \Pi - \{ n \}$, then we set $L_{(a,X)} ^{\mathsf A} = L_{(a,X)} ^{S}$. Let ${} ^{\mathsf t} a = ( w _0 ^{\mathsf A} \cdot s ^{-1}, \vec{q} )$, where $w _0 ^{\mathsf{A}} \in \mathfrak S _n \subset W$ is the longest element. Since {$X \in \mathfrak{gl} ( n ) \subset \mathbb V$ (with the inclusion as $L _S$-modules)}, we have ${}^{\mathsf t} a X = X$. As a consequence, the $\mathbb H ^{\mathsf A}$-modules
$${} ^{\mathsf{t}} M _{(a,X)} ^{\mathsf{A}} := M _{({} ^{\mathsf{t}}a,X)} ^{\mathsf{A}} \text{ and }{} ^{\mathsf{t}} L _{(a,X)} ^{\mathsf{A}} := L _{({} ^{\mathsf{t}}a,X)} ^{\mathsf{A}}$$
are well-defined.
\end{convention}

The following results are not explicitly stated in \cite{K}, but are
immediate consequences.

\begin{theorem}\label{sgnchar}
Let $a \in \mathcal G _0$ be given. Let $L$ be an irreducible $\mathbb H _a$-module which contains $\mathsf{sgn}$ as its $W$-module constituent. Then, we have necessarily $L \cong L _{(a,X)} = M _{(a,X)}$ for $X \in \mathcal O _0^a$.
\end{theorem}

\begin{proof}
Taking into account the fact that the stabilizer of $X$ in $G$ is connected and $\mathfrak N ^a$ is a vector space (i.e. a smooth algebraic variety), the assertion follows by exactly the same argument as in \cite{EM}.
\end{proof}

\begin{corollary}\label{c:springer}
Let $G' = \mathop{Sp} (2n, \mathbb C)$ or $\mathop{SO} (2n+1, \mathbb C)$ and let $\mathfrak g'$ denote its Lie algebra. Let $Y \in \mathfrak g'$ be a nilpotent element. Let $A_Y$ be the component group of the stabilizer of $G'$-action on $Y$. Let $\mathcal B_Y$ be the Springer fiber of $Y$. Then, there exists $X \in \mathfrak N ^{a_0}$ such that
$$H _{\bullet} ( \mathcal B _Y ) ^{A_Y} \cong H _{\bullet} ( \mathcal E _X )$$
as $W$-modules $($without grading$)$.
\end{corollary}

\begin{proof}
Let $\mathbb H ^{G'}$ denote the one-parameter affine Hecke algebra
coming from $T^* (G/B)$ in the sense of the DLL correspondence. Let $G'
_{>0}$ denote the set of points of $G'$ which act on the natural
representation with only positive real eigenvalues. The Bala-Carter
theory implies that there exists a semi-simple element $a _Y \in G' _{>0}
\times \mathbb R _{>0}$ so that $G' ( a _Y ) Y$ defines a open dense
subset of ${\mathfrak g'} ^{a_Y}$. By an argument of Lusztig \cite{L3}, we
deduce that $H _{\bullet} ( \mathcal B _Y ^{s_Y} ) ^{A_Y} \cong H
_{\bullet} ( \mathcal B _Y ) ^{A_Y}$ acquires a structure of
irreducible $\mathbb H ^{G'}$-module. Here we know that $H _{0} (
\mathcal B _Y ) ^{A_Y} = H _{0} ( \mathcal B _Y ) = \mathsf{sgn}$ as
$\mathbb C [W]$-modules. It is easy to verify that we have some
positive real element $a \in
\mathcal G$ such that $\mathbb H ^{G'} _{a _Y} \cong \mathbb H _a$,
where $\mathbb H ^{G'} _{a _Y}$ is the specialized Hecke algebra of
$\mathbb H ^{G'}$ defined in a parallel fashion to (\ref{sH}). If $X \in
\mathcal O ^a _0$, the module $M _{(a,X)} = L _{(a,X)}$ is the
unique $\mathbb H_a$-module which contains $\mathsf{sgn}$ as
$W$-modules. This forces 
$$M _{(a,X)} \cong H _{\bullet} ( \mathcal B _Y ) ^{A_Y}$$
as $W$-modules (without grading). Therefore, Corollary \ref{Green} implies the result.
\end{proof}

\begin{remark}\label{r:bipartitions}
Before presenting an example of the correspondence in Corollary
\ref{c:springer}, let us recall the parameterization of
$\mathsf{Irrep}~W_n$ in terms of bipartitions. Recall that $W = W_n\cong
\mathfrak S_n\ltimes (\mathbb Z/2\mathbb Z)^n.$ Let $\xi=(\underbrace{\mathsf{triv},\dotsc,\mathsf{triv}}_k,\underbrace{\mathsf{sgn},\dotsc,\mathsf{sgn}}_{n-k})$ be a character of $(\mathbb
Z/2\mathbb Z)^n$, and let $\mathfrak S_k\times \mathfrak S_{n-k}=\mathsf{Stab}_{\mathfrak S_n}(\xi)$. The representations of  symmetric
groups are parameterized by partitions. Let $\alpha$ be a partition
of $k$ and $\beta$ be a partition of $n-k$, and $(\alpha)$, $(\beta)$
the corresponding representations of $\mathfrak S_k$ and $\mathfrak S_{n-k}$,
respectively. We denote by $\{(\alpha)(\beta)\}$ (and call it a
bipartition) the irreducible representation of $W_n$ obtained by
induction from $(\alpha)\boxtimes(\beta)\boxtimes \xi$ on $\mathfrak S_k\times
\mathfrak S_{n-k}\times (\mathbb Z/ 2 \mathbb Z)^n.$ All elements of
$\mathsf{Irrep}~W_n$ are obtained by this procedure, hence the
one-to-one correspondence with bipartitions of $n$.
\end{remark}

\begin{example} We explain \ref{c:springer} in the case $n=3.$
  There are $7$ nilpotent adjoint orbits for $SO(7,\mathbb C),$ and
  $8$ for $Sp(6,\mathbb C).$ There are $10$ exotic orbits (the same as
  the number of irreducible representations for $W_3$). Let us denote
  the representatives of these $10$ orbits as follows:

\begin{center}
\begin{tabular}{|c|c|c|}
\hline
$X_1$
&$v_{\epsilon_1-\epsilon_2}+v_{\epsilon_2-\epsilon_3}+v_{\epsilon_3}$
&$\{(0)(1^3)\}$\\\hline
$X_2$
&$v_{\epsilon_1-\epsilon_2}+v_{\epsilon_2-\epsilon_3}+v_{\epsilon_2}$
&$\{(1)(1^2)\}+\{(0)(1^3)\}$\\\hline
$X_3$
&$v_{\epsilon_1-\epsilon_2}+v_{\epsilon_2-\epsilon_3}+v_{\epsilon_1}$
&$\{(1^2)(1)\}+\{(1)(1^2)\}+\{(0)(1^3)\}$\\\hline
$X_4$ &$v_{\epsilon_1-\epsilon_2}+v_{\epsilon_2}$
&$\mathsf{Ind}_{W_2}^{W_3}(\{(0)(1^2)\})$\\\hline
$X_5$ &$v_{\epsilon_1-\epsilon_2}+v_{\epsilon_1}$
&$\mathsf{Ind}_{W_2}^{W_3}(\{(1)(1)\}+\{(0)(1^2)\})$\\\hline
$X_6$ &$v_{\epsilon_1-\epsilon_2}+v_{\epsilon_2-\epsilon_3}$ &$\mathsf{Ind}_{W(A_2)}^{W_3}(\mathsf{sgn})$\\\hline
$X_7$ &$v_{\epsilon_1-\epsilon_2}+v_{\epsilon_3}$
&$\mathsf{Ind}_{W_1\times W(A_1)}^{W_3}(\{(0)(1)\}\boxtimes \mathsf{sgn})$\\\hline
$X_8$  &$v_{\epsilon_1}$ &$\mathsf{Ind}_{W_1}^{W_3}(\{(0)(1)\})$\\\hline
$X_9$ &$v_{\epsilon_1-\epsilon_2}$ &$\mathsf{Ind}_{W(A_1)}^{W_3}(\mathsf{sgn})$\\\hline
$X_{10}$ &$0$ &$\mathsf{Ind}_{\{1\}}^{W_3}(\mathsf{triv})$\\
\hline
\end{tabular}
\end{center}

\noindent The last column gives the $W_3$-structure of $H_\bullet(\mathcal E_X)$
in every case.
With this notation, the correspondences from Corollary \ref{c:springer} are as
follows:
{
\begin{center}{$SO(7)$}
\begin{tabular}{|c|c|c|c|c|c|c|}
\hline
$(7)$ &$(51^2)$ &$(3^21)$ &$(32^2)$ &$(31^4)$ &$(2^21^3)$ &$(1^7)$\\
\hline
$X_1$ &$X_2$ &$X_3$ &$X_7$ &$X_5$ &$X_9$ &$X_{10}$\\
\hline
\end{tabular}
\end{center}

\begin{center}{$Sp(6)$}
\begin{tabular}{|c|c|c|c|c|c|c|c|}
\hline
$(6)$ &$(42)$ &$(41^2)$ &$(3^2)$ &$(2^3)$ &$(2^21^2)$ &$(21^4)$ &$(1^6)$\\
\hline
$X_1$ &$X_2$ &$X_4$ &$X_6$ &$X_7$ &$X_5$ &$X_8$ &$X_{10}$\\
\hline
\end{tabular}
\end{center}}

\noindent The notation for the parameterization for classical nilpotent orbits  is as \cite{Ca}.

\end{example}

\subsection{A parameterization of exotic orbits $G(a)\backslash \mathfrak N^a$}\label{sec:1.3}

We recall the combinatorial parameterization of $G(a)$-orbits in
$\mathfrak N^a$ following \cite{K}. Fix $q>1$, $m \in \mathbb R$ such that $\vec q=(-1,q_1,q)$ is generic, and $a=(s,\vec q)\in \mathcal T_0.$

\begin{definition}[Marked partition] A segment adapted to $a$ is a subset $I \subset [1, n]$ such that for every $i \in I$, we have either
  \begin{itemize}
  \item there exists no $j \in I$ such that $\langle \epsilon_i, s \rangle > \langle \epsilon_{j}, s \rangle$;
  \item there exists a unique $j \in I$ such that $\langle \epsilon_i, s \rangle=q\langle \epsilon_{j}, s\rangle$.
  \end{itemize}
 A marked partition adapted to $a$
  is a pair $(\{I_l\}_{l=1}^k,\delta)$ where
\begin{enumerate}
\item $I_1\sqcup
  I_2\sqcup\dots\sqcup I_k=[1,n]$ is a division of the set of integers $[1,n]$ into a union of segments {(adapted to $a$)};
  \item $\delta:[1,n]\to\{0,1\}$ such that $\delta(i)=0$ whenever
  $\langle \epsilon_i,s\rangle\neq q_1.$
\end{enumerate}
We refer $\{I_l\}_{l=1}^k$ as the support of $(\{I_l\}_{l=1}^k,\delta)$.
\end{definition}

Let us denote by $\mathsf{MP}(a)$ the set of marked partitions adapted
to $a.$

\begin{proposition}[\cite{K}]\label{p:MP1} The map $\Upsilon:~\mathsf{MP}(a)\to
  G(a)\backslash \mathfrak N^a$, given by
\begin{equation}
(\{I_l\}_{l=1}^k,\delta)\mapsto \sum_{l=1}^k
v_{I_l}+\sum_{i=1}^n\delta(i)v_{\epsilon_i},\quad\text{where }
v_{I_l}=\sum_{i,j\in I_l, \left< \epsilon _i, s \right> = q \left< \epsilon _j, s \right>} v_{\epsilon_i-\epsilon_{j}},
\end{equation}
is a surjection.
\end{proposition}

For each $\tau \in \mathsf{MP} (a)$, we put $v _{\tau} := \Upsilon ( \tau )$ and $\mathcal O _{\tau} := G ( a ) v _{\tau}$.

In order to describe the fibers of the map $\Upsilon$, we define first a partial
order on the set of subsets of $[1,n].$ Assume
$I,I'\subset[1,n]$. Set $\underline I=\{\langle \epsilon_i, s \rangle:
i\in I\}$, and similarly for $\underline I'.$ Then we define
\begin{equation} 
I\lhd I' \Longleftrightarrow \min \underline I\le\min\underline
I'\le\max\underline I\le {\max \underline I'}.
\end{equation}
If $I\lhd I',$ we say that $I$ is dominated by $I'$. We also introduce another partial ordering on $\mathbf I$, weaker than $\lhd$:
\begin{equation}
I\prec I'\iff \min\underline I\le\min\underline I'.
\end{equation}
If $(\{I_l\}_{l=1}^k,\delta)\in \mathsf{MP}(a)$ is given, we define
$(\{I_l\}_{l=1}^k,\widetilde\delta)$ by modifying $\delta$ as follows. If
some $i$ such that $\langle \epsilon_i, s \rangle=q_1$ belongs to an
$I_l$ which is dominated by a marked $I_{l'}$ (i.e.,
$\delta(I_{l'})=\{0,1\}$), then we set $\widetilde\delta(i)=1$ (i.e., we
mark $I_l$ as well). 

A permutation $w \in \mathfrak S _n$ is said to be adapted to $a$ if we have $w \cdot s = s$. Let $\mathfrak S _n ^a$ denote the subgroup of $\mathfrak S _n$ formed by permutations adapted to $a$.

It is straight-forward that $\mathfrak S _n ^a$ acts on $\mathsf{MP}
(a)$ by applying $w$ to $(\{I_l\}_{l=1}^k,\delta) \in \mathsf{MP} (a)$
as $(\{I_l\}_{l=1}^k,\delta) \mapsto (\{w ( I_l )\}_{l=1}^k,w
^*\delta)$, where $(w^*\delta)(i)=\delta(w^{-1}(i)).$

\begin{proposition}[\cite{K}]\label{p:MP2} Let $\Upsilon$ be the map
  defined in Proposition \ref{p:MP1}. Then
$\Upsilon((\{I_l\}_{l=1}^k,\delta))=\Upsilon((\{I_l'\}_{l=1}^{k'},\delta'))$ if
and only if $\{I_l\}_{l=1}^k=\{I_l'\}_{l=1}^{k'}$
and $\widetilde\delta=\widetilde\delta'$ up to $\mathfrak S _n ^a$-action.
\end{proposition}

The marked partition corresponding to the open $G(a)$-orbit in
$\mathfrak N^a$ is obtained by extracting the longest possible $I_1$
subject to $a v_{I_1} = v _{I_1}$, then the longest possible $I_2$ from
$[1,n]\setminus I_1$ subject to the same condition etc. Then we mark
all $I_j$ such that $q_1\in\underline{I_j}.$ 

Let $\mathsf{MP} _0 ( a )$ be the set of marked partitions with trivial markings. (I.e. $\tau = ( \mathbf I, \delta ) \in \mathsf{MP} (a)$ with $\delta \equiv 0$.)
The following result is a re-interpretation of the closure relations of type $\mathsf{A}$ quiver orbits with uniform orientation.

\begin{theorem}[\cite{AD},\cite{Z2}]\label{AD}
Let $\tau = ( \mathbf I, 0 ) \in \mathsf{MP} _0 ( a )$ be given. Let $\tau' = ( \mathbf I', 0 ) \in \mathsf{MP} _0 ( a )$ be obtained from $\tau$ by the following procedure:
\begin{itemize}
\item[$(\spadesuit)$] For some two segments $I _k, I _l \in \mathbf I$ such that
$$\min \underline{I _{k}} < \min \underline{I _l} \le q \max \underline{I _{k}} < q \max \underline{I _l},$$
we define $\mathbf I'$ to be the set of segments obtained from
$\mathbf I$ by replacing $\{ I _k, I _l \}$ with $\{ I ^+, I ^- \}$,
where $I^{\pm}$ are segments such that $I ^+ \sqcup I ^- = I _k \sqcup
I _l$, $\underline{I ^+} = \underline{I _k} \cup \underline{I _l}$,
and $\underline{I ^-} = \underline{I _k} \cap \underline{I _l}$. $(I
  ^-$ may be an emptyset.$)$
\end{itemize}
Then, we have $\mathcal O _{\tau} \subset \overline{\mathcal O _{\tau'}}$. Moreover, every $G ( a )$-orbit which is larger than $\mathcal O _{\tau}$ and parameterized by $\mathsf{MP} _0 ( a )$ is obtained by a successive application of the procedure $(\spadesuit)$.
\end{theorem}

\begin{convention}
For each $\tau \in \mathsf{MP} (a)$, we sometimes denote $L _{(a, v _{\tau})}$ by $L _{(a, \tau)}$ or just $L _{\tau}$ when the central character is clear. We may use similar notation like $M _{(a, \tau)}$ or $M _{\tau}$.
\end{convention}

Each of $\tau = ( \{ I _l \}, 0 ) \in \mathsf{MP} _0 ( a )$
defines a representation $R _{\tau}$ of type $\mathsf{A}$-quiver
corresponding to the multisegment $\{ \underline{I _l} \} _l$
in the sense of Zelevinsky. (Here we identify $G ( a )
\circlearrowright V _2 ^{(s,q)}$ with the representation space of
type $\mathsf{A}$-quiver of an appropriate dimension vector.) In
particular, $R _{\tau}$ is a direct sum of indecomposable modules $R
_{I _l}$ corresponding to a segment $I _l$ (or rather $\underline{I _l}$). 

\begin{lemma}\label{aut}
Let $\tau = ( \{ I _l \}, \delta ) \in \mathsf{MP} _0 ( a )$ be given. We have a non-zero map $R _{I _l} \to R _{I _{l'}}$ (as modules of type $\mathsf{A}$-quivers) if and only if $I _l \triangleleft I _{l'}$. Moreover, such a non-zero map is unique up to scalar.
\end{lemma}
\begin{proof}
Straight-forward.
\end{proof}

\begin{theorem}[c.f. Brion \cite{B} Proposition 2.29]\label{idstab}
Let $a \in \mathcal T _0$ and $\tau = ( \mathbf I, 0 ) \in
\mathsf{MP} _0 ( a )$ be given. The group of automorphisms of $R _{\tau}$ as type $\mathsf{A}$-quiver representation is isomorphic to $\mathsf{Stab} _{G ( a )} v _{\tau}$. \hfill $\Box$
\end{theorem}

\begin{corollary}\label{stab}
Keep the setting of Theorem \ref{idstab}. Let $r _{\tau}$ be the
number of segments of $\mathbf I$, and let $u _{\tau}$ be the set of
distinct pairs of segments $I, I'$ in $\mathbf I$ such that
$I\triangleleft I$. Then, $\mathsf{Stab} _{G ( a )} v _{\tau}$ is a
connected algebraic group of rank $r_{\tau}$ and dimension $( r_{\tau}
+ u_{\tau} )$. 
\end{corollary}

\begin{proof}
Taking into account Theorem \ref{idstab} {and the fact that an automorphism group of a module over an algebra is connected}, the assertion is a straight-forward corollary of Lemma \ref{aut}.
\end{proof}

\section{Some weight calculations}\label{sec:2}

\subsection{Varieties corresponding to weight spaces}\label{sec:2.1}

In this section, we use the language of perverse sheaves
(corresponding to middle perversity) on complex algebraic
varieties. Some of the standard references for this theory are
Beilinson-Bernstein-Deligne \cite{BBD}, Kashiwara-Schapiro \cite{KS},
Gelfand-Manin \cite{GM}, and Hotta-Tanisaki \cite{HT}. 

For a variety $\mathcal X$, we denote by $\underline{\mathbb C}$ the
constant sheaf (shifted by $\dim \mathcal X$). For a locally closed
subset $\mathcal O \subset \mathcal X$, we have an embedding $j
_{\mathcal O} : \mathcal O \to \mathcal X$. We have a {locally} constant sheaf
$( j _{\mathcal O} ) _{!} \underline{\mathbb C}$ obtained by extending
the constant sheaf on $\mathcal O$ by zero to $\mathcal X$. We have an
intermediate extension object $\mathsf{IC} ( \mathcal O )$ {(simple object in the category of perverse sheaves)} obtained by appropriately truncating $( j _{\mathcal O} ) _{!} \underline{\mathbb C}$.

Fix $a \in T \times ( \mathbb C ^{\times} ) ^3$. For every $w\in W$, let $\dot{w}$ denote a representative in $N_G(T).$
We put $^w\!\mathbb V^+:=\dot{w} ^{-1} \mathbb V^+$ and ${}^w \!
\mathbb V (a) := \mathbb V^a\cap ~^{w}\!\mathbb V^+$. We denote $(
\mathrm{Ad} ( \dot{w} ^{-1} ) B ) (s)$ by ${}^w B (s)$. It is clear
that these definitions do not depend on the choice of
$\dot{w}$. Recall the restriction of the collapsing map $\mu^a: F^a\to
\mathfrak N^a.$ Set
\begin{equation}\label{eq:Fw}
F_w^a=G(s)\times^{^{w}\! B(s)} {}^{w}\!\mathbb V (a).
\end{equation}
Let $W_s$ be the reflection subgroup of $W$ corresponding to the subroot system of $R$ defined by the roots $\alpha$ such that $\alpha ( s )=1.$ Following Lemma 3.6 in \cite{K}, we have a decomposition
\begin{equation}
F^a=\sqcup_{w\in W/W_s} F_w^a.
\end{equation} 
Denote by $\mu_w^a$ the restriction of $\mu^a$ to a piece $F_w^a,$ where
$w$ is a representative in $W/W_s.$

For each $u =w \cdot s ^{-1}\in W \cdot s^{-1}$, let $\mathcal E_{X}^a [u]$ denote the preimage of $X \in \mathfrak N ^a$ under $\mu_w^a$, projected to $G/B$:
\begin{equation}\label{eq:2.3}
\mathcal E_{X}^a [u] = \{ g\dot{w}^{-1}B ;\ gs=sg, X\in g {}^{w}\!\mathbb V (a) \}.
\end{equation}
 Notice that replacing $w$ by $ww'$ ($w'\in W_s$) in this construction
 gives the same variety, hence $\mathcal E_X^a[u]$ only depends on
 $w\in W/W_s.$

\begin{proposition}\label{p:2.1} Let $\tau \in \mathsf{MP} (a)$ be given. For $w\in W/W_s$, the $u = ( w\cdot s^{-1} )$-weight space of the standard module $M_{(a,v_\tau)}$ is $H_\bullet( \mathcal E _{v_\tau} ^a [u] ).$ 
Moreover, $u$ is a $R ( T )$-weight of $L_{(a,v_\tau)}$ if and
  only if $(\mu_w^a)_*\underline{\mathbb C}$ contains
  $\mathsf{IC}(\mathcal O_\tau).$
\end{proposition}

\begin{proof}
See Chriss-Ginzburg \cite{CG} \S 8.6.
\end{proof}

\noindent An important consequence is that we can characterize certain $R(T)$-weights
of $L_{(a,v_\tau)}.$

\begin{corollary}\label{c:ftemp} If $\mathcal O_\tau$ meets ${}^{w}\!\mathbb V (a)$ in a dense open subset, then $w\cdot s ^{-1}$ is
  a $R(T)$-weight of $L_{(a,v_\tau)}.$
\end{corollary}

\begin{proof}
We have $(\mu_w^a)^{-1}(X) \neq \emptyset$ only if $X \in
\overline{\mathcal O _{\tau}}$. We have $\dim H_{\bullet} (
(\mu_w^a)^{-1}(X) ) \neq 0$ when $X \in \mathcal O _{\tau}$. It
follows that there exist a simple constituent of $(\mu_w^a)_*
\underline{\mathbb C}$ in $D^b ( \mathfrak N ^a )$ which has support
contained in $\overline{\mathcal O _{\tau}}$. By the
BBD-Gabber theorem and \cite{K} Theorem 4.10, we have
$\mathsf{IC} ( \mathcal O _{\tau} )$ as a direct summand of
$(\mu_w^a)_* \underline{\mathbb C}$ (up to degree shift). Now Proposition \ref{p:2.1} implies the result. 
\end{proof}

\subsection{Conditions on weights of $\mathbb H_{n,m}$-modules}

Later in this section, we work in the same setting as in \S \ref{sec:1.3}.

\begin{proposition}\label{re}
Let $\tau = ( \mathbf I, \delta ) \in \mathsf{MP} ( a )$. Let $w \in W$. If ${} ^{w}\mathbb V (a)$ meets $\mathcal O _{\tau}$, then there exists $v \in \mathfrak S _n ^a$ which satisfies:
\begin{itemize}
\item[$(\clubsuit) _w$] For every $i, j \in I_k \in \mathbf I$ such that $\left< \epsilon _i, s \right> \in q ^{\mathbb Z _{>0}} \left< \epsilon _j, s \right>$, we have either
$$w v ( i ) < w v ( j ) \text{ or } w v ( i ) > 0 > w v ( j ).$$
\end{itemize}
\end{proposition}

\begin{proof}
The space ${} ^{w} \mathbb V ( a )$ is stable under the action of ${} ^{w} B (s)$. It follows that the space $G ( s ) {} ^{w} \mathbb V ( a )$ is a closed subset of $\mathfrak N ^a$. In particular, a $G (s)$-orbit $\mathcal O _{\tau} \subset \mathfrak N ^{a}$ meets ${} ^{w} \mathbb V ( a )$ if and only if we have $\mathcal O _{\tau} \subset \overline{\mathcal O}$, where $\mathcal O$ is the open dense $G(s)$-orbit of $G ( s ) {} ^{w} \mathbb V ( a )$.

Condition $(\clubsuit) _w$ is independent of the marking. Consider $\tau'
:= ( \mathbf I, 0 ) \in \mathsf{MP} _0 ( a )$. We have
$\mathcal O _{\tau'} \subset \overline{\mathcal O _\tau}$. Hence, it
suffices to verify $(\clubsuit) _w$ in the case $\tau = \tau'$. We put
$T _{\tau} :=\mathsf{Stab} _T v _{\tau}$. It is easy to verify that $T _{\tau}$ is a connected torus of rank $\# \mathbf I$ (the number of segments of the support of $\tau$). Since we have $\left< \epsilon _i,s \right> \neq \left< \epsilon _j,s \right>$ for each distinct $i,j \in I _k \in \mathbf I$, it follows that $Z _{G(s)} ( T_{\tau} ) = T$ and $T v _{\tau} = \mathcal O ^{T_{\tau}}_{\tau}$. The set of ${} ^v T _{\tau}$-fixed points of
$$F ^a _w := G ( s ) \times ^{{} ^{w}B ( s )} {} ^{w} \mathbb V (a) \stackrel{\pi _w ^a}{\longrightarrow} G ( s ) / B$$
is concentrated on the fiber of $( G (s) / B ) ^T$ for every $w\in W$
and $v \in W _s$. The image of $\mu _w ^a$ contains all of $\mathcal
O _{\tau}$ if and only if $\mathcal O _{\tau} \subset \mathfrak N
^{a}$ meets ${} ^{w} \mathbb V ( a )$. Therefore, ${} ^{w}\mathbb V
(a)$ meets $\mathcal O _{\tau}$ if and only if
$$\dim \left( {} ^{w}\mathbb V (a) \right) ^{{} ^{v'} T _{\tau}} = \dim T v _{\tau}  \text{ for some  } v' \in W
_s.$$
Since $a \in \mathcal T _0$, we have $\mathbb V ^a \subset \mathfrak{gl} ( n ) \oplus \mathbb C ^n$. Moreover, we have $W_s \cong \mathfrak S _n ^a$. Now $(\clubsuit) _w$ is equivalent to the fact that $\dot{v} v _{I _k} \in {} ^{w}\mathbb V ( a )$ for each $I _k \in \mathbf I$.
\end{proof}

\begin{corollary}\label{wcrit}
Keep the setting of Proposition \ref{re}. We have $w \cdot s ^{-1} \in \Psi ( M _{(a, v _{\tau})})$ only if there exists some $v \in \mathfrak S_n ^a$ satisfying $(\clubsuit) _w$. In particular, $w \cdot s ^{-1} \in \Psi ( L _{(a, v _{\tau})})$ only if there exists some $v \in \mathfrak S_n ^a$ satisfying $(\clubsuit) _w$.
\end{corollary}

\begin{proof}
By definition, we have $\mathcal E^a _{v _{\tau}} [ w \cdot s ^{-1}] \neq \emptyset$ if and only if $G ( a ) v_{\tau} \cap {}^w \mathbb V ^a \neq \emptyset$. It follows that $\mathcal E^a _{v _{\tau}} [ w \cdot s ^{-1} ] \neq \emptyset$ only if there exists some $v \in \mathfrak S_n ^a$ satisfying $(\clubsuit) _w$. Therefore, Proposition \ref{p:2.1} implies the result.
\end{proof}

\begin{theorem}[see e.g. Barbasch-Moy \cite{BM2} \S 6.4]\label{wt-ind}
Fix $n = n _1 + n _2$ with $n_1, n_2 \in \mathbb Z _{\ge 0}$. Let $L_1$ be a $\mathbb H_{n_1} ^{\mathsf A}$-module and let $L _2$ be a $\mathbb H_{n_2}$-module. We form
$$M := \mathrm{Ind} ^{\mathbb H _n} _{\mathbb H_{n_1} ^{\mathsf A} \otimes \mathbb H _{n_2}} L _1 \boxtimes L _2.$$
If an irreducible $\mathbb H _n$-module $L$ is a subquotient of $M$, then we have
\begin{equation}
\Psi ( L ) \subset \Psi ( M ) = \bigcup _{w} w \cdot ( \Psi ( L_1 ^{\mathsf A} ) \times \Psi ( L_2 )),\label{psi1}
\end{equation}
where $w \in W$ runs over the minimal length representatives of $W / ( \mathfrak S _{n_1} \times W_{n_2} )$. \hfill $\Box$
\end{theorem}

\begin{proposition}\label{pw-ind}
Keep the setting of Theorem \ref{wt-ind}. Assume that $L_1 = L _{(a_1, v_{\tau_1})} ^{\mathsf A}$ holds for $\tau_1 = ( \mathbf I_1, 0 ) \in \mathsf{MP} _0 ( a_1 )$, where $a_1 = ( s_1, \vec{q} )$. Then, we have
\begin{equation}
\Psi ( L ) \subset \bigcup _{w} w \cdot ( \Psi ( M _{(a_1, v_{\tau_1})} ) \times \Psi ( L _2 )),\label{psi2}
\end{equation}
where $w \in \mathfrak S _n$ runs over the minimal length representatives of $\mathfrak S_{n} / ( \mathfrak S _{n_1} \times \mathfrak S _{n_2} )$.
\end{proposition}

\begin{proof}
Thanks to Theorem \ref{wt-ind}, we have
\begin{align*}
& \Psi ( M _{(a_1, v_{\tau_1})} ) = \bigcup _{v} v \cdot \Psi ( M _{(a_1, v_{\tau_1})} ^{\mathsf A} )\\
& = \Psi ( \mathrm{Ind} _{\mathbb H _{n_1} ^{\mathsf A}} ^{\mathbb H _{n_1}} M _{(a_1, v_{\tau_1})} ^{\mathsf A} ) \supset \Psi ( \mathrm{Ind} _{\mathbb H _{n_1} ^{\mathsf A}} ^{\mathbb H _{n_1}} L _{(a_1, v_{\tau_1})} ^{\mathsf A} ),
\end{align*}
where $v$ runs over the minimal coset representatives of $W _{n_1} / \mathfrak S _{n_1}$. A minimal length representative $w \in W / ( \mathfrak S _{n_1} \times W _{n_2} )$ decomposes uniquely into $w = w _2 w _1$, where $w_1$ is a minimal length representative of $W _{n_1} / \mathfrak S _{n_1}$ and $w _2$ is a minimal length representative in $\mathfrak S_{n} / ( \mathfrak S _{n_1} \times \mathfrak S _{n_2} )$. Therefore, the comparison between (\ref{psi1}) and (\ref{psi2}) implies the result.
\end{proof}

{
\begin{remark}
Notice that in Proposition \ref{pw-ind}, the module $M _{(a_1,
  v_{\tau_1})}$ is a standard $\mathbb H _{n_1}$-module and not a
standard $\mathbb H _{n_1} ^{\mathsf A}$-module. This is an important
point when using Corollary \ref{wcrit}.
\end{remark}}

\subsection{Nested component decomposition}\label{sec:nested}

Let $\tau=(\mathbf I,\delta)$, $\mathbf
  I=\{I_l\}_{l=1}^k$, be a marked partition adapted to $a =
  (s,\vec{q}) \in \mathcal T _0$. Assume that $\mathbf I$
  can be split into a disjoint union of two collections of segments
  $\mathbf I_1$ and $\mathbf I_2$ with the property
\begin{align}\label{ncond}
&\underline I \sqsubset \underline I',\quad\text{for every }I\in\mathbf
I_1, I'\in\mathbf I_2, \text{ or } \\\notag
&\underline I' \sqsubset \underline I, \quad\text{for every }I\in\mathbf
I_1, I'\in\mathbf I_2,
\end{align}
where $\underline{I} \sqsubset \underline{J}$ means
$$\min \underline{J} < \min \underline{I} \text{ and } \max
\underline{I} < \max \underline{J}.$$

Let $n_1$ and $n_2$ be the sums of cardinalities of segments of $\mathbf
I_1$ and $\mathbf I_2$, respectively. 
By applying an appropriate permutation, we assume that
\begin{itemize}
\item $\mathbf I_1$ and $\mathbf I_2$ are divisions of $[1,n_1]$ and $(n_1,n]$, respectively.
\end{itemize}

Then we can regard
$\tau_1=(\mathbf I_1,\delta|_{\mathbf I_1})$ and $\tau_2=(\mathbf
I_2,\delta|_{\mathbf I_2})$ as marked partitions for $G_1=Sp(2n_1)$ and
$G_2=Sp(2n_2)$ respectively, where $Sp(2n_1)\times Sp(2n_2)$ is embedded
diagonally in $Sp(2n).$

The marked partition $\tau$ parameterizes a $G(s)$-orbit $\mathcal
O_\tau$ on $\mathfrak N^a.$ We define semisimple elements $s_1\in
Sp(2n_1),$ $s_2\in Sp(2n_2)$ to be the projections of $s$ onto the
$Sp(2n_1),$ $Sp(2n_2)$ factors, respectively. We set $a _1 := ( s _1, \vec{q})$, and $a _2 := ( s _2, \vec{q})$. The marked partitions $\tau_1$ and $\tau_2$ define orbits $\mathcal O_{\tau_1}$, $\mathcal O_{\tau_2}$,
of $G_1=G_1(s_1)$, $G_2=G_2(s_2)$ respectively on the corresponding exotic
nilcones.

\begin{lemma}\label{cr-rigid}
Every $G ( s )$-orbit $\mathcal O$ of $\mathfrak N ^a$ which contains
$\mathcal O _{\tau}$ in its closure can be written as $G ( s ) ( \mathcal
O _1 \times \mathcal O _2)$, where $\mathcal O _1$ and $\mathcal O _2$
are exotic nilpotent orbits of $G_1$ and $G_2$, respectively. In
addition, the marked partitions corresponding to $\mathcal O _1$ and
$\mathcal O _2$ are nested in the sense of $(\ref{ncond})$. 
\end{lemma}

\begin{proof}
The condition (\ref{ncond}) is independent of markings. Thus,
it suffices to prove the assertion for all orbits with no markings. In
the algorithm of Theorem \ref{AD}, it is straight-forward to see that we cannot choose $I _l \in \mathbf I _1$ and $I _l \in \mathbf I _2$. Therefore, the above procedure preserves $\bigsqcup _{I _l \in \mathbf I _1} I _l$ and $\bigsqcup _{I _l \in \mathbf I _2} I _l$, respectively. Moreover, it preserves the nestedness of the modified $\mathbf I _1$ and $\mathbf I _2$, which implies the result.
\end{proof}

Let $\mathbb V _{(1)}$ and $\mathbb V _{(2)}$ be the exotic
representations of $G _1$ and $G _2$, respectively. We set
$$\mathcal O ^{\uparrow} _{\tau} := \displaystyle{\bigcup _{
    \mathcal O _{\tau} \subset \overline{\mathcal O}}} \mathcal O,$$
where $\mathcal O$ runs over all $G ( s )$-orbits of $\mathfrak N
^a$. This is a $G ( s )$-stable open subset of $\mathfrak N
^a$. Similarly, for $i=1,2$, we define $\mathcal O ^{\uparrow} _{\tau _i}$ to be
the union of $G _i ( s _i )$-orbits of $\mathfrak N ^a \cap \mathbb V
_{(i)}$ which contain $\mathcal O _{\tau _i}$ in their closure.

\begin{corollary}
Keep the setting of Lemma \ref{cr-rigid}. We have
$$G ( s ) ( \mathcal O ^{\uparrow} _{\tau _1} \times \mathcal O ^{\uparrow} _{\tau _2} ) = \mathcal O ^{\uparrow} _{\tau}.$$
\end{corollary}

\begin{lemma}\label{stab-rigid}
For each $x = x _1 \times x _2 \in \mathcal O ^{\uparrow} _{\tau _1} \times \mathcal O ^{\uparrow} _{\tau _2} \subset \mathcal O ^{\uparrow} _{\tau}$, we have
$$\mathsf{Stab} _{G ( s )} x = \mathsf{Stab} _{G _1 ( s _1 )} x _1 \times \mathsf{Stab} _{G _2 ( s _2 )} x _2.$$
\end{lemma}

\begin{proof}
Without  loss of generality, we can assume that $x = v _{\tau}$ with $\tau = ( \mathbf I, \delta )$. We have $\mathsf{Stab} _{G ( s )} x \subset \mathsf{Stab} _{G ( s )} x'$ if $x' = v_{\tau'}$ with $\tau' = ( \mathbf I, 0 )$. Therefore, it suffices to prove the result when $\delta \equiv 0$. In that case, the assertion follows by Lemma \ref{aut} and Theorem \ref{idstab}.
\end{proof}

\begin{corollary}\label{sr-str}
Keep the setting of Lemma \ref{stab-rigid}. Then, we have
$$G ( s ) \times ^{(G _1 ( s _1 ) \times G _2 ( s _2 ) )} \left( \mathcal O ^{\uparrow} _{\tau _1} \times \mathcal O ^{\uparrow} _{\tau _2} \right) \stackrel{\cong}{\longrightarrow} \mathcal O ^{\uparrow} _{\tau}.$$
\end{corollary}

Let $M$ and $L$ be two $\mathbb H_a$-modules with $L$ simple. Let $[M:L]$ denote the
multiplicity of $L$ in $M$ in the Grothendieck group of $\mathbb H\mathsf{-mod}$.

{We have a geometric multiplicity formula of irreducible modules in a standard
  module (Kazhdan-Lusztig conjectures). It is obtained in \cite{K} as an application of the Ginzburg theory
  (\cite{CG}).

\begin{theorem}[\cite{K} Theorem 11.2]\label{mult} Let $a \in \mathcal G$ be a positive real element. Let $X, X' \in \mathfrak N ^a$ be given. We set $\mathcal O := G ( a ) X$ and $\mathcal O' := G ( a ) X'$. Then we have
\begin{equation}
[M_{(a,X)}:L_{(a,X')}]=\dim H^\bullet_{\mathcal
  O }(\mathsf{IC}( \mathcal O ' )).
\end{equation}
In particular, if $\mathcal O \subset\overline{\mathcal O'}$ is smooth, then $[M_{(a,X)}:L_{(a,X')}]=1.$ \hfill $\Box$
\end{theorem}}

\begin{corollary}\label{sr-mult}
Keep the setting of Lemma \ref{stab-rigid}. Then, for every $y = y _1 \times y _2 \in \mathcal O ^{\uparrow} _{\tau _1} \times \mathcal O ^{\uparrow} _{\tau _2}$, we have
$$[ M _{(a, x)} : L _{(a,y)} ] = [ M _{(a _1, x _1 )} : L _{(a _1,y _1)} ][ M _{(a _2, x _2 )} : L _{(a _2,y _2)} ].$$
\end{corollary}

\begin{proof}{
We have $[ M _{(a _*, x _*)} : L _{(a _*,y _*)} ] = \dim H ^{\bullet} _{G _* ( a _* ) x _*} ( \mathsf{IC} ( G _* ( a _* ) y _* ) )$ for $* = \emptyset, 1,2$ in a uniform fashion by Theorem \ref{mult}. Hence, we deduce}
\begin{align*}
\dim & ~H ^{\bullet} _{G ( a ) x} ( \mathsf{IC} ( G ( a ) y ) ) = \dim H ^{\bullet} _{( G _1 ( a _1 ) \times G _2 ( a _2 ) ) x} ( \mathsf{IC} ( (G _1 ( a _1 ) \times G _2 ( a _2 ) ) y ) ) \\
&  = ( \dim H ^{\bullet} _{G _1 ( a _1 ) x _1} ( \mathsf{IC} ( G _1 ( a _1 ) y _1 ) ) ) ( \dim H ^{\bullet} _{G _2 ( a _2 ) x _2} ( \mathsf{IC} ( G _2 ( a _2 ) y _2 ) ) ),
\end{align*}
which implies the assertion.
\end{proof}

\subsection{Specialization of parameters}

We start with a corollary of a well-known result as is presented in Lemma 2.3.3 of \cite{CG}, for example. 

\begin{proposition}\label{nearby}
Let $R$ be a $\mathbb C [\mathbf t]$-algebra of finite rank. Let $M$ be a $R$-module which is free as a $\mathbb C [\mathbf t]$-module. Assume that we have an $R$-submodule $N \subset M$ whose localization $\mathbb C [\mathbf t ^{\pm 1}] \otimes _{\mathbb C [\mathbf t]} N$ is a free $\mathbb C [\mathbf t ^{\pm 1}]$-module. Then, there exists an $R$-submodule $N'\subset M$ such that
$$\mathbb C [\mathbf t ^{\pm 1}] \otimes _{\mathbb C [\mathbf t]} N = \mathbb C [\mathbf t ^{\pm 1}] \otimes _{\mathbb C [\mathbf t]} N' \subset \mathbb C [\mathbf t ^{\pm 1}] \otimes _{\mathbb C [\mathbf t]} M$$
and the quotient $M/N'$ is free over $\mathbb C [\mathbf t]$. \hfill $\Box$
\end{proposition}

\begin{corollary}\label{spec}
Let $a ^t = a \exp ( \gamma t )$ be a one-parameter family depending on $t\in\mathbb R$, with $a^t\in\mathcal T_0$ for all but finitely many values of $t$, and where
$$\gamma \in \mathfrak t \oplus \{ 0\} \oplus \mathbb R ^2 _{\ge 0} \subset \mathrm{Lie} ( T \times ( \mathbb C ^{\times} ) ^3 ).$$
Let $\tau$ be a marked partition adapted to each of $a^t$.  Then, we have
$$\Psi ( L _{(a^0, v _{\tau})}) \subset \lim _{t \to 0} \Psi  ( L _{(a^t, v _{\tau})}).$$
In particular, the module $L _{(a^0, v _{\tau})}$ is tempered if $L _{(a^t, v _{\tau})}$ defines a tempered module in $($at least$)$ one of the region
$$\epsilon > t > 0 \text{ or }- \epsilon < t < 0 \text{ for some positive number } \epsilon \ll 1.$$
\end{corollary}

\begin{proof}
Let $\underline{a^t} \in \mathfrak t \oplus \mathbb R^2$ be the
element defined from $a^t$ via the statement of Proposition
\ref{inclW}. We have $\underline{a^t} = \underline{a} + t \gamma$. We
choose $A$ in Proposition \ref{inclW} so that $a^t \in A$. Let $\ell
\subset \mathfrak a$ denote the line $\{\underline{a ^t}\} _{t \in
  \mathbb R}$. We have the corresponding surjection $\mathbb C
[\mathfrak a] \rightarrow \mathbb C [ \mathbf t ]$. Therefore, we have a
family of $\mathbb H _{\mathbf t} := \mathbb C [ \mathbf t ] \otimes _{\mathbb C
  [\mathfrak a]} \mathbb H_{a} ^+$-modules $M_{\mathbf t} := \mathbb C [ \mathbf t ]
\otimes _{\mathbb C [\mathfrak a]} H_{\bullet} ^A ( \mathcal E _{v
  _{\tau}} )$. We apply Proposition \ref{nearby} to the family of
maximal $\mathbb H _{a^t}$-submodules $(t\neq 0)$ of $M _{(a^t, v
  _{\tau})}$ for which the corresponding quotients are $L _{(a^t, v
  _{\tau})}$. Every such $\mathbb H_{a^t}$-submodule
extends to an $\mathbb H_{\mathbf t}$-submodule $N \subset M_{\mathbf t}$ whose quotient
specializes to $L _{(a^t,v _{\tau})}$ unless $t = 0$. Since a
finite-dimensional $W$-module is rigid under flat deformation, it
follows that the $W$-module structure of $N$ must be constant along $t
\in \mathbb R$. Therefore, $\mathbb C \otimes _{\mathbb C[\mathbf t]} M_{\mathbf t} / N$
contains a non-trivial $\mathbb H _a$-module which contains $L _{v
  _{\tau}}$ (as $W$-modules). This must be $L _{(a^0,
  v_{\tau})}$. Since $M_{\mathbf t}$ is an algebraic family of $\mathbb
H_{\mathbf t}$-modules, we have 
$$\Psi ( L _{(a^0, v _{\tau})}) \subset \lim _{t \to 0} \Psi ( \mathbb C_{t} \otimes _{\mathbb C[\mathbf t]} M_{\mathbf t} / N ) = \lim _{t \to 0} \Psi  ( L _{(a^{t}, v _{\tau})}),$$
where $\mathbb C_{t}$ is the quotient of $\mathbb C[\mathbf t]$ by the ideal
$(\mathbf t -t)$. The rest of the assertions are clear. 
\end{proof}

\begin{corollary}\label{specds} With the notation from Corollary
  \ref{spec}, assume that $L_{(a^t,v_\tau)}$ is a discrete series for
  $t\in (-\epsilon,\epsilon)\setminus\{0\}.$ Then $L_{(a^0,v_\tau)}$
  is a discrete series.
\end{corollary}

\begin{proof} By \ref{spec}, $\Psi ( L _{(a^0, v _{\tau})}) \subset
  \lim _{t \to 0} \Psi  ( L _{(a^t, v _{\tau})}).$ Let $w\cdot a^t$ be
  a one-parameter family of weights, $w\in W,$ such that $w\cdot
  a^t\in \Psi  ( L _{(a^t, v _{\tau})})$ and $w\cdot a^0\in \Psi ( L
  _{(a^0, v _{\tau})}).$ By the discrete series condition, $\langle \varpi_j, w\cdot a^t \rangle<1$, for all $1\le j\le 
  n,$ and for all $t\in(-\epsilon,\epsilon)\setminus\{0\}.$ Since
  $\langle \varpi_j, w\cdot a^t \rangle$ is continuous and
  linear in $t$, it follows that $\langle 
  \varpi_j, w\cdot a^0\rangle<1$ (for every $j$) as
  well. 
\end{proof}

\section{Parameters corresponding to discrete series}\label{sec:3}

Recall that for any finite dimensional
$\mathbb H_a$-module $V$, we denote by $\Psi(V)\subset T$ the set of
its $R(T)$-weights.

\subsection{Distinguished marked partitions}\label{sec:dmp}

We restrict now to the case of the specialized affine Hecke algebra of type
$C_n$ with $\vec q=(-1, q^{m},q),$ $q\in \mathbb R_{>1},
m\in\mathbb R,$ and we assume the genericity condition, i.e.,
$ m  \notin \frac 12\mathbb Z.$

Let $a=(s,\vec q)\in \mathcal T_0$ be given.

\begin{definition}\label{d:dist} We say that $a$ (or $s$) is distinguished if the dense
  $G(a)$-orbit on $\mathfrak N^a$ is parameterized by a marked
  partition $(\{I_j\}_{j=1}^k, \delta)$ which satisfies:
\begin{enumerate}
\item $\max\underline{I_1}>\max\underline{I_2}>\dots>\max\underline{I_k}$;
\item $\min\underline{I_1}<\min\underline{I_2}<\dots<\min\underline{I_k}$;
\item  $\delta(I_j)=\{0,1\},$ for all $j$ (which in particular means
  $q_1\in \underline{I_j}$ for all $j$).
\end{enumerate}
We call such a marked partition distinguished as
well.
\end{definition}

Notice that the distinguished marked partitions are in one to one
correspondence with partitions of $n$ by a ``folding'' procedure: for
every $J\in \{I_j\}_{j=1}^k,$ define $\underline\#J$ to be the number
of elements in $\underline J$ strictly smaller than $q_1$, and $\overline\#J$ to be the number
of elements in $\underline J$ greater than or equal to $q_1$. If
$\mathsf{mp}(\sigma)$ is a distinguished marked partition, then one can build a
left-justified nondecreasing 
partition (tableau) $\sigma$  of $n$, as follows: put $\overline\#
I_1$ boxes on the first row and $\underline\# I_1$ boxes on the first
column below the first row (so the $I_1$ looks bent like a hook), then
add $\overline\#I_2$ boxes on the second row and $\underline\#I_2$ on
the second column, below the second row etc. Remark that, in the end,
the diagonal of the tableau $\sigma$ has boxes exactly corresponding
to the markings of $\mathsf{mp}(\sigma)$ (see figure \ref{fig:fold}).

\begin{figure}[h]
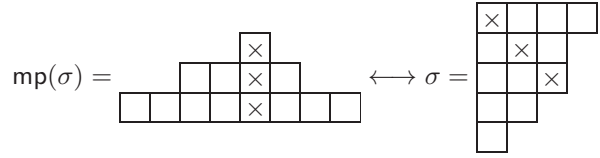
\label{fig:fold}
$$\mathsf{mp}(\sigma)=\young(::::\times,::\hfil\hfil \times\hfil,\hfil\hfil\hfil\hfil
\times\hfil\hfil\hfil)\longleftrightarrow \sigma=\young(\times\hfil\hfil\hfil,\hfil
\times\hfil,\hfil\hfil \times,\hfil\hfil,\hfil)
$$
\caption{The correspondence
  $\mathsf{mp}(\sigma)\leftrightarrow\sigma$, for $\sigma=(4,3,3,2,1).$}
\end{figure}

\begin{theorem}[\cite{O}, Lemma 3.31]\label{tempss} Assume $s\in T$, and
  $a=(s,\vec q)$ is  as above. Then there exists a discrete series
  module with central character $s$ if and only if $s$ is
  distinguished in the sense of Definition \ref{d:dist}.
\end{theorem}

In particular, a distinguished semisimple $a$ (or $s$) corresponds to
a partition $\sigma$ of $n$. We write $a_\sigma$ and $s_\sigma$ to
emphasize this dependence. {(We remark that $a_{\sigma}$ and $s_{\sigma}$ are well-defined up to $\mathfrak S _n$-action even if we require $a_{\sigma} \in \mathcal T _{0}$)}. Notice that, by \S \ref{sec:1.3}, the marked partition
$\mathsf{mp}(\sigma)$ above parameterizes the open $G(a_\sigma)$-orbit
in $\mathfrak N^{a_\sigma}.$ The goal of this section is to identify which $G(a_\sigma)$-orbit in $\mathfrak
N^{a_\sigma}$ parameterizes the discrete
series $\mathbb H_{a_\sigma}$-module under Theorem \ref{eDL}. By
Propositions \ref{p:MP1} and \ref{p:MP2}, we need to describe a marked
partition, denoted $\mathsf{ds}(\sigma)$ or $\mathsf{ds}(s_\sigma)$, which we may regard as a representative of an orbit via the map
$\Upsilon.$

\subsection{Algorithm}\label{sec:alg}
We start with a distinguished marked partition $\mathsf{mp}(\sigma)$
corresponding to a partition $\sigma$ of $n$ as in \S
\ref{sec:dmp}, and let $s_\sigma$ denote the corresponding semisimple
element. We put integer coordinates $(i,j)$
in the boxes of $\sigma$ such that the boxes on the first row have
coordinates: $(1,-1),$ $(2,-1)$, $(3,-1)$ etc., the boxes on the
second row: $(1,-2)$, $(2,-2)$, $(3,-2)$ etc., the numbering starting
from the left. Note that the boxes of the diagonal have coordinates
$(i,-i).$

We define a function on the boxes of $\sigma,$ which we call an $e$-function. For a box $(i,j)$, we
set{
\begin{equation}\label{eq:3.2}e(i,j)= \log_q (q_1q^{i+j}) = m+i+j.\end{equation}
Let $\mathtt{av} ( I )$ denote the sum of all $e$-values of $I \subset \sigma$.}

Given $\sigma,$ the following algorithm gives a marked partition
$\mathsf{out}(\sigma)$ which turns out to parameterize the discrete series with
central character $s_\sigma$ (i.e., $\mathsf{out}(\sigma)=\mathsf{ds}(\sigma)$).

\begin{algorithm}\label{al:ds}
\begin{enumerate}
\item Set $\ell=0,$ $\sigma_{(\ell)}=\sigma,$ $L^+=L^-=\emptyset.$
  ($L^+$ and $L^-$ will be collections of subsets of  $\sigma$.)
\item Find the unique $(i,j)\in \sigma_{(\ell)}$ such that $e(i,j)$ or
  $- e(i,j)$ attains the maximum in the set $\{ | e(i,j) |:
  (i,j)\in \sigma_{(\ell)}\}.$ 
\begin{enumerate}
\item If the maximum is at $e(i,j)$, append the set (horizontal strip)
  $\{(i-k,j)\in
  \sigma_{(\ell)}: k\ge 0\}$ to $L^+$.
\item If the maximum is at $- e(i,j)$, append the set (vertical strip)
  $\{(i,j+k)\in
  \sigma_{(\ell)}: k\ge 0\}$ to $L^-$. 
\end{enumerate}
Remove the horizontal or vertical string as above from
$\sigma_{(\ell)}$ and call the resulting partition
$\sigma_{(\ell+1)}.$ If $\sigma_{(\ell+1)}\neq\emptyset,$ increase
$\ell$ to $\ell+1$ and go back to the beginning of step 2.

\item Set $L=\emptyset.$ (This will be a collection of sets.) For every $-n\le k\le n,$ form 
\begin{equation}
L_k^+=\{I\in L^+:\min_{(i,j)\in I} e(i,j)=m+k\},\quad
L_k^-=\{I\in L^-:\max_{(i,j)\in I} e(i,j)=m+k-1\}.
\end{equation} 
For every $k$, order the elements in $L_k^+,$ respectively $L_k^-$
decreasingly with respect to their cardinality: $I^+_{k,1},\dots,
I_{k,h_1}^+$ and $I^-_{k,1},\dots, I^-_{k,h_2}.$ By adding empty sets
at the tail of the appropriate sequence, we may assume $h_1=h_2.$ Then for
$j=1,\ldots,h_1,$ form the segment $I_{k,j}^+\sqcup I_{k,j}^-$, and append it
to $L.$ (Notice that $I_{k,j}^+\sqcup I_{k,j}^-$ is a
  segment since we have an identification of $\mathsf{mp} ( \sigma )$ with $\sigma$.) 

Then $L$ is the collection of segments in the marked partition
$\mathsf{out}(\sigma).$ We specify the marking $\delta$ next.

\item Define a temporary marking $\delta'$ first. For every $I\in
  L,$ let $e(I)$ denote the set of $e(i,j)$ for $(i,j)\in I.$ Recall
  that $I$ could be marked only if $m\in e(I),$ and if so, the
  marking could only be on the box $(i,j)$ with $e(i,j)=m.$ Set
\begin{equation}
\delta'(I)=\left\{\begin{aligned} 1, &\text{  if } m\in
    e(I) \text{ and } \mathtt{av} ( I ) > 0,\\ 0, &\text{ otherwise.} \end{aligned}\right.
\end{equation}
We refine $\delta'$ to $\delta$ by removing the marking of any
segment $I$ which is dominated by marked segment $I'.$ 

\end{enumerate}
\end{algorithm}

\begin{remark}
\begin{enumerate}
\item The hypothesis that $a=(s,\vec q)$ is generic is essential for the
  uniqueness of the box $(i,j)$ realizing the maximum in step 2 of the algorithm.
\item The first two steps of the algorithm are identical with
the algorithm conjectured by Slooten (\cite{Sl}) for a generalized Springer
correspondence for {the graded Hecke algebra $\overline{\mathbb
    H}_{n,m}$ (Definition \ref{gradedC})} with generic
 unequal
labels. We will see that this algorithm is equivalent with the one
described by Lusztig-Spaltenstein (\cite{LS}) for $\overline{\mathbb
  H}_{n,m}$ with (representative) generic unequal labels
constructed from cuspidal local systems in Spin groups. We explain this in more detail in \S\ref{sec:cuspidal}.
\item To clarify the algorithm, we offer an example. Consider $n=14,$
  and the partition $\sigma=(4,3,3,2,1),$ assuming that
  $2<m<\frac 52,$ see figure \ref{fig:14} (in
  the figure an entry $k$ in the box means the $e$-value is $\log( q_1q^{k}) = m+k$).

\begin{figure}[h]\label{fig:14}
$$
\begin{tableau}
:.{0}.{1}.{2}.{3} \\
:.{{-1}}.{0}.{1} \\
:.{-2}.{{-1}}.{0} \\
:.{{-3}}.{{-2}} \\
:.{{-4}}\\
\end{tableau}$$
\caption{Partition $\sigma=(4,3,3,2,1)$ and $2<m<\frac 52$}
\end{figure}

If we identify $I$ with $e(I)$ for a segment $I$, then we find
\begin{align*}
&L^+=\{[m,m+1,m+2,m+3],[m-1,m,m+1],[m-2,m-1,m],[m-2]\}\\
&L^-=\{[m-4,m-3]\}.
\end{align*}
We separate the segments based on where they begin or end:
$L^+_0=\{[m,m+1,m+2,m+3]\},$ 
$L^+_{-1}=\{[m-1,m,m+1]\}$, 
$L^+_{-2}=\{[m-2,m-1,m],[m-2]\}$, and
$L^-_{-2}=\{[m-4,m-3]\}.$ Next we may combine the
segment in
$L^+_{-2}$ with the longest segment in $L^-_{-2}$. The resulting marked partition
$\tau$ has support $\mathbf I$ given by the segments $I_1,I_2,I_3,I_4$, such that
$e(I_1)=[m,m+1,m+2,m+3],$
$e(I_2)=[m-1,m,m+1],$ and
$e(I_3)=[m-4,m-3,m-2,m-1,m],$ $I_4=[m-2].$ According
to the algorithm, we temporarily mark the first three segments at $m,$ but
then since $I_3\lhd I_2\lhd I_1,$ we remove the markings on $I_2$ and
$I_3.$ In conclusion, the output of the algorithm is the marked
partition $\mathsf{out}(\sigma)$ (see figure \ref{fig:3.4}) with support given by $\{I_1,I_2,I_3,I_4\}$ and a single
marking on $I_1$. (This marked partition is in the same orbit with the
one where all three $I_1,I_2,I_3$ are marked.)

\begin{figure}[h]\label{fig:3.4}
$$\begin{tableau}
:.{{\mathbf 0}}.{1}.{2}.{3} \\
:.{-1}.{0}.{1} \\
:.{-4}.{{-3}}.{-2}.{-1}.{0} \\
:.{-2}\\
\end{tableau}   \quad \text{or, after aligning the rows, } $$
$$  \mathsf{out}(\sigma)=\young(::::{\times}\hfil\hfil\hfil,:::{\hfil}{\hfil}{\hfil},{\hfil}{\hfil}{\hfil}{\hfil}{\hfil},::\hfil) $$
\caption{Output of Algorithm \ref{al:ds} when $\sigma=(4,3,3,2,1)$ and $2<m<{\frac 52}$}
\end{figure}

\end{enumerate}

\end{remark}

The main result of this section is:

\begin{theorem}\label{t:main}
Let $\sigma$ be a partition of $n,$ and let $a_\sigma,s_\sigma$ be the
semisimple elements constructed from $\sigma$ in \S
\ref{sec:dmp}. The discrete series $\mathbb H_{a_\sigma}$-module (with central character
$s_\sigma$) is $L_{(a_\sigma,\Upsilon(\mathsf{out}(\sigma)))}$, where
$\mathsf{out}(\sigma)$ is the marked partition constructed in
Algorithm \ref{al:ds}. In other words, $\mathsf{out}(\sigma)=\mathsf{ds}(\sigma).$
\end{theorem}

The proof will be divided into several parts in the next sections.

\begin{example}\label{ex:onehook}{\bf (One hook partitions.)} Before proving Theorem \ref{t:main},
let us present a particular case of the algorithm. Assume
$\sigma$ is a partition given by a single hook, i.e.,
$\sigma=(k,1^{n-k}),$ for some $1\le k\le n.$
This means that the semisimple element is
$s_\sigma=(q_1q^{k-1},q_1q^{k-2},\dots,q_1,\dots,q_1q^{k-n}).$
In this case, 
\begin{equation}
\mathfrak N^{a_\sigma}=\bigoplus _{i=1}^n\mathbb C
v_{\epsilon_i-\epsilon_{i+1}}\oplus \mathbb C
v_{\epsilon_k},\quad\text{and } G (s_\sigma)=(\mathbb C^\times)^n\text{
(i.e., the maximal torus)},
\end{equation}
so there are $2^n$ orbits in $G (s_\sigma)\backslash \mathfrak
N^{a_\sigma},$ each orbit corresponding to a subset $\mathcal S$ of
$\{v_{\epsilon_i-\epsilon_{i+1}},\ 1 \le i < n, v_{\epsilon_k}\}.$ To $\mathcal S$,
there corresponds a marked partition $\tau_{\mathcal S}=(\mathbf I,\delta),$ as
follows: for every maximal string of 
consecutive weight vectors
$\{v_{\epsilon_i-\epsilon_{i+1}},\dots,v_{\epsilon_{i+t-1}-\epsilon_{i+t}}\}$
  in $\mathcal S$, we attach a segment $I=[i,i+1,\dots, i+t-1] \in \mathbf I,$ of
  length $t$ with $e$-values $(m+k-i,\dots, m+k-i-t+1).$
  In addition, we mark at $q_1$, and write $\delta=1$, if $v_{\epsilon_k}\in \mathcal S,$ and we don't
  mark, and write $\delta=0$, otherwise.

We remark that, as a consequence of the results about weights, one
sees that any $R(T)$-weight space of an irreducible $\mathbb
H_{a_\sigma}$-module with central character $s_\sigma$, i.e.,
parameterized by a marked partition of the form $\tau_{\mathcal S},$ is
one-dimensional.
  
By applying the algorithm explicitly, as a corollary of Theorem
\ref{t:main}, we find that the discrete series $\mathsf{ds}(\sigma)$
is parameterized by the marked partition $\tau=(\mathbf I,\delta)$ as follows:

\begin{enumerate}
\item[(a)] if $q_1q^{k-n}>1,$ then $\tau$ has $\mathbf
  I=\{[1,2,\dots,k],[k+1],[k+2],\dotsc,[n]\}$ and $\delta=1$;
\item[(b)] if $q_1q^{k-n}<1,$ and
\begin{enumerate}
\item[(b1)] $(q_1q^{k-n})^{-1}>q_1q^{k-1},$ then $\tau$ has
  $\mathbf I=\{[1,2,\dots,n]\}$ and $\delta=0$;
\item[(b2)] $q_1q^{-1}<(q_1q^{k-n})^{-1}<q_1q^{k-1},$ then
  $\tau$ has $\mathbf I=\{[1,2,\dots n]\}$ and $\delta=1$;
\item[(b3)] there exists $l<0$ with the property that
  $q_1q^{l-1}<(q_1q^{k-n})^{-1}<q_1q^l$, then $\tau$ had
  $\mathbf I=\{[1,2,\dots,k],[k+1],\dots,[l-1],[l,l+1,\dots,n]\}$ and
  $\delta=1.$ 
\end{enumerate}

\end{enumerate}

\end{example}

\begin{convention}\label{cfcn}
A bijection $c : \sigma \rightarrow [ 1, n ]$ is called a $c$-function for $\sigma$ if it satisfies the following condition
$$c ( \boxed i ) > c ( \boxed j ) \text{ if } e ( \boxed i ) < e ( \boxed j ) \text{ for } \boxed i, \boxed j \in \sigma.$$
\end{convention}

In the following, using an appropriate $c$-function if necessary, we identify the set of boxes $\Box \in \sigma$ with an interval $[1,n]$. Moreover, we define $s_{\sigma}$ so that $e ( \Box ) = \log _q \langle \epsilon _{c ( \Box )}, s _{\sigma} \rangle$ for each $\Box \in \sigma$. In addition, we identify a segment $I$ (adapted to $s _{\sigma}$) as a set of boxes in $\sigma$. Notice that we have $e ( I ) = \log _q ( \underline{I} )$ in this setting.

\subsection{A particular case: $(\pm)$-ladders}\label{sec:ladder}

Before we begin, we present a corollary of Theorem \ref{mult} which is used
  repeatedly in the proofs.

\begin{corollary}\label{indmult}
Assume $n=n_1+n_2,$ and let $\mathbb H_P$
  be the Hecke algebra for $GL(n_1)\times Sp(2n_2)$. Consider {$a = (s,\vec{q}) \in \mathcal T _0$ which decomposes $s=s_1\times s_2$ with $s_1\in GL(n_1)$ and $s_2\in Sp(2n_2).$ Let $\tau_1 \in \mathsf{MP} _0 ( s_1, \vec{q} )$, $\tau_2\in \mathsf{MP} ( s_2, \vec{q} )$, and $\tau \in \mathsf{MP} ( a )$ be given.
We assume:
\begin{itemize}
\item For {every
  $\widetilde\tau_2\in \mathsf{MP}(s_2,\vec{q})$}, which gives a strictly larger
  orbit than $\tau_2$, we have $v_{\tau_1}\oplus v_{\widetilde\tau_2} \notin
  \overline{Gv_\tau}$;
\item{The induction condition (\ref{denom}) in Theorem \ref{indt} is satisfied.} 
\end{itemize}}
Then, we have
\begin{equation} 
[\Ind_{\mathbb H_P}^{\mathbb H}(M^{\mathsf A}_{\tau_1}\boxtimes M_{\tau_2}):L_\tau]=[\Ind_{\mathbb H_P}^{\mathbb H}(M^{\mathsf A}_{\tau_1}\boxtimes L_{\tau_2}):L_\tau],
\end{equation}
and the same formula holds with $L^{\mathsf A}_{\tau_1}$, ${}^{\mathsf 
t}L^{\mathsf A}_{\tau_1}$, or ${}^{\mathsf t}M^{\mathsf A}_{\tau_1}$
in place of $M^{\mathsf A}_{\tau_1}.$
\end{corollary}

\begin{proof}
This is an obvious consequence of Theorem \ref{mult} and of the
exactness of the parabolic induction functor.
\end{proof} 
In the following, whenever we need to apply Corollary
\ref{indmult}, we omit the details of the verification of $v_{\tau_1}\oplus v_{\widetilde\tau_2} \notin \overline{Gv_\tau}$ since they are easily checked by inspection.

\smallskip


We begin with a particular instance of Algorithm \ref{al:ds}: the cases
when the algorithm produces $L^-=\emptyset$ or $L^+=\emptyset.$

\begin{definition}[$(\pm)$-ladder] Let $a_\sigma$ be a distinguished semisimple element as in \S\ref{sec:dmp}. 
A positive ladder corresponding to $a _{\sigma}$ is a marked partition $\tau = ( \mathbf I, \delta )$ adapted to $a _{\sigma}$ which satisfies the following conditions:
\begin{enumerate}
\item We have $\mathbf I = \{ I _1, I _2, \ldots \}$ such that
$$e ( I _i ) = \{ m + 1 - i, \ldots, m + \lambda _i - i \};$$
\item We have $\delta ( \Box ) = 1$ if $e ( \Box ) = m$ and $\Box
  \in I_1$, and $\delta (\Box) = 0$ otherwise.
\end{enumerate}
A negative ladder corresponding to $a _{\sigma}$ is a marked partition $\tau = ( \mathbf I, \delta )$ adapted to $a _{\sigma}$ which satisfies the following conditions:
\begin{enumerate}
\item We have $\mathbf I = \{ I _1, I _2, \ldots \}$ such that
$$e ( I _i ) = \{ m+ i - \lambda _i, \ldots, m + i - 1 \};$$
\item We have $\delta\equiv 0$.
\end{enumerate}
\end{definition}

For every distinguished $\sigma$ there are unique $(\pm)$-ladders: the
positive ladder has the collection of segments $\mathbf I$ as the rows
of $\sigma$, and every $I\in\mathbf I$ with $m\in e(I)$ is marked, while the
negative ladder has the collection of segments $\mathbf I$ as the columns
of $\sigma$, and has no marking.

Recall that in general, the weights $\Psi(L_\tau)$ are a subset of
$W \cdot s_\sigma^{-1}.$ If $\tau$ is a $(\pm)$-ladder, then the weights have
a particular form:

\begin{proposition} \label{p:ladder}
Choose a $c$-function for $\sigma$ (see Convention \ref{cfcn}) in order to fix $s_{\sigma} \in T$. Then, we have
\begin{enumerate} 

\item Assume that $\tau$ is the
    positive ladder for $\sigma.$ Then, we have $\Psi(L_\tau)\subset \mathfrak S_n \cdot s_\sigma^{-1}.$
\item Assume that $\tau$ is the
    negative ladder for $\sigma.$ Then, we have $\Psi(L_\tau)\subset \mathfrak S_n \cdot s_\sigma.$

\end{enumerate}

\end{proposition}

\begin{proof}
The proofs of the two assertions are completely analogous, therefore we
only present the details when $\tau$ is the positive ladder. The proof is by
induction on $k,$ the number of rows of $\sigma,$ or equivalently, the
number of segments $J_i$ in the support of $\tau$. 

In the base case, $k=1,$ the orbit corresponding to $\tau$ is
$G$-regular in $\mathfrak N=\mathfrak N_1$, and the corresponding
module is the Steinberg module. It follows that
$\Psi(L_\tau)=\{s_\sigma^{-1}\},$ which proves the assertion in this case.

Assume the result holds for all $\sigma'$ with less than $k$ rows, and
assume $\sigma$ has $k$ rows. We want to show that
for every weight $w^{-1} \cdot s_\sigma^{-1}$, we have $w^{-1}i>0$ for $1\le i\le n$
(which implies that $w\in \mathfrak S_n$), or, equivalently, that
$w^{-1} c(\boxed{j})>0$ for every box $\boxed j$ of $\tau.$

Let $\tau_1$ and $\tau_2$ be the positive ladder partitions corresponding to the last row, respectively the first $k-1$ rows, of $\sigma,$ and let $s_1,s_2$ be the corresponding semisimple
elements. We form the $\mathbb H^{\mathsf{A}} _{\lambda_k}$-module
(one dimensional) $L_{\lambda_k}^{\mathsf A}=M_{\lambda_k}^{\mathsf A}$ corresponding to
$( s_1, q, v_{\tau_1})$, and let $M_{\tau_2}$ be the standard module of $\mathbb
H_{n-\lambda_k}$. Theorem \ref{indt} applies, and
we have
\begin{equation}
M_\tau\cong \mathsf{Ind}_{\mathbb H ^{\mathsf A} _{\lambda_k} \times
  \mathbb H_{n-\lambda_k}}^{\mathbb H_{n}} ( L_{\tau_1}^{\mathsf A}\boxtimes M_{\tau_2}).
\end{equation}
(The notation is as in Convention \ref{conv:1.21}.) 
Using Corollary \ref{indmult}, we deduce that
$$[ \mathsf{Ind}_{\mathbb H ^{\mathsf A} _{\lambda_k} \otimes
  \mathbb H_{n-\lambda_k}}^{\mathbb H_{n}} ( L_{\tau_1}^{\mathsf A}\boxtimes L_{\tau_2}) : L_{\tau}] > 0.$$
For every {minimal length} coset representative $w$ of $W_n/ (\mathfrak
S_{\lambda_k}\times W_{n-\lambda_k} )$, we analyze the homology
$H_\bullet( \mathcal E _{v_\tau}^{a_\sigma} [w \cdot s_\sigma^{-1}])$ to see if the
$w \cdot s_\sigma^{-1}$-weight space is nonempty. By the induction
hypothesis and Theorem \ref{wt-ind}, we have $w c(\boxed j) > 0$ for all $\boxed j\in \tau_2.$ It remains to show that the same holds for all $\boxed j\in \tau_1.$

Notice that the minimal $e$-value $e_\mathsf{min}$ in $\tau$ is attained by an element
$\boxed{\mathsf{min}}$ of $\tau_2.$ Recall that this makes
$c(\boxed{\mathsf{min}}) = n$. There are at most two elements in
$\tau$ which have $e$-value equal to $(e_\mathsf{min} + 1)$: one in
$\tau_2$, denoted $\boxed 2$, and, if $\tau_1$ is not a singleton, one
in $\tau_1$, denoted $\boxed 1.$ 

If $w c(\boxed{\mathsf{min}})>0,$ then in order to have $v_{\tau_1}\in {} ^{w} \mathbb V ( a )$, one must
have $w c (\boxed{j})>0,$ for all $\boxed{j} \in \tau_1$ by Proposition \ref{pw-ind} and Corollary \ref{wcrit}.

If $w c(\boxed{\mathsf{min}})<0,$ then we have $v_\alpha\in {} ^{w} \mathbb V( a )$ for
$\alpha=\epsilon_{n-j}-\epsilon_{n},$ where we set $c ( \boxed{j} ) = n - j$ for $j=1,2$ (the first case follows by the induction hypothesis, and the second case is by Proposition \ref{pw-ind} and Corollary \ref{wcrit}, and appears if $\tau_1$ is not a singleton). But
this implies that in order for
$$H_\bullet( \mathcal E _{v_\tau}^{a_\sigma} [w \cdot s_\sigma^{-1}]) = H_\bullet((\mu_w^{a_\sigma})^{-1}(v_\tau))$$
to contribute a non-trivial weight space of $L_\tau,$ the orbit $\mathcal O _{\tau}$ must meet
$\operatorname{Hom}_{\mathbb C}(\mathbb C,\mathbb C^2)\subset \mathfrak
N^a$ if $\tau_1$ is
not a singleton, respectively $\operatorname{Hom}_{\mathbb C}(\mathbb
C,\mathbb C)\subset \mathfrak
N^a$ if $\tau_1$ is a singleton, in its open dense part. But since
$\boxed{\mathsf{min}}$ and $\boxed{1}$ are not in the same segment of $\tau$, this is
not the case for $\tau.$ 
\end{proof}

\begin{corollary} \label{c:ladder}
\begin{enumerate}
\item Assume Algorithm \ref{al:ds} produces $L^-=\emptyset$ for $\sigma.$ Then
  the output of the algorithm $\mathsf{out}(\sigma)$, which is the positive ladder, is a
  discrete series. In particular, this is the case when
  $q_1>q^{n-1}>1.$ 
\item Assume Algorithm \ref{al:ds} produces $L^+=\emptyset$ for $\sigma.$ Then
  the output of the algorithm $\mathsf{out}(\sigma)$, which is the negative ladder, is a
  discrete series. In particular, this is the case when {$q_1<q^{1-n}<1.$}
\end{enumerate}
\end{corollary}

\begin{proof} Assume $L^-=\emptyset,$ so that $\mathsf{out}(\sigma)$ is the positive ladder $\tau$. Then any weight
  $w \cdot s_\sigma^{-1}$ of $L_\tau$ is given by a permutation of the
  entries of $s_\sigma^{-1}$ by Proposition \ref{p:ladder}. 
I.e., we have $w \in \mathfrak S _n$. It follows that
$$\left< \epsilon _{k}, w \cdot s _{\sigma} ^{-1} \right> = \left< \epsilon _{w^{-1} k}, s _{\sigma} ^{-1} \right> \le q_1 ^{-1} q ^{n-1} < 1 \text{ for each } k=1,\ldots,n.$$
Therefore, we have $\langle \varpi_j, w \cdot s_\sigma^{-1} \rangle = \prod _{k=1}^j \left< \epsilon _{k}, w \cdot s _{\sigma} ^{-1} \right>< 1$ for all $1\le j\le n.$ The case $L^+=\emptyset$ is analogous.
\end{proof}

\subsection{Proof of the main theorem} We continue with the
proof of Theorem \ref{t:main}. Recall that 
$\sigma$ is a partition of $n$.
We wish to prove that
$\mathsf{out}(\sigma)$ is tempered, or equivalently $\mathsf{out}(\sigma)=\mathsf{ds}(\sigma).$ Assume that in the first two steps of {Algorithm \ref{al:ds}}, the
segments produced are $L^+\sqcup L^-=\{I_1,\dots,I_N\}.$ Section
\ref{sec:ladder} proves the claim when either $L^+=\emptyset$ or
$L^-=\emptyset.$  We may now assume that both $L^+$ and $L^-$ are nonempty. {We first prove two more structural results on weights in particular cases.}

\begin{proposition}\label{p:rt1} Assume that the algorithm runs
  as $$I_1,\dots,I_r\in L^+,I_{r+1},\dots,I_{r+t}\in L^-,$$
for some $0\le t\le r.$ Choose a $c$-function for $\sigma$ (see Convention \ref{cfcn}) in order to fix $s_{\sigma} \in T$. If $w \cdot s_\sigma^{-1} \in \Psi ( L_{\mathsf{out}(\sigma)} )$ for $w \in W$, then
  $w c(\boxed x)>0$ for all $\boxed x\in \tau$ such that $e(\boxed x)\ge m+t-r.$
\end{proposition}

\begin{proof}
The proof is by induction. For every $0\le u\le t,$ define the
marked partitions $\tau_1^{(u)}$ and $\tau_2^{(u)}$ as follows: the
support of $\tau_1^{(u)}$ is $\{I_{r+u+1}\}$ and it is unmarked, while
the support of $\tau_2^{(u)}$ is $\{I_1,\dots,I_{r-u},I_{r-u+1}\sqcup
I_{r+u},\dots, I_r\sqcup I_{r+1}\}$, and every $\boxed x$ such that
$e(\boxed x)=m$ is marked. Note that $e(\boxed y)<m$ for all $\boxed y\in \tau_1^{(u)}.$ For $u=0,$ 
$\tau_1^{(0)}$ has support $\{I_{r+1}\}$, while $\tau_2^{(0)}$ has
support $L^+,$ and in fact it is a positive ladder. We have $\tau _2^{(0)} = \mathsf{out} ( \sigma' )$ for some smaller partition $\sigma'$. By Proposition \ref{p:ladder}, we deduce the assertion for $t = 0$. Notice that $\tau_1^{(t)}=\emptyset,$ and $\tau_2^{(t)}=\mathsf{out}(\sigma).$

We proceed by induction on $u$ to prove that $\tau_2^{(u)}$, with $u=t$, satisfies the assertion. As just mentioned, this holds for $u=0.$ Let $u>0$ be fixed, and assume the theorem holds for all smaller
$u'\le u$, and we will prove it for $u+1.$ We define $\sigma ^{(u)}$ as the partition such that $\mathsf{out} ( \sigma ^{(u)}) = \tau_2^{(u)}$.

Let $n _1 ^{(u)}$ and $n _2 ^{(u)}$ be the sizes of $\tau _1 ^{(u)}$
and $\tau _2 ^{(u)}$, respectively. Let $\mathbb H _P$
  be the subalgebra of $\mathbb H _{n^{(u+1)}}$ corresponding to $\mathop{GL} ( n _1 ^{(u)} ) \times \mathop{Sp} ( 2 n _2 ^{(u)} ) \subset \mathop{Sp} ( 2 n _2 ^{(u+1)} )$. We regard $v _{\tau _1 ^{(u)}}$ as a regular nilpotent Jordan normal form
of $\mathfrak{gl} _{n _1 ^{(u)}}$ and $v _{\tau _2 ^{(u)}}$ as
a normal form (see Proposition \ref{p:MP1}) of an exotic representation of $\mathop{Sp} ( 2 n _2
^{(u)})$. We have $M_{\tau_1^{(u)}}^{\mathsf A}=L_{\tau_1^{(u)}}^{\mathsf A}$ as
(one-dimensional) modules for $\mathbb H ^{\mathsf A} _{n_1}$.

\begin{claim}
We have $\mathcal O _{\tau_1^{(u)} \times \tau _2 ^{(u)}} \subset \overline{\mathcal O _{\tau _2 ^{(u+1)}}}$ and
$$1 + \dim \mathcal O _{\tau_1^{(u)} \times \dot{\tau} _2 ^{(u)}} = \dim
\mathcal O _{\dot{\tau} _2 ^{(u+1)}},$$
{where $\dot{\tau}$ denotes the marked partition obtained from $\tau$
  by removing the markings.}
\end{claim}

\begin{proof}
We set $I ^{\star} := I _{r+u+1} \cup I _{r - u}$. The segment $J := I
_{r+u+1}$ satisfies $J \sqsubset I$ or $\underline{J} \cap
\underline{I} = \emptyset$ for each $I \in \tau _2 ^{(u)}$. We have $I
_{r-u} \triangleleft I$ or $I _{r-u} \triangleright I$ for $I \in \tau
_2 ^{(u)}$ if and only if $I ^{\star} \triangleleft I$ or $I ^{\star}
\triangleright I$, respectively. It follows that $u_{\dot{\tau} _2 ^{(u+1)}}
= u _{\dot{\tau} _2 ^{(u)}} + u _{\tau _1 ^{(u)}}$. (The definition of $u
_{\tau}$ is as in  Corollary \ref{stab}.) Using Corollary \ref{stab},
we conclude the dimension estimate. The existence of closure relation
is straight-forward since we have an attracting map from $v _{\dot{\tau} _2 ^{(u+1)}}$ to $v _{\tau _1 ^{(u)}} +
v _{\dot{\tau} _2 ^{(u)}}$ defined as the scalar multiplication of the $T$-component of $v _{\dot{\tau} _2 ^{(u+1)}}$ which does not appear in $v _{\tau _1 ^{(u)}} + v _{\dot{\tau} _2 ^{(u)}}$.
\end{proof}

We return to the proof of Proposition \ref{p:rt1}. Notice that both of
$\mathcal O _{\tau_1^{(u)} \times \tau _2 ^{(u)}}$ and $\mathcal O
_{\tau _2 ^{(u+1)}}$ are open subsets of vector bundles over their
projections to $V _2 ^{(s, q)}$ with their fibers isomorphic to $V_1^{(s,q_1)}$. It
follows that the regularity of the orbit closure $\mathcal O
_{\tau_1^{(u)} \times \tau _2 ^{(u)}} \subset \overline{\mathcal O
  _{\tau _2 ^{(u+1)}}}$ is equivalent to the regularity of the
corresponding orbit closure in $V _2 ^{(s, q)}$. We identify $V _2
^{(s, q)}$ with some type $\mathsf{A}$-quiver representation
space as in \S \ref{sec:1.3}. By the Abeasis-Del Fra-Kraft theorem \cite{ADK},
$\overline{\mathcal O _{\tau _2 ^{(u+1)}}}$ is normal along $\mathcal
O _{\tau_1^{(u)} \times \tau _2 ^{(u)}}$ since its projection to $V _2
^{(s, q)}$ is so.  Since normality implies regularity in codimension
one, it follows that $\dim H ^{\bullet} _{\mathcal O _{\tau _1^{(u)}
    \times \tau _2^{(u)}}} ( \mathsf{IC} ( \mathcal O _{\tau _2
  ^{(u+1)}} ) ) = 1$. Hence, Theorem \ref{mult} implies
\begin{equation}\label{ind2u}
[ M _{\tau _1^{(u)} \times \tau _2^{(u)}} : L _{\tau _2 ^{(u+1)}}]=[ \operatorname{Ind} _{\mathbb H _P} ^{\mathbb H} ( {} ^{\mathsf t} M _{\tau^{(u)} _1} ^{\mathsf A}
\boxtimes M _{\tau^{(u)} _2}) : L _{\tau _2 ^{(u+1)}}] = 1 > 0.
\end{equation}
{Now we choose a $c$-function for $\sigma^{(u+1)}$ to fix
  $s _{\sigma^{(u+1)}}$. Let $w \cdot s _{\sigma^{(u+1)}} ^{-1} \in
  \Psi ( L _{\tau^{(u+1)}})$ ($w \in W$) be given.} Taking into account the fact that we have no $\boxed x \in \tau^{(u)} _1$
such that $e ( \boxed x ) \ge m+t-r$, we deduce
\begin{eqnarray}
w c ( \boxed x ) > 0 \text{ if } e ( \boxed x ) \ge m+t-r.\label{pos0}
\end{eqnarray}
by Theorem \ref{wt-ind}.
\end{proof}

{
\begin{corollary}\label{c:rt1}
Keep the setting of Proposition \ref{p:rt1}. Let $\sigma'$ be a partition so that the corresponding algorithm runs as
$$I_1,\dots,I_r\in L^+,I_{r+1},\dots,I_{r+t-1}\in L^-,\quad 0\le t\le r.$$
Let $\tau_1 = ( \{ I _{r+t} \}, 0 )$ be the (unmarked) marked partition adapted to a semisimple element determined by $e ( I _{r+1} )$. Let $n'$ and $n''$ $(n = n' + n'')$ be the sizes of $\sigma'$ and $\tau_1$, respectively. Then, we have
$$[ \operatorname{Ind} _{\mathbb H _{n''} ^{\mathsf A} \otimes \mathbb H _{n'}} ^{\mathbb H _n} ( {} ^{\mathtt t} M _{\tau _1} ^{\mathsf A} \boxtimes L _{\mathsf{out} ( \sigma' )}) : L _{\mathsf{out} ( \sigma )} ] > 0.$$
\end{corollary}

\begin{proof}
Let us assume that we have $u=t-1$ and
(\ref{ind2u}) as in the proof of Proposition
\ref{p:rt1}. In the notation therein, we have
  $\tau_2^{(t-1)}=\mathsf{out}(\sigma')$, $\tau_1^{(t-1)}=\tau_1$, and
  $\tau_2^{(t)}=\mathsf{out}(\sigma).$ Then the claim follows from
  (\ref{ind2u}) by verifying the hypothesis of Corollary \ref{indmult}.
\end{proof}

Proposition \ref{p:rt1} and Corollary \ref{c:rt1} have the following counterparts, with the
analogous proofs.

\begin{proposition}\label{p:rt2} Assume that the algorithm runs
  as $$I_1,\dots,I_r\in L^-,I_{r+1},\dots,I_{r+t}\in L^+,$$
for some $0\le t\le r.$ {Choose a $c$-function for $\sigma$ (see Convention \ref{cfcn}) in order to fix $s_{\sigma} \in T$.} If $w \cdot s_\sigma^{-1} \in \Psi ( L_{\mathsf{out}(\sigma)} )$ for some $w \in W$, then
  $w c(\boxed x)<0$
  for all $\boxed x\in \mathsf{out}(\sigma)$ such that $e(\boxed x)\le m+r-t.$
\end{proposition}

\begin{proof}
The proof (of Proposition \ref{p:rt1}) works by changing the definition of $\tau
_2 ^{(u)}$ so that the support is 
$$I _1,\ldots,I_{r-u},(I_{r-u+1}\cup I _{r+u}), \ldots, (I_r\cup I_{r+1}),$$
and set the support of $\tau _1 ^{(u)}$ to be $I _{r+u+1}$.
\end{proof}

\begin{corollary}\label{c:rt2}
Keep the setting of Proposition \ref{p:rt2}. Let $\sigma'$ be a partition so that the corresponding algorithm runs as
$$I_1,\dots,I_r\in L^-,I_{r+1},\dots,I_{r+t-1}\in L^+,\quad
0\le t\le r.$$
Let $\tau = ( \{ I _{r+t} \}, 0 )$ be the (unmarked) marked partition adapted to a semisimple element determined by $e ( I _{r+1} )$. Let $n'$ and $n''$ $(n = n' + n'')$ be the sizes of $\sigma'$ and $\tau$, respectively. Then, we have
$$[ \operatorname{Ind} _{\mathbb H _{n''} ^{\mathsf A} \otimes \mathbb H _{n'}} ^{\mathbb H _n} ( M _{\tau _1} ^{\mathsf A} \boxtimes L _{\mathsf{out} ( \sigma' )}) : L _{\mathsf{out} ( \sigma )} ] > 0. $$
\end{corollary}}

\begin{theorem}[also Theorem \ref{t:main}]\label{main:p} The output
  $\mathsf{out}(\sigma)$ of Algorithm
  \ref{al:ds} defines a tempered module.
\end{theorem}

\begin{proof}
Algorithm \ref{al:ds} begins with $I_1\in L^+$ or $I_1\in L^-.$ We present the case $I_1\in L^+,$
the other situation being analogous. We fix some $c$-function of $\mathsf{out}(\sigma)$. Denote
$\mathsf{out}(\sigma)=(\mathbf I,\delta).$ Assume that the algorithm runs as 
$$I_1,\dots, I_r\in L^+, I_{r+1},\dots, I_{r+t}\in L^-,\dotsc \hskip 5mm (0 \le t \le r)$$
and {if $I_{r+t+1} \neq \emptyset$, then it} belongs to $L^+$ if $0< t < r$ and either $L^+$ or $L^-$ if $t = r$. (Note that Proposition \ref{p:ladder} considers the situation $L^{-}=\emptyset$, {which means $t=0$}.) From step 3 of Algorithm \ref{al:ds}, we see that the first $t$ segments in $\mathbf I$ (with respect to $\prec$) are
$I_r\sqcup I_{r+1}, I_{r-1}\sqcup I_{r+2},\dots, I_{r-t+1}\sqcup
I_{r+t}.$

Set $\tau_1=(\{I_r\sqcup I_{r+1}\},\delta|_{\{I_r\sqcup I_{r+1}\}})$,
and let $\tau_2$ be the marked partition obtained from
$\mathsf{out}(\sigma)$ by removing the segment $I_r\sqcup
I_{r+1}$. In terms of the partitions, $I_r$ and $I_{r+1}$ correspond to some row part and column
  part of the partition obtained by extracting $I _1, \ldots, I _{r-1}$ from $\sigma$. So if $\sigma_2$ is the partition obtained from removing these two pieces (which form a hook of $\sigma$), then we have
  $\tau_2=\mathsf{out}(\sigma_2).$ By induction, we may assume that
$\tau_2$ is tempered. Denote by $n_1$, $n_2$, the sizes of $\tau_1$
and $\tau_2$ respectively, $n_1+n_2=n.$
Let $\mathbb H_P$ be the Hecke
algebra for $GL(n_1)\times Sp(2n_2)$. It is clear that $I_r \sqcup I_{r+1}$ attains the minimal $e$-value, and also attains the maximal $e$-value if $r=1$. Hence, Theorem \ref{indt} is applicable and we conclude that
\begin{align*}
& [ \operatorname{Ind}_{\mathbb H_P}^{\mathbb H}(M_{\tau_1} ^{\mathsf A} \boxtimes M_{\tau_2}) : L_{\mathsf{out}(\sigma)}] = [ M_{\mathsf{out}(\sigma)} : L_{\mathsf{out}(\sigma)}] = 1 \text{ if } r > 1\\
& [ \operatorname{Ind}_{\mathbb H_P}^{\mathbb H} ({} ^{\mathtt t} M_{\tau_1} ^{\mathsf A} \boxtimes M_{\tau_2}) : L_{\mathsf{out}(\sigma)}] = [ M_{\tau'} : L_{\mathsf{out}(\sigma)}] = 1, \text{ if } r = 1,
\end{align*}
where $\tau'$ is a marked partition obtained from
  $\mathsf{out}(\sigma)$ by removing its marking on $\tau_1$ (and
  hence $\mathcal O _{\tau} \subset \overline{\mathcal O
    _{\mathsf{out} ( \sigma )}}$ is smooth). In particular, we applied
  Theorem \ref{mult} when $r=1$.

By verifying in addition the hypothesis of Corollary
\ref{indmult}, we find:
\begin{align}
& [ \operatorname{Ind}_{\mathbb H_P}^{\mathbb H}(M_{\tau_1} ^{\mathsf A} \boxtimes L_{\tau_2}) : L_{\mathsf{out}(\sigma)}] = 1 \text{ if } r > 1,\label{mm}\\
& [ \operatorname{Ind}_{\mathbb H_P}^{\mathbb H}({} ^{\mathtt t} M_{\tau_1} ^{\mathsf A} \boxtimes L_{\tau_2}) : L_{\mathsf{out}(\sigma)}] = 1 \text{ if } r = 1.\label{mm'}
\end{align}

Here $\tau_1$ is not a tempered module of $GL(n_1)$ in general, so we need to check that
  $L_{\tau_1}$'s contribution to $\Psi ( L_{\mathsf{out}(\sigma)} )$ satisfies the temperedness
  condition.

To do this, we define the subset $I^*\subset I_r\sqcup I_{r+1}$ by 
$$\boxed x\in I^* \text{ if } e(\boxed x)>-e(\boxed y),\text{ for
  every } \boxed y\in I_{r+1}.$$ (Since $I_r\in L^+$ is picked in the
algorithm before $I_{r+1}\in L^-,$ we have
$I^*\neq\emptyset.$) If $\boxed x\in I^*,$
then we have
\begin{equation}\label{eest}
e(\boxed x)> \max \{ | e(\boxed y) | \mid \boxed y\in
I_{r+k} \} \text{ for all }k\ge 1,
\end{equation}
(equivalently, $c(\boxed x)<c(\boxed y)$ for every $\boxed y\in I_{r+k})$ since the maximal $| e |$-value in $\sqcup_{k\ge 1} I_{r+k}$ is in $I_{r+1}.$ Moreover,
we deduce $e(\boxed x)>m+t-r$ for every $\boxed x\in I^*$ since otherwise we have
$$\max \{ e(\boxed y) \mid \boxed y\in I_{r+t+1} \} > \max \{ | e(\boxed y) | \mid \boxed y\in I_{r+1} \},$$
which contradicts Algorithm \ref{al:ds}.

The reason for defining $I^*$ is that it gives a
  criterion for checking that $L_{\mathsf{out}(\sigma)}$ is tempered
  (assuming by induction that $L_{\mathsf{out}(\sigma_2)}$ is tempered).

\begin{claim}
The $\mathbb H _{a}$-module $L_{\mathsf{out}(\sigma)}$ is tempered if 
\begin{equation}\label{eq:3.10}
w ( I^* )>0,\text{ for every }w \cdot s_\sigma^{-1}\in \Psi(L_{\mathsf{out}(\sigma)}).
\end{equation}
\end{claim}

\begin{proof}
By Proposition \ref{pw-ind}, we deduce that $w \cdot s_\sigma^{-1} \in
\Psi(L_{\mathsf{out}(\sigma)})$ can be written as $w \cdot
s_\sigma^{-1} = v ( s_1 \times s _2 )$ with $s_1 \in \Psi ( M
_{\tau_1})$ and $s_2 \in \Psi ( L _{\mathsf{out} ( \sigma _2 )})$,
where $v \in \mathfrak S _n / ( \mathfrak S _{n_1} \times \mathfrak S
_{n_2} )$ is the minimal coset representative in $\mathfrak S _n$. Let
$\mathtt I _1 := v ( [ 1, 2, \ldots, n _1 ] )$ and $\mathtt I _2 := v
( ( n_1, \ldots, n ] )$. Notice that the ordering of numbers is
preserved by applying $v$. For every $k\ge 1,$ we set
$$\varpi _k ^i := \sum _{1 \le j \le k;~ j \in \mathtt I_i}
\epsilon _j \text{ for } i = 1,2.$$
Notice that $\varpi_k^1+\varpi_k^2=\varpi_k$. We will show
 that
$\left< \varpi _k^i , w \cdot s _{\sigma} ^{-1} \right>\le 1$,
$i=1,2$, which means in particular that $\left< \varpi _k , w \cdot s
  _{\sigma} ^{-1} \right>\le 1$. The condition for $i=2$ follows
immediately from the temperedness hypothesis for $L _{\mathsf{out} ( \sigma _2
  )}$.
It remains to prove the condition for $i=1.$ Set $I:=I_r\sqcup I_{r+1}.$ We can assume that $w ( I ) = \mathtt I _1$ (up to sign change) and either $w ( I ) > 0$ or there exists $l \in I$ so that
$$w ( \{ i \in I \mid i < l \} ) > 0 \text{ and } w ( \{ i \in I \mid i \ge l \} ) < 0.$$
Let $e_1 <e_2< \cdots < e _{n_1}$ be the list of $e$-values of $I$. By
Corollary \ref{wcrit}, we deduce that there exists $0\le
  k_1\le k$ such that
\begin{equation}\label{wtdiv}
\log_q\left< \varpi _k ^1, w \cdot s _{\sigma} ^{-1} \right> = \log_q\left< \varpi _k ^1, v \cdot ( s_1 \times s_2 ) \right> =  \sum _{p = 1} ^{k_1} e _p - \sum ^{n_1} _{q = n_1-k_1+1} e _{q},
\end{equation}
Assumption
(\ref{eq:3.10}) implies that every $e _q > - e_1$ that appears in
the right hand side of
(\ref{wtdiv}), must appear in the form $- e_q$. It follows that the quantity in (\ref{wtdiv})
must be nonpositive, which concludes the proof of the
  claim.
\end{proof}

We return to the proof of Theorem \ref{main:p}. If the
  algorithm {stops} at $r+t$, i.e., $I_{r+t+1}=\emptyset$, then
  Proposition \ref{p:rt1} implies condition (\ref{eq:3.10}), and
  therefore the proof is complete in this case. Assume this is not the
  case. 
We define two smaller marked subpartitions $\tau_2^+$ and
  $\tau^+$ of $\tau=\mathsf{out}(\sigma)$ corresponding to two smaller
  subpartitions $\sigma_2^+$ and $\sigma^+$ of $\sigma$ such that
  $\tau_2^+=\mathsf{out}(\sigma_2^+)$ and
  $\tau^+=\mathsf{out}(\sigma^+).$ It will be sufficient to show that
\begin{enumerate}
\item[{\bf (C1)}] condition (\ref{eq:3.10}) holds for $\tau_2^+$ by
  Proposition \ref{p:rt1} applied to $\sigma_2^+$;
\item[{\bf (C2)}] if (\ref{eq:3.10}) holds for $\tau_2^+$ (and $\sigma_2^+$)
  then it holds for $\tau^+$ (and $\sigma^+$) (using parabolic induction);
\item[{\bf (C3)}] if (\ref{eq:3.10}) holds for $\tau^+$ (and $\sigma^+$)
  then it holds for $\tau$ (and $\sigma$) (using parabolic induction).
\end{enumerate}

Let $\tau ^+ _2$ be the output
of Algorithm \ref{al:ds} steps 3 and 4 applied to $I_1,
I_2,\ldots,I_{r+t}$, and let $\tau ^+_1$ be the set of segments $\{
I_k \} _{k > r + t}$ in Algorithm \ref{al:ds} step 2 which are glued
to $I_{1}, \ldots, I_{r-t}$ in step 3. By examining $e$-values, $\tau
^+_1$ cannot be marked. Here {\bf Claim C1} holds for $\tau_2^+$ by Proposition \ref{p:rt1}.

Define $\tau^+$ to be the
marked subpartition of $\mathsf{out}(\sigma)$, whose support consists
of all segments of the form $I_k$ or $I_k\sqcup I_\ell$ (for some
$\ell$), $0\le k\le r$ produced by step 3 in Algorithm
\ref{al:ds}. The marking in $\tau^+$ is set to be the one inherited
from $\tau.$ Let $\tau ^- = ( \mathbf I ^-,
\delta ^- )$ be the complementary marked subpartition of $\tau ^+$ in
$\mathsf{out}(\sigma)$. 
Let $n^+$ and $n^-$ ($n=n^+ + n^-$) be sizes of $\tau^+$ and $\tau
^-$, respectively. By Algorithm \ref{al:ds}, we deduce that $\tau ^+ =
\mathsf{out} ( \sigma ^+ )$ and $\tau ^- = \mathsf{out} ( \sigma ^- )$
for some $\sigma ^+$, $\sigma ^-$.

\smallskip

To prove {\bf Claim C2}, by successive applications of Corollary
\ref{c:rt1}, we deduce that
$$[\mathrm{Ind}( {} ^{\mathtt t} M_{\tau ^+_1} ^{\mathsf A} \boxtimes L_{\tau ^+_2} ) : L _{\tau^+}] > 0;$$
notice that here $M_{\tau ^+_1} ^{\mathsf A}$ and $L _{\tau^+ _2}$ are standard and irreducible modules of smaller Hecke algebras. We have no $\boxed x \in \tau ^+ _1$ such that $e ( \boxed x ) \in e ( I ^* )$ by (\ref{eest}). Applying Proposition \ref{pw-ind} and Corollary \ref{wcrit}, we conclude that $(\ref{eq:3.10})$ must hold for $\tau ^+$ since it holds for $\tau ^+_2$.

\smallskip

Finally, we verify {\bf Claim C3}.
Let $\tau ^-_0 := ( \mathbf I ^-, 0 )$ be obtained from $\tau^-$ by
removing the markings, and let $\tau_0$ be the marked partition of $n$
obtained as the union of $\tau ^-_0$ and $\tau ^+$. We have $[ M _{\tau _0} : L _{\tau}] = 1$ by Theorem
  \ref{mult}, since $\mathcal O _{\tau_0} \subset \overline{\mathcal O
  _{\tau}}$ is smooth. We would like to realize $M_{\tau_0}$ as an
induced module from $M ^{\mathsf A}\boxtimes M_{\tau^+}$, for some standard $\mathbb H
_{n^-} ^{\mathsf A}$-module $M ^{\mathsf A}$ (defined from $\tau_0^-$). In order for the induction
theorem \ref{indt} to be applicable, we need the following
construction.  
 We divide $\tau ^-_0$ into two (unmarked) marked partitions
 $\tau ^-_1$ and $\tau ^-_2$, where $\tau ^-_1$ consists of segments
 of $\mathbf I ^-$ which are not dominated by a segment in the support
 of $\tau^+$, and $\tau ^-_2$ consists of segments of $\mathbf I ^-$
 which are dominated by a segment in the support of $\tau^+$. We define
  {$M ^{\mathsf A}$ to be the induction of ${} ^{\mathtt t} M ^{\mathsf A} _{\tau
   ^-_1} \boxtimes M ^{\mathsf A} _{\tau ^-_2}$ to $\mathbb H ^{\mathsf A} _{n ^-}$}. 

Then, we can apply
Theorem \ref{indt} again, to deduce that
$$M _{\tau _0} \cong \mathrm{Ind} _{\mathbb H _P} ^{\mathbb H} ( M
^{\mathsf A} \boxtimes M _{\tau^+}), \text{ and so }[\mathrm{Ind} _{\mathbb H _P} ^{\mathbb H} ( M
^{\mathsf A} \boxtimes M _{\tau^+}):L_\tau]>0,$$ where $\mathbb H _P := \mathbb H _{n^-} ^{\mathsf A} \times \mathbb H
_{n^+} \subset \mathbb H := \mathbb H _{n^+ + n ^-}$.
By verifying in addition the hypothesis of Corollary \ref{indmult}, we find
$$[ \mathrm{Ind} _{\mathbb H _P} ^{\mathbb H} ( M ^{\mathsf A} \boxtimes L _{\tau^+} ) : L _{\tau} ] > 0.$$
Applying Theorem \ref{wt-ind}, we conclude that (\ref{eq:3.10})
follows from the corresponding statement for $\tau ^+$ as desired.

\end{proof}

\subsection{A characterization of $\mathsf{ds}(\sigma)$}\label{sec:3.6} We
finish this section with certain combinatorial properties that the
output $\mathsf{out}(\sigma)$ must satisfy. By examining Algorithm \ref{al:ds}, we deduce that $\mathsf{out}(\sigma)$ acquires a
nested component decomposition whenever the sums of $e$-values in 
the hooks of $\sigma$ are not uniformly greater than $0$ or not
uniformly less than $0.$ By Lemma \ref{cr-rigid}, the same is true for
every other $G(s_\sigma)$-orbit which contains $G(s_\sigma)\cdot
\mathsf{out}(\sigma)$ in its closure. Let us refer to this
decomposition here as the ``$\sigma$-hook nested components''
decomposition. 
The hooks of $\sigma$ which contribute to a given
$\sigma$-hook nested component have the sum of $e$-values
uniformly greater than $0$, in which case we call the component
positive, or uniformly less than $0$, in which case we call it
negative. 
As
an application of Theorem \ref{t:main} and Algorithm \ref{al:ds}, we
obtain a combinatorial characterization of
$\mathsf{out}(\sigma)=\mathsf{ds}(\sigma).$ Consider the following
properties for a marked partition $\tau=(\mathbf J',\delta')$:

\begin{enumerate}
\item[(p1)] for every $i\ge 0,$ there are at most $i$ segments in
  $\mathbf J'$ with all $e$-values greater than $m-i+1$.
\item[(n1)] for every $i\ge 0,$ there are at most $i$ segments in
  $\mathbf J'$ with all $e$-values less than $m+i-1$. 
\item[(p2)] for every segment $J'\in\mathbf J',$ we have $\mathtt{av} ( J' )>0$.
\item[(n2)] for every segment $J'\in\mathbf J',$ we have $\mathtt{av}
   ( J' )<0$.
\item[(p3)] for every $J'\in \mathbf J',$ if $m\in e(J')$, then $\delta(J')=1.$
\end{enumerate}

\begin{corollary}  The $G(s_\sigma)$-orbit
  $\mathsf{out}(\sigma)=\mathsf{ds}(\sigma)$ is minimal among all 
{$\tau \in \mathsf{MP} ( a_\sigma )$} admitting the $\sigma$-hook nested
  decomposition and satisfying the properties:
\begin{enumerate}
\item (p1), (p2), (p3) on every positive $\sigma$-hook nested component.
\item (n1), (n2) on every negative $\sigma$-hook nested component.
\end{enumerate}
\end{corollary}

\begin{proof}  
It is sufficient to check the claim when $\sigma$ has only one hook
nested component. The case when $\sigma$ consists of a single hook is
easily verified directly (see Example \ref{ex:onehook}).

{Since the proofs in both cases are similar, we provide a proof only when $(p1)$--$(p3)$ hold (i.e. the case $\sigma$-hook nested component is positive) for the unique $\sigma$-hook nested component.}

Let us assume that the second step of the algorithm runs as $I_1,I_2,\dotsc.$
  There are two situations with respect to $I_1$: either there exists
  $k\ge 2$ such that $I_k$ combines with $I_1$ in the third step of
  the algorithm, or if not, then $I_1$ appears in the support of
  $\mathsf{out}(\sigma)=(\mathbf J,\delta)$ by itself.

$\bullet$ In the first case, we can assume further that the algorithm for $\sigma$ runs as $I_1,I_2,\dots, I_{2M}$
for some $M\ge 1,$ and such that $\#\{I_1,\dots, I_{2M}\}\cap
L^+=M=\#\{I_1,\dots,I_{2M}\}\cap L^-.$ Then $\mathsf{out}(\sigma)$ has
exactly $M$ segments in its support all of the form $I_j\sqcup
I_{j'}$ (see figure \ref{fig:3.20}).

\begin{figure}[h]
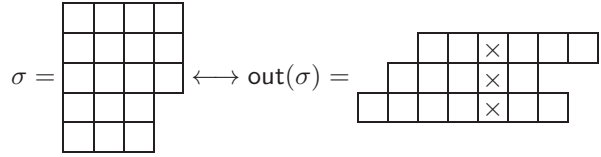
\label{fig:3.20}
$$\sigma=\yng(4,4,4,3,3)
\longleftrightarrow   
\mathsf{out}(\sigma)=\young(::\hfil\hfil{\times}\hfil\hfil\hfil,:{\hfil}{\hfil}{\hfil}\times\hfil,{\hfil}{\hfil}{\hfil}{\hfil}\times{\hfil}\hfil) $$
\caption{Output of Algorithm \ref{al:ds} when  $\sigma=(4,4,4,3,3)$ and $1<m<\frac 32$}
\end{figure}

From the algorithm we see 
that every segment $J\in\mathbf J$ in $\mathsf{out}(\sigma)=(\mathbf J,\delta)$
contains an $e$-value $m$ and $\delta(J)=1$ (by 
$\mathtt{av} ( J ) > 0$). We claim that there is no $\tau$
in the closure of $\mathsf{out}(\sigma)$ which can satisfy the
required conditions. Let $J,J'$ be two segments in $\mathbf J,$ and we assume that $J\prec J'.$ Since
$e(J)\cap e(J')\neq\emptyset$ (because $m$ is in the intersection),
there are only two cases: either $J \lhd J',$ or else $J'\sqsubset J.$
If $J\lhd J',$ by the closure relations of section \ref{sec:1.3}, we
see that $J,J'$ cannot combine to give a smaller orbit. Assume
$J'\sqsubset J.$ Then from step 2 of the algorithm one sees that
necessarily $-\min e ( J ) > \max e ( J' )$. If they combine to give a smaller orbit $\tau,$ then
$\tau$ must contain $J_1,J_1'$ such that  $e(J_1)=\{\min
e(J),1+\min
e(J),\ldots, \max e(J')\},$ and $e(J_1')=\{\min e(J'),1+\min e(J'), \ldots, \max e(J)\}.$  But then condition
(p2) fails for $J_1.$ By a similar argument, one may also see that if
a single 
segment $J\in \mathbf J$ is broken into two pieces such (p1) holds,
then the smaller segment with respect to $\prec$ has to fail (p2).

$\bullet$ In the second case, $I_1$ forms a segment in $\mathbf J$ by itself. If
$\tau=(\mathbf J',\delta')$ is in the closure of $\mathsf{out}(\sigma)$ and satisfies the
required assumptions, then we see that $I_1\in\mathbf J'.$ (This is
because of the conditions (p1,2), the segment $I_1$ cannot be broken into
two pieces to yield such a $\tau$, and it is also clear that if it is
combined with some other segment, the resulting marked partition would
not be in the closure of $\mathsf{out}(\sigma).$  So one can ignore
the segment $I_1$ from consideration. This amounts to analyzing a
smaller partition $\sigma'$ which is obtained from $\sigma$ by
removing the first row and replacing $m$ by $m-1$. Then one proceeds
by induction.
\end{proof}

\section{Applications of the classification}\label{sec:4}

We present some consequences of the
classification to the structure of discrete series. Recall that $\mathbb H_{n,m}$
denotes the affine Hecke algebra of type $C_n$ with parameter
$\vec{q} = ( - 1, q ^{m}, q)$ with $q >1$ and $m \in \mathbb R.$ The generic values of $m$ are
all positive real numbers except half integers.

\subsection{Discrete series and deformations}
One immediate corollary of the algorithm is the classification of discrete series which
contain the $\mathsf{sgn}$ $W$-representation, including for
nongeneric values $m$. (At generic values of $m$, the inequalities in
Theorem \ref{c:sgn} are all strict.)

\begin{definition} Let $\sigma$ be a partition of $n$, which is identified with the Young diagram as in \S \ref{sec:alg} or the RHS of Fig. \ref{fig:fold}. Let $\{k\}$ denote the fractional part of $k$ (which is $0$ or $1/2$ if $k$ is a critical value). The extremities of $\sigma$ at $k$ is the set
  $E(\sigma,k) := E(\sigma,k) ^+ \cup E(\sigma,k) ^-$ defined by the procedure: put in $E(\sigma,k) ^+$ the
  maximal entry in every row above or on the $\{k\}$-diagonal. Also put in $E(\sigma,k) ^-$
  the negative of the minimal entry in every column below or on the
  $-\{k\}$-diagonal. One allows repetitions in this set, if they exist.
\end{definition}

\begin{theorem}\label{c:sgn} Let $\sigma$ be a partition of $n$. The discrete series 
  $\mathsf{ds}(\sigma)$ contains the $\mathsf{sgn}$ $W$-representation if
  and only if one of the following equivalent conditions hold:
\begin{enumerate}
\item $\mathsf{ds}(\sigma)$ parameterizes the open $G(a)$-orbit in $\mathfrak N^a$;
\item The support $\{I_j\}_{j=1}^k$ of $\mathsf{out} ( \sigma )$ satisfies:
\begin{equation}
 \max e(I_1)\ge -\min e(I_1)\ge \max
 e(I_2)\ge -\min e(I_2) \ge \dots.\label{Unest}
\end{equation}
\item We have a sequence $e_1,e_3,e_5,\ldots \in E ( \sigma, m) ^+$ and $e_2,e_4,e_6,\ldots \in E ( \sigma, m )^-$ such that
$$e _1 \ge e _2 \ge e _3 \ge e _4 \ge \cdots \ge 0.$$
\end{enumerate}
\end{theorem}

\begin{proof}
By Theorem \ref{sgnchar}, we know that $\mathsf{mp}(\sigma)=\mathsf{ds}(\sigma)$ {is equivalent to} $\mathsf{sgn} \subset L _{\mathsf{out} ( \sigma )}$ as $W$-submodule. The condition (\ref{Unest}) implies that we have $I _p \sqsubset I _l$ for every $p > l$. Hence, we cannot apply the algorithm of Theorem \ref{AD}. It follows that $\mathsf{ds} ( \sigma )$ defines the dense open orbit when projected to $V_2 ^{(s,q)}$-part. Moreover, we have $\mathtt{av} ( I _l ) > 0$ for every $l$. This implies that the temporary marking $\delta'$ in Algorithm \ref{al:ds} is uniformly marked. Therefore, we deduce that $\mathsf{mp}(\sigma)=\mathsf{ds}(\sigma)$ if and only if 2) holds. Finally, the condition 3) is equivalent to condition 2), since $e _1 = \max e(I_1'), e _2 = \min e(I_2'), e_3 = \max e ( I_3' ),\dots$ in these cases by Algorithm \ref{al:ds}, where $I_1',I_2',\ldots$ denote the output (from $\sigma$) in the second step of the algorithm.
\end{proof}

Recall that the discrete series for $\mathbb H_{n,m}$ are in
one-to-one correspondence $\sigma\leftrightarrow
\mathsf{ds}(\sigma)$ with partitions $\sigma$ of $n.$ To every
$\sigma$ and $m$, one attached a semisimple element $s_{\sigma,m}.$
The following corollary describes the properties of this family of
$\mathbb H_{n,m}$-modules, as $m$ varies in appropriate intervals.

\begin{theorem}\label{c:deform}
Consider $k < m < k + 1/2$ for some $k \in \frac 12\mathbb Z$. Let $\mathsf{ds} ( \sigma )$ be the parameter of discrete series of $\mathbb H
_{n,m}$ attached to $\sigma$. Let $a _m$ be a family of semi-simple
elements attached of $\sigma$ with $\vec{q} = ( - 1, q ^{m},
q)$. Assume that $L _{( a _m, v _{\mathsf{ds} ( \sigma )} )}$ contain $\mathsf{sgn}$ as $\mathbb C [W]$-modules. Then, the modules in the family $\{L _{( a _m, v _{\mathsf{ds} ( \sigma )} )}\}_m$ in
the region $k \le m \le k + 1/2$:

\begin{enumerate}
\item have the same dimension;
\item are all simple tempered modules;
\item are all isomorphic as $W$-representations. 
\end{enumerate}
\end{theorem}

\begin{proof}
Taking into account Corollary \ref{spec} and Theorem \ref{sgnchar}, it
suffices to prove that $v _{\mathsf{ds} ( \sigma )}$ defines an open
dense orbit of $\mathfrak N ^{a _m}$ when $m = k, k + 1/2$. By the
description of the orbit structure of $\mathfrak N ^a$ in \cite{K}
1.17, we deduce that we obtain no new orbits by the specialization
process (inside $V _2$). It follows that the condition (\ref{Unest}) with $>$ replaced by $\ge$
is already enough to guarantee that $\mathcal O _{\mathsf{ds} ( \sigma
  )} ^{a _m} \subset \mathfrak N ^{a_m}$ is dense open for every $k
\le m \le k + 1/2$. 
\end{proof}

Notice that when $k<m<k+1/2$ with $k$ critical value, every $L _{( a _m, v _{\mathsf{ds} (
    \sigma )} )}$ is in fact a discrete series. But at the endpoints
of the interval, $m\in \{k,k+1/2\}$ they could be just tempered. On the other hand, Corollary \ref{specds} effectively says that if $L _{( a _m, v _{\mathsf{ds} (
    \sigma )} )}$ is a discrete series in the interval $k<m<k+1/2,$
but also in the interval $k-1/2<m<k,$ then it is a discrete series at
$m=k.$ This gives a combinatorial condition on $\sigma$, viewed as a
tableau for $m$, as
follows. The idea is due to \cite{Sl} \S 4.5.

\begin{corollary} Assume $k$ is a critical value. If the family  $\{L _{( a _m, v _{\mathsf{ds} (
    \sigma )} )}\}$ consists of discrete series in the interval
  $k-1/2<m<k+1/2$, then the set $E(\sigma,k)$ does not have repetitions.
\end{corollary}

\begin{proof}
In combinatorial terms, the condition that $\{L _{( a _m, v _{\mathsf{ds} (
    \sigma )} )}\}$ consists of discrete series in the interval
  $k-1/2<m<k+1/2$ means that the output of Algorithm \ref{al:ds}
is the same for $\sigma$ when $k-1/2<m<k$ or $k<m<k+1/2$. This is
equivalent to the fact that step 2 of the algorithm is the same in
these two intervals, which implies that step 2 of the algorithm is
well-defined at $m=k$ as well. From this, it is easy to see that
$E(\sigma,k)$ must not allow repetitions.
\end{proof}

\subsection{Tempered modules in generic parameters}
Let us assume that $m$ is generic. Let $\mathbb H ^{\mathsf A} _{n,m}$ be the image of of $\mathbb H _{n} ^{\mathsf A}$ under the specialization map $\mathbb H _{n} \to \mathbb H _{n,m}$. This is an affine Hecke algebra for $GL(n).$

\begin{theorem}[\cite{KL2}, \cite{Z}]\label{ZT}
The set of tempered modules with positive real central character of $\mathbb H
^{\mathsf A} _{n,m}$ is in one-to-one correspondence with the set of
partitions of $n$. Fix a
partition $\sigma = ( \sigma _1, \sigma _2, \ldots,\sigma_\ell )$ of
$n$. The corresponding tempered module $L_{( a,v_\tau )} ^{\mathsf A}$ is obtained by the following procedure:
\begin{enumerate}
\item Form a sequence $\mathbf I = \{ I _k \}$ of subsets of $[1,n]$
  by setting $I_1=[1,\sigma_1],\dotsc,
  I_k=[\sigma_1+\dotsb+\sigma_{k-1}+1,\sigma_1+\dots+\sigma_k]\dots$. 
\item Set the trivial labeling $\delta \equiv 0$ and form $\tau := ( \mathbf I, \delta )$;
\item Form a semi-simple element $s \in T$ such that $\tau$ is adapted to $a=(s, \vec{q})$ and
$$\{\left< \epsilon _i, s \right> ; i \in I _k \} = \{\left< \epsilon _i, s \right> ^{-1} ; i \in I _k \} \text{ for each }k.$$
\end{enumerate}
\end{theorem}

More precisely, the set of tempered modules with positive real central
character for type $A_{n-1}$ are in one to one correspondence with  nilpotent
adjoint orbits of type $A_{n-1}$, and this is in turn are parameterized by
partitions of $n$. If $\sigma$ is a partition as above, then one forms
the (block upper triangular) parabolic subgroup $P$ with Levi subgroup
$GL(\sigma_1)\times\dotsb\times GL(\sigma_\ell).$ Let $\mathbb
H_{P}=\mathbb H_{\sigma_1,m}^\mathsf{A}\times\dotsb\times \mathbb H_{\sigma_\ell,m}^\mathsf{A}$
be the Hecke subalgebra of $\mathbb H_{n,m}^\mathsf{A}$ corresponding
to $P$. Let $\mathsf{St}_{\sigma_j}$ denote the Steinberg $\mathbb
H_{\sigma_j,m}^\mathsf{A}$-module. The induced module $\mathsf{Ind}_{\mathbb H_P}^{\mathbb
  H_{n,m}^{\mathsf{A}}}(\mathsf{St}_{\sigma_1}\boxtimes\dotsb\boxtimes
  \mathsf{St}_{\sigma_\ell})$ is irreducible and tempered, and this
  is {$L_{( a,v_\tau)} ^{\mathsf A}$}, in the notation of \ref{ZT}. The element $s$
  corresponds to the middle element of the nilpotent orbit
  parameterized by $\sigma.$ Explicitly, we have $s=\exp
  (\frac{\sigma_1-1}2,\ldots, -\frac{\sigma_1-1}2,\ldots,
  \frac{\sigma_\ell-1}2,\ldots, -\frac{\sigma_\ell-1}2).$

When we have $S_0 = \Pi - \{ \alpha _{n_1} \}$, then we have $L _{S_0}
\cong \mathop{GL} (n_1) \times \mathop{Sp} ( 2 n_2 )$. We have
$\mathfrak N \cap \mathbb V _{S_0} \subset \mathfrak{gl} (n_1) \oplus
\mathbb V _{(2)}$ (as $L _{S_0}$-varieties), where $\mathbb V _{(2)}$
is the exotic representation of $\mathop{Sp} ( 2 n_2 )$. 

\begin{theorem}\label{indAD}
Set $S_0 = \Pi - \{ \alpha _{n_1} \}$. Let $a = ( s, \vec{q} ) \in \mathcal G _0$ be given. We define $s_1, s_2$ to be the projections of $s$ to $\mathop{GL} (n_1)$ and $\mathop{Sp} ( 2 n_2 )$, respectively. Fix $X \in \mathfrak N ^{a}$ and a decomposition $X = X _1 \oplus X _2 \in \mathfrak{gl} (n_1) \oplus \mathbb V _{(2)}$. We assume
\begin{enumerate}
\item $L _{(s_1, q, X_1)} ^{\mathsf A}$ and $L _{(s_2, \vec{q}, X_2)}$ define irreducible modules of $\mathbb H _{n_1,m} ^{\mathsf A}$ and $\mathbb H_{n_2,m}$, respectively;
\item We have $\{ \left< \epsilon _i, s \right> ; 1 \le i \le n _1 \} \subset q ^{\frac{1}{2}\mathbb Z}$ and $\{ \left< \epsilon _i, s \right> ; n _1 < i \le n \} \subset q _1 q ^{\mathbb Z}$.
\end{enumerate}
Then, we have an isomorphism
\begin{eqnarray}
\mathsf{Ind} ^{\mathbb H_{n,m}} _{\mathbb H ^{S _0}_{n,m}} ( L _{(s_1, q, X_1)} ^{\mathsf A} \boxtimes L _{(s_2, \vec{q}, X_2)} ) \cong L _{(a,X)} \label{irrind}
\end{eqnarray}
as $\mathbb H_{n,m}$-modules.
\end{theorem}

\begin{proof}
We have $q ^{\frac{1}{2}\mathbb Z} \cap q _1 q ^{\mathbb Z} = \emptyset$. If follows that $( \mathbb V ^{S_0} ) ^a = \{ 0 \}$, hence the induction theorem is applicable. Since we have $L _{S_0} ( s ) = G ( s )$ and $\mathfrak N ^a = \mathfrak N ^a \cap ( \mathfrak{gl} (n_1) \oplus \mathbb V _{(2)} )$, it follows that the isomorphism classes of irreducible $\mathbb H _{a}$-modules and irreducible $\mathbb H _a ^{S_0}$-modules are in one-to-one correspondence through the identification of parameters. For a point $X$ in the open dense $G(s)$-orbit of $\mathfrak N^a$, Theorem \ref{mult} implies that both $M _{(a,X)} ^{S_0}$ and $M_{(a,X)}$ are irreducible modules of $\mathbb H ^{S_0} _a$ and $\mathbb H _a$, respectively. Hence, the assertion holds in this case. We prove the assertion by induction on the closure relation of orbits. Let $\mathcal O \subset \mathfrak N ^a$ be a $G (s)$-orbit. Assume that (\ref{irrind}) holds for all $\mathbb H _a ^{S_0}$-module such that the orbit closure of the corresponding $G (s)$-orbit contains $\mathcal O$. Fix $X \in \mathcal O$. By the induction theorem, we have
$$\mathsf{Ind} ^{\mathbb H_{n,m}} _{\mathbb H ^{S _0}_{n,m}} M _{(a, X)} ^{S_0} \cong M _{(a,X)}.$$
Moreover, Theorem \ref{mult} asserts that the multiplicities of each
irreducible module inside $M _{(a, X)} ^{S_0}$ and $M _{(a,X)}$ (as
$\mathbb H ^{S_0}_{n,m}$ and $\mathbb H_{n,m}$-modules, respectively) are the same,
under the correspondence between the irreducibles of $\mathbb H^{S_0}_{n,m}$ and
$\mathbb H_{n,m}$. Hence we deduce
$$\mathsf{Ind} ^{\mathbb H_{n,m}} _{\mathbb H ^{S _0}_{n,m}} L _{(a, X)} ^{S_0} \cong L _{(a,X)},$$
where $L _{(a,X)} ^{S_0}$ is the unique $\mathbb H^{S_0}_{n,m}$-module corresponding to $\mathcal O$. This is nothing but (\ref{irrind}). Hence, the induction proceeds and we conclude the result.
\end{proof}

\begin{theorem}\label{indADt}
Set $S_0 = \Pi - \{ \alpha _{n_1} \}$. Let $a = ( s, \vec{q} ) \in \mathcal G _0$ be given. We define $s_1, s_2$ to be the projections of $s$ to $\mathop{GL} (n_1)$ and $\mathop{Sp} ( 2 n_2 )$, respectively. Fix $X \in \mathfrak N ^{a}$ and a decomposition $X = X _1 \oplus X _2 \in \mathfrak{gl} (n_1) \oplus \mathbb V _{(2)}$. We assume
\begin{enumerate}
\item $L _{(s_1, q, X_1)} ^{\mathsf A}$ is a tempered module of $\mathbb H _{n_1,m} ^{\mathsf A}$;
\item $L _{(s_2, \vec{q}, X_2)}$ is a discrete series of $\mathbb H_{n_2,m}$.
\end{enumerate}
Then, $L _{(a,X)}$ is a tempered $\mathbb H_{n,m}$-module.
\end{theorem}

\begin{proof}
The proof of the induction theorem (see \cite{K} \S 7 or \cite{KL2} \S 7) claims that we have an isomorphism
\begin{eqnarray}
H _{\bullet} ( \mu ^{-1} ( X ) ^a ) \cong \bigoplus _{w} H _{\bullet} ( \mu ^{-1} ( X ) ^a \cap P _{S_0} \dot{w} ^{-1} B / B ) \hskip 2mm \text{as vector spaces},\label{split}
\end{eqnarray}
where $w \in W / ( \mathfrak S _{n_1} \times W _{n_2} )$
denotes its minimal length representative in $W$. Here we have $( P
_{S_0} \dot{w} ^{-1} B/ B ) ^s \cong G ( s ) / B$, which implies that
(\ref{split}) is in fact a direct sum decomposition as $\mathbb H
^{S_0}_{n,m}$-modules (up to semi-simplification). It follows that the weight $u
= w \cdot s ^{-1} \in \Psi ( L _{(a,X)} )$ is written as $v \cdot ( (
w _1 \cdot s _1 ^{-1} ) \times s _2' )$, where $v \in \mathfrak S _n /
( \mathfrak S _{n_1} \times \mathfrak S _{n_2} )$ is a minimal length
representative, $w _1 \in W _{n_1}$, and $s _2' \in \Psi ( L _{(s_2,
  \vec{q}, X _2)})$. {Let us take a subdivision $\mathbf I
  =\{I_j\}_j$ of $[1,n_1]$ (i.e. $\sqcup _j I _j = [1,n_1]$) so that
  $I _j = \{i _1^j, \ldots, i _k^j\}$ satisfies $i _1^j < i _2^j <
  \cdots < i _k^j$, $\langle  
  \epsilon _{i _l^j}, s \rangle  = q \langle \epsilon _{i _{l+1}^j}, s
\rangle$ and $\langle \epsilon _{i _l^j}, s \rangle  = \langle
\epsilon _{i _{k+1-l}^j}, s \rangle ^{-1}$ holds for each $1 \le l \le
k$. Let $s_*$ be a semi-simple element such that $G ( s_* ) = L
_{S_0}$. Then, we have necessarily $s_* ( X ) = X$. By rearranging
$\mathbf I$ and taking $L_{S_0}$-conjugate if necessary, we can assume
$X_1 = \sum _j v _{I_j}$. Since $\mathsf{Stab} _{\mathop{GL} (n _1)}
X_1$ contains a torus of rank $\# \mathbf I$, we need $v _{I _j} \in
{} ^{w} \mathbb V ( a )$ (for all $j$) in order for (\ref{split}) to
be non-empty. This implies that  
$$\left< \varpi^j _{p}, s ^{-1} \right> \le 1, \text{ where } \varpi^j _p = \sum _{|w _1 ( i ^j _l )| < p} \epsilon _{|w _1 ( i ^j _l )|} \text{ for each }p.$$
By the minimality of $v$, we deduce
$$\left< \varpi'' _{p}, s ^{-1} \right> = \left< \varpi _{p'}, s _2' \right> \le 1, \text{ for each } p,$$
where $\varpi'' _p = \varpi _p - \sum_{j} \varpi^j _p$ and $p' = p - \sum _j \# \{ l \in I _j \mid |w _1 ( i ^j _l )| < p \}$. This implies that
$$\left< \varpi _i, u \right> \le 1 \text{ for every } u = w^{-1} \cdot s \in \Psi ( L _{(a,X)}) \text{ and every } 1 \le i \le n,$$}
which implies that $L _{(a,X)}$ is tempered as desired.
\end{proof}

\subsection{Linear independence of $R ( T )$-characters}

\begin{construction}\label{tau}
Let $a \in \mathcal T _0$ and $\tau = ( \mathbf I, \delta) \in \mathsf{MP} (a)$ be given.
We divide $\mathbf I$ into {four sets $D _+ ^1, D _+ ^2, D
_- ^1, D _- ^2$} as follows: 
\begin{itemize}
\item If $\max\underline I < q_1$, we put $I \in \mathbf I$ into $D _+ ^2$;
\item If $\min\underline I > q_1$, we put $I \in \mathbf I$ into $D _- ^2$;
\end{itemize}
Notice that all segments $I$ in $D_+^2\cup D_-^2$ are unmarked, since
$q_1\notin \underline I.$ Now we consider only segments in $\mathbf I
\setminus ( D _+ ^2 \cup D 
_- ^2 )$. 
\begin{itemize}
\item If there exists some $I ^{\prime}$ such that $\delta ( I
  ^{\prime} ) = \{ 0, 1\} $ (i.e., $I^\prime$ is marked) and $I \lhd I
  ^{\prime}$, then we put $I$ into $D _+ ^1$; 
\item If we have $\delta (I) = \{ 0 \}$ and there exists no $I
  ^{\prime}$ such that $\delta ( I ^{\prime} ) = \{ 0, 1\}$ and $I
  \lhd I ^{\prime}$, then we put $I$ into $D _- ^1$
\end{itemize}
We denote $D _{+} := ( D _+ ^1 \cup D _+ ^2 )$ and $D _- :=
( D_- ^1 \cup D _- ^2 )$. 
\end{construction}
Notice that $D _+ \cup D _-$ exhausts $\mathbf I$. One sees immediately that in this construction, {$I \in \mathbf I$ satisfies $\widetilde{\delta} ( I ) = \{ 0,1 \}$ if and only if $I \in D_+$ and $q_1 \in \underline{I}$.} We change the marking of $\tau$ so that every
$I \in D _+$, with $q_1\in \underline I$, is
marked. By Proposition \ref{p:MP2}, this procedure does not change the
$G ( s )$-orbit of $v _{\tau}$.

The following proposition is a criterion for finding some special weight of each simple $\mathbb H_a$-modules. The notation $w(j)$
refers to the usual action of $W _n$ by permutations and sign changes
on $[-n,n]$.

\begin{proposition}\label{t:charfake} Keep the setting of Construction \ref{tau}. Assume that we have $\left<\epsilon _i ,s\right> >
  \left< \epsilon _j,s \right>$ for every $i < j$. Assume that  $w \in W$  satisfies the following conditions:
\begin{itemize}
\item Assume $I \in D _+$. Then, we have $w ( j ) > 0$ for all $j \in I$. Moreover, we have $w ( i ) < w ( j )$ for each $i,j \in I$ such that $\left< \epsilon _i,s \right> > \left< \epsilon _j,s \right>$;
\item Assume $I \in D _-$. Then, we have $w ( j ) < 0$ for all $j \in I$. Moreover, we have $w ( i ) < w ( j )$ for each $i,j \in I$ such that $\left< \epsilon _i,s \right> > \left< \epsilon _j,s \right>$;
\item {Assume $I, I ^{\prime} \in  D _+$ or $I, I ^{\prime} \in  D _-$. If we have $I \prec I ^{\prime}$, then we have}
\begin{equation}
w ( j ) < w ( j ^{\prime} ) \text{ for every } (j, j ^{\prime}) \in I
\times I ^{\prime};
\end{equation}
\item If $I, I ^{\prime} \in \mathbf I$ and $\min \underline{I} = \min \underline{I ^{\prime}}$, then we have either
\begin{align}
& w ( j ) > w ( j ^{\prime} ) \text{ for every } (j, j ^{\prime}) \in I
\times I ^{\prime} \text{, or }\\
& w ( j ) < w ( j ^{\prime} ) \text{ for every } (j, j ^{\prime}) \in I
\times I ^{\prime} 
\end{align}
\end{itemize}
Then $\mathcal O _{\tau}$ meets ${} ^{w} \mathbb V (a)$ densely. In particular, we have $w\cdot s^{-1} \in \Psi ( L_{(a,v _{\tau})} ).$ 
\end{proposition}

\begin{proof}
The first two conditions imply $\dot{w} v _{I} \in \mathbb V
^+$, for all $I\in \mathbf I$, and $\dot{w} v _{\epsilon _i} \in \mathbb V ^+$ (if $\delta ( i ) > 0$). Therefore, we deduce $v _{\tau} \in {} ^{w} \mathbb V ^+$, which implies $v _{\tau} \in {} ^{w} \mathbb V (a)$.

For each ordered pair $(l,r) \in \mathbb Z^2$, we define
$${\mathfrak p _{\tau} ^{l,r} := \bigoplus _{i \in I _l, j \in I _m;
  (\star)} (\mathfrak g ( s )\cap \mathfrak g[ \epsilon _i - \epsilon _j ])},$$
where $(\star)$ denotes the condition $\epsilon _{w (i)} - \epsilon
_{w (j)} \in R ^+$, and $\mathfrak g[\epsilon_i-\epsilon_j]$ are the
weight spaces. The condition $(\star)$ is also rephrased as:
\begin{itemize}
\item[$(\star) _1$] If $w ( i ) w ( j ) > 0$, then we have $w ( i ) < w ( j )$;
\item[$(\star) _2$] If $w ( i ) < 0$, then we have $w ( j ) < 0$.
\end{itemize}
It is straight-forward to see that $\mathfrak p _{\tau} ^{l,r}$ is an abelian subalgebra of $\mathfrak g$.

Since $\{ I _r \} _r$ exhaust $[1,n]$,  condition $(\star)$ implies
$$\mathfrak p _{\tau} := \mathfrak t \oplus \bigoplus _{l, r}
\mathfrak p _{\tau} ^{l,r} = ( \dot{w} ^{-1} \mathfrak b ) \cap
\mathfrak g ( s ).$$ 

Hence, the Lie algebra $\mathfrak p _{\tau}$ preserves ${} ^{w}
\mathbb V ^+$. Since $\mathfrak p _{\tau} \subset \mathfrak g ( s )$,
it preserves $\mathbb V ^a$. Thus, $\mathfrak p _{\tau}$ acts on ${}
^{w} \mathbb V (a)$. Moreover, the connected algebraic subgroup $P _{\tau} \subset G ( s )$ with
$\mathrm{Lie} P _{\tau} = \mathfrak p _{\tau}$ acts on ${} ^{w} \mathbb V (a)$. 
We wish to prove that $P_\tau v_\tau$ is dense in ${} ^{w} \mathbb V
(a)$. We will be able to deduce this from the following claim, which
is proved by computations.

\begin{claim} $\mathfrak p_\tau v_\tau={}
^{w} \mathbb V (a)$. 
\end{claim}

\begin{proof} 
Since $\mathfrak p _{\tau} ^{l,r}$ is a direct sum of $T$-weight spaces, we deduce that $( \mathfrak t \oplus \mathfrak p _{\tau} ^{l,r} )$ is again a Lie subalgebra of $\mathfrak g ( s )$. We set $\mathfrak t ^{r} := \bigoplus _{i \in I _r} \mathbb C \epsilon _i$, where $\epsilon _i \in \mathfrak t ^*$ is identified with the dual basis $\epsilon _i \in \mathfrak t$ by the pairing $( \epsilon _i, \epsilon _j ) = \delta _{i,j}$. We have
$$\mathfrak t ^r v _{\tau} = \mathfrak t  ^r ( v _{I _r} + \sum _{i
  \in I _r} \delta ( i ) v _{\epsilon _i}  ) = \bigoplus _{i, j \in I
  _r ; \left< \epsilon _i,s \right> = q \left< \epsilon _j,s
  \right>} \mathbb C v _{\epsilon _i - \epsilon _j} \oplus \bigoplus
_{i \in I _r; \delta ( i ) = 1} \mathbb C v _{\epsilon _i}$$ 
by a simple calculation. (Here we used the fact the the weights appearing in the RHS are linearly independent.)
 
By the first two conditions on $w$, the signs of the entries in
$w(I_l)$ and $w(I_m)$ are constant on each segment. We calculate
$\mathfrak p _{\tau} ^{l,r} v _{\tau}$ in each of the four possible
cases of signs.

\item {\bf Case 1)} ($w ( I _l ), w ( I _r ) > 0$) This means
  $I_l,I_r\in D_+.$ We have either $0 < w ( I _r ) < w ( I _l )$ or $0
  < w ( I _l ) < w ( I _r )$. If we have $0 < w ( I _r ) < w ( I _l
  )$, then we have $\epsilon _{i} - \epsilon _j \not\in w ^{-1} R ^+$
  for every $i \in I _l, j \in I _r$. Therefore, $\mathfrak p _{\tau}
  ^{i,j} = \{ 0 \}$ in this case. 

Now we assume $0 < w ( I _l ) < w ( I _r )$. We have $\min \underline{I _r} \ge \min \underline{I _l}$ by assumption. We have $\epsilon _{i} - \epsilon _j \in \Psi ( \mathfrak p _{\tau} ^{l,r} )$ if and only if $i \in I _l, j \in I _r$ and $\left< \epsilon _i - \epsilon _j,s \right> = 1$. By the definition of segments, we deduce that
$$\mathfrak p _{\tau} ^{l,r} = \bigoplus _{b \in \underline{I _l} \cap \underline{I _r}} \mathfrak g [ \epsilon _{i _b} - \epsilon _{j _b} ],$$
where $i _b \in I _l, j _b \in I _r$ satisfy $\left< \epsilon _{i _b},s \right> = b = \left< \epsilon _{j _b},s \right>$. By an explicit computation, we have
$$\mathfrak g [ \epsilon _{i _b} - \epsilon _{j _b} ] v _{\tau} = \mathbb C ( v _{\epsilon _{i _{q b}} - \epsilon _{j _b}} + v _{\epsilon _{i _b} - \epsilon _{j _{q ^{-1} b}}} ).$$
Let $b ^-$ be the minimal element of $\underline{I _l} \cap
\underline{I _r}$. Then, the number $j _{q _2 ^{-1} b ^-}$ does not exist. It follows that
$$\mathfrak p ^{l,r}_\tau v _{\tau} = \sum _{b \in \underline{I _l} \cap \underline{I _r}} \mathfrak g [ \epsilon _{i _b} - \epsilon _{j _b} ] v _{2,\tau} = \sum _{b \in \underline{I _l} \cap \underline{I _r}} \mathbb V [ \epsilon _{i _b} - \epsilon _{j _{q ^{-1} b}} ],$$
where $v _{2,\tau}$ be the $V_2$-part of $v_{\tau}$.

Here $I _l$ and $I _r$ are marked if $q_1 \in \underline{I_l}$ and $q_1 \in \underline{I_r}$, respectively. Hence, we conclude that
$$( \mathfrak t ^l \oplus \mathfrak t ^r \oplus \mathfrak p ^{l,r}_\tau ) v _{\tau} = \sum _{b \in \underline{I _l} \cap \underline{I _r}} \mathfrak g [ \epsilon _{i _b} - \epsilon _{j _b} ] v _{\tau} = \mathfrak t ^{l} v _{\tau} \oplus \mathfrak t ^{r} v _{\tau} \oplus V,$$
where
$$V = \begin{cases} \sum _{b \in \underline{I _l} \cap \underline{I _r}} \mathbb V [ \epsilon _{i _b} - \epsilon _{j _{q _2 ^{-1} b}} ] & (w ( I _r ) > w ( I _l )) \\ \{0\} & (w ( I _r ) < w ( I _l )) \end{cases}.$$

\item {\bf Case 2)} ($w ( I _l ), w ( I _r ) < 0$) This means   $I_l,I_r\in D_-$. This case is exactly the same as {\bf Case 1} if we uniformly take the inverse of every weight. Therefore, we conclude that
$$\mathfrak p _{\tau} ^{l,r} = \begin{cases}\bigoplus _{b \in \underline{I _l} \cap \underline{I _r}} \mathfrak g [ \epsilon _{i _b} - \epsilon _{j _b} ] & (w ( I _l ) < w ( I _r))\\ 0 & (w ( I _l ) > w ( I _r ))\end{cases},$$
and
$$( \mathfrak t ^l \oplus \mathfrak t ^r \oplus \mathfrak p ^{l,r}_\tau ) v _{\tau} = \sum _{b \in \underline{I _l} \cap \underline{I _r}} \mathfrak g [ \epsilon _{i _b} - \epsilon _{j _b} ] v _{\tau} = \mathfrak t ^{l} v _{\tau} \oplus \mathfrak t ^{r} v _{\tau} \oplus V,$$
where
$$V = \begin{cases} \sum _{b \in \underline{I _l} \cap \underline{I _r}} \mathbb V [ \epsilon _{i _b} - \epsilon _{j _{q _2 ^{-1} b}} ] & (w ( I _r ) < w ( I _l )) \\ \{0\} & (w ( I _r ) > w ( I _l )) \end{cases}.$$

\item {\bf Case 3)} ($w ( I _l ) < 0, w ( I _r ) > 0$) This means 
  $I_l\in D_-$, $I_r\in D_+.$ We have $\epsilon _i - \epsilon _j \not\in w ^{-1} R ^+$ when $i \in I _l$ and $j \in I _r$. It follows that {$\mathfrak p _{\tau} ^{l,r} = \{ 0 \}$}. Therefore we have
$$( \mathfrak t ^l \oplus \mathfrak t ^r \oplus \mathfrak p ^{l,r}_\tau ) v _{\tau} = \mathfrak t ^{l} v _{\tau} \oplus \mathfrak t ^{r} v _{\tau}.$$

\item {\bf Case 4)} ($w ( I _l ) > 0, w ( I _r ) < 0$) This means
  $I_l\in D_+, $ $I_r\in D_-.$ We have $\epsilon _i - \epsilon _j \in w ^{-1} R ^+$ when $i \in I _l$ and $j \in I _r$. By a similar argument as in {\bf Case 1}, we deduce that
$$\mathfrak p _{\tau} ^{l,r} = \bigoplus _{b \in \underline{I _l} \cap \underline{I _r}} \mathfrak g [ \epsilon _{i _b} - \epsilon _{j _b} ],$$
where $i _b \in I _l, j _b \in I _r$ satisfies $\left< \epsilon _{i _b},s \right> = b = \left< \epsilon _{j _b},s \right>$. By assumption, we have $I _l \succ I _r$ only if $\underline{I _l} \subsetneq \underline{I _r}$. If $\min \underline{I _l} \le \min \underline{I _r}$, then we have $j _{q ^{-1} b ^-} = \emptyset$ for $b ^- = \min ( \underline{I _l} \cap \underline{I _r} )$. If $\min \underline{I _l} \ge \min \underline{I _r}$, then we have $i _{q b ^+} = \emptyset$ for $b ^+ = \max ( \underline{I _l} \cap \underline{I _r} )$.

The segment $I _l$ is marked if $q_1 \in \underline{I_l}$, while $I _r$ is never so. In particular, the vector $\sum _{i \in I _l \cup I _r} \delta ( i ) v _{\epsilon _i}$ is annihilated by $\mathfrak p _{\tau} ^{l,r}$.

Therefore, by a similar argument as in {\bf Case 1}, we have
$$( \mathfrak t ^l \oplus \mathfrak t ^r \oplus \mathfrak p ^{l,r}_\tau ) v _{\tau} = \mathfrak t ^{l} v _{\tau} \oplus \mathfrak t ^{r} v _{\tau} \oplus V,$$
where
$$V = \begin{cases} \sum _{b \in \underline{I _l} \cap \underline{I _r}} \mathbb V [ \epsilon _{i _b} - \epsilon _{j _{q _2 ^{-1} b}} ] & (\min \underline{I _l} {\ge} \min \underline{I _r}) \\ \sum _{b \in \underline{I _l} \cap \underline{I _r}} \mathbb V [ \epsilon _{i _{q b}} - \epsilon _{j _b} ] & (\min \underline{I _l} {\le} \min \underline{I _r}) \end{cases}.$$

\item We can rephrase the conclusion of the above case-by-case calculations as follows:
\begin{itemize}
\item $\mathfrak t ^l v _{\tau}$ is a sum of $T$-weight spaces of ${} ^{w} \mathbb V (a)$ of weight $\epsilon _i$ or $\epsilon _{i} - \epsilon _{j}$ such that $i,j \in I _l$;
\item $\mathfrak p _{\tau} ^{l,r} v _{\tau}$ is a sum of $T$-weight spaces of ${} ^{w} \mathbb V (a)$ of weight $\epsilon _{i} - \epsilon _{j}$ such that $i \in I _l$ and $j \in I _r$.
\end{itemize}
From this, we deduce that
$$\mathfrak p _{\tau} v _{\tau} = \mathfrak t \oplus \sum _{l,r}
\mathfrak p ^{l,r} v _{\tau} = {} ^{w} \mathbb V (a).$$
\end{proof}

We have a natural identification $\mathfrak p _{\tau} v _{\tau} = T _{v _{\tau}} ( P _{\tau} v _{\tau} )$ (the RHS must be read as the tangent space of $P _{\tau} v _{\tau}$ at $v _{\tau}$). We deduce that
$$\dim P _{\tau} v _{\tau} = \dim {} ^{w} \mathbb V (a).$$

Since $P _{\tau} v _{\tau} \subset {} ^{w} \mathbb V (a)$, this forces $\overline{P _{\tau} v _{\tau}} = {} ^{w} \mathbb V (a)$, which implies the result.
\end{proof}

\begin{definition}Let $L$ be an irreducible $\mathbb H$-module. We define the $R ( T )$-character of $L$ as a formal linear combination
\begin{equation}
\mathsf{ch} L := \sum _{s \in \Psi ( L )} ( \dim R ( T ) _s \otimes _{R ( T )} L ) \left< s \right>,
\end{equation}
where $\left<s\right>$ is a formal symbol for each $s \in T$ and $R ( T )_s$ is the localization of $R ( T )$ with respect to the kernel of the evaluation map at $s$.
\end{definition}

\begin{corollary}
Let $\mathfrak M _{m}$ be the set of isomorphism classes of irreducible $\mathbb H _{n,m}$-modules. Assume that $m$ is generic. Then, the set $\{\mathsf{ch} L \mid L \in \mathfrak M _{m} \}$ is linearly independent.
\end{corollary}

\begin{proof}
Since $\Psi ( L )$ is contained in the $W$-conjugacy class of a
central character of $L$ for each $L \in \mathfrak M _{m}$, it follows
that we can argue by fixing one central character $a = ( s, \vec{q}
)$. Let $L, L'$ be irreducible $\mathbb H _a$-modules with the
corresponding $G(a)$-orbits $\mathcal O, \mathcal O'$, respectively
(via the eDL correspondence). By Proposition \ref{t:charfake}, there
exists $w _L \in W$ such that ${}^{w _L} \mathbb V ^a \cap \mathcal O$
is an open dense subset of ${}^{w _L} \mathbb V ^a$. By Proposition
\ref{p:2.1}, we deduce that $w _L \cdot s ^{-1} \in \Psi ( L' )$ only
if $\mathcal O' \subset \overline{\mathcal O}$. We introduce a linear
order $\succ$ on the set of irreducible $\mathbb H _a$-modules such
that $L \succ L'$ if $\mathcal O' \subset \overline{\mathcal
  O}$. Thanks to Corollary \ref{c:ftemp}, we conclude that $\{ w _L
\cdot s ^{-1} \mid L \in \mathfrak M _{m} \}$ appears in $\{ \Psi ( L
) \mid L \in \mathfrak M _{m} \}$ triangularly with respect to
$\succ$, which implies the result. 
\end{proof}

\subsection{Cuspidal local systems in Spin groups}\label{sec:cuspidal} In this section, we
explain the constructions and algorithms of Lusztig and Slooten and
the relation with our setting.

Let $\overline{\mathbb H}_{n,m}$ be the affine graded Hecke algebra
as in Definition \ref{gradedC}, with
parameters normalized as:
\begin{align}
&\xymatrix@R=4pt{
& &*{1}  &*{1}  & *{1} &  & *{1} & *{1} & *{m}\\
& & *{\circ} \ar@{-}[r] & *{\circ} \ar@{-}[r] & *{\circ} \ar@{-}[rr]|{\cdots\cdots} &  & *{\circ} \ar@{-}[r] & *{\circ} \ar@{=}[r]|{<} & *{\circ}
}, \\\notag
&\text{ where } 
4m\equiv 1 \text{ or } 3\text{ mod } 4.
\end{align}
Define $X_\ell$ to be the set of nilpotent orbits in $so(\ell)$
parameterized by partitions containing odd parts with multiplicity
one, and even parts with even multiplicity. For every nilpotent orbit
$\mathcal O\subset so(\ell)$ given by the partition 
$(a_1,\dots, a_s)$ define the defect of $\mathcal O$
\begin{equation}
d(\mathcal O):=\sum_{i=1}^s d(a_i), \text{ where }
d(a_i)=\left\{\begin{matrix} 1, &\text{ if } a_i\equiv 1 \text{ mod
    }4\\
                           0, &\text{ if } a_i\equiv 0,2 \text{ mod
    }4\\
                           -1, &\text{ if } a_i\equiv 3 \text{ mod
    }4.\end{matrix}\right.
\end{equation}
For every $d\in \mathbb Z,$ set $X_{\ell,d}$ to be the set of elements
in $X_\ell$ of defect equal to $d.$ Then one has
\begin{equation}
X_\ell=\bigcup_{d\in\mathbb Z, 4|(\ell-d)} X_{\ell,d}.
\end{equation}
The generalized Springer correspondence (\cite{L0}) for the cuspidal
local systems in $Spin(\ell)$ which do not factor through $SO(\ell)$
takes the following combinatorial form.

\begin{theorem}[\cite{L0,LS}]\label{t:spin} There is a one to one correspondence
$$X_{\ell,d}\longleftrightarrow \mathsf{Irrep} W_{n}, \text{ where } n=\frac{\ell-d(2d-1)}4.$$
\end{theorem} 

\begin{remark}

\begin{enumerate}
\item It is not hard to see using a generating functions argument
(\cite{L0}) that the two sets in Theorem \ref{t:spin} have the same
cardinality. Moreover, a slight modification of that argument shows that the number of
distinguished orbits in $X_{\ell,d}$ equals $P(n)$, the number
of partitions of $n.$ 
\item In the generalized Springer correspondence in this setting,
  there is a unique local system on each orbit in $X_{\ell,d}$ which
  enters, and this is why the correspondence can be regarded as one
  between orbits and Weyl group representations.
\item The relation between $\overline{\mathbb H}_{n,m}$ and $X_{\ell,d}$ is given by
\begin{equation} 4n+d(2d-1)=\ell,\ d=-d(4m)[m+1/4].
\end{equation}
\end{enumerate}
\end{remark}

The left to right map in Theorem \ref{t:spin} is given by an
explicit algorithm  which we recall now. We use the notation for
$\mathsf{Irrep}~W_n$ from Remark \ref{r:bipartitions}.

\begin{algorithm}[\cite{LS}]\label{al:ls}Let $\lambda=(\lambda_1,\lambda_2,\dots,\lambda_k)$ be a
partition of $\ell$ of defect $d.$ Here, the convention is $0<\lambda_1<\lambda_2<\dots<\lambda_k.$ We will produce inductively a
bipartition $\rho(\lambda)$ of $n=\frac {\ell-d(2d-1)}4$, which
parameterizes an element of $\mathsf{Irrep} W_n.$ Define the (smaller)
partition $\mu$ as follows:
\begin{enumerate}
\item[(i)] if $\lambda_p$ is odd, then set
  $\mu=(\lambda_1,\dots,\lambda_{p-1})$;
\item[(ii)] if $\lambda_p$ is even, then set $\mu=(\lambda_1,\dots,\lambda_{p-2})$.
\end{enumerate}
By induction $\rho(\mu)$ is known, say it is of the form
$\rho(\mu)=\{(\gamma)(\delta)\}$ for {some (ordered) pair of partitions}
$(\gamma),(\delta).$

\begin{enumerate}
\item[(a)] $d(\lambda_p)=0$ ($\lambda$ is even). Set
  $e=[(\lambda_p+2)/4]-d(\mu)$ and $f=[\lambda_p/4]+d(\mu).$ (Note that
  $e+f=\frac {\lambda_p}2$.)
\begin{enumerate}
\item[(a1)] If $d(\mu)>0$, set $\rho(\lambda)=\{(\gamma,e)(\delta,f)\}.$
\item[(a2)] If $d(\mu)\le 0$, set $\rho(\lambda)=\{(\gamma,f)(\delta,e)\}.$
\end{enumerate}

\item[(b)] $d(\lambda_p)=1$ ($\lambda_p\equiv 1$ mod $4$). Set
  $e=\frac {\lambda_p-1}4-d(\mu).$

\begin{enumerate}
\item[(b1)] If $d(\mu)>0,$ set $\rho(\lambda)=\{(\gamma,e)
  (\delta)\}$.
\item[(b2)] If $d(\mu)=0,$ set $\rho(\lambda)=\{(\delta,e)
  (\gamma)\}.$
\item[(b3)] If $d(\mu)<0,$ set $\rho(\lambda)=\{(\gamma)
  (\delta,e)\}.$ 
\end{enumerate}

\item[(c)] $d(\lambda_p)=-1$ ($\lambda_p\equiv 3$ mod $4$). Set
  $e=\frac{\lambda_p-3}4+d(\mu)$.

\begin{enumerate}
\item[(c1)] If $d(\mu)>1$, set $\rho(\lambda)=\{(\gamma)(\delta,e)\}.$
\item[(c2)] If $d(\mu)=1$, set $\rho(\lambda)=\{(\delta,e)(\gamma)\}.$
\item[(c3)] If $d(\mu)<1,$ set $\rho(\lambda)=\{(\gamma,e)(\delta)\}.$
\end{enumerate}

\end{enumerate}

\end{algorithm}

\begin{theorem}[\cite{L6}]\label{t:spintemp} The tempered modules of $\overline{\mathbb H}_{n,m}$ with
  positive real central character are parameterized by the orbits in
  $X_{\ell,d}.$ The discrete series of $\overline{\mathbb H}_{n,m}$ with positive real central
  character are parameterized by the distinguished orbits in
  $X_{\ell,d}$. In particular, there are $\# P(n)$ discrete series.
\end{theorem}

In \cite{Sl}, a conjecture relating discrete series of $\overline{\mathbb
H}_{n,m}$, partitions of $n$, and Weyl group representations (a
Springer correspondence) was proposed. We explain this next.

\smallskip

{\bf Partitions of $n$ to distinguished orbits.} Let $\sigma$ be a partition of $n$. We think of $\sigma$ as left
justified Young tableau, with the length of rows decreasing, same as in \S
\ref{sec:alg}. Fill out the boxes of $\sigma$ starting at the left
upper corner with $m$ and increase by one to the right, and decrease
by one down. In this way, all the boxes on the diagonal have the entry
$m.$ Recall that $4 m\equiv 1$ or $3$ mod $4$. Let $\bar s_\sigma$ denote the collection of the absolute values of
the entries of $\sigma$ (with multiplicities), ordered
nonincreasingly. We think of $\bar s_\sigma$ as being a central character
for $\overline{\mathbb H}_{n,m}.$  (The connection with the previous sections is that $s_\sigma=q^{\bar s_\sigma}$ is a distinguished semisimple element.)

To $\bar s_\sigma$, we attach a distinguished nilpotent orbit $\mathcal
O_\sigma$ in $X_{\ell,d}$, $\ell=4n+d(2d-1),$  as
follows (we are thinking of $\mathcal O_\sigma$ as a partition of $\ell$ with
defect $d$). Let $\{m\}=m-[m]$ denote the fractional part of $m$. This
is either $1/4$ or $3/4.$ Start with the cuspidal part
$\lambda_c=\{4m-2,4m-6,\dots,4-4\{m\}\}.$ This is of the form
$\{3,7,11,\dots\}$ or $\{1,5,9,\dots\}$, depending if $\{m\}=1/4$ or
$3/4$, respectively. When $m=1/4,$ we have $\lambda_c=\emptyset.$ Note
that the defect of $\lambda_c$ is $d$, and the sum of entries in
$\lambda_c$ is $2(m+1/4)(m-1/4)=d(2d-1).$ Set $\lambda=\lambda_c.$ For
every hook in $\sigma,$ we will modify $\lambda$ so that the defect
remains the same, and the sum of entries increases by the four times
the length of
the hook. Assume there are $h$ hooks, $h\ge 1.$ For every hook $j$
starting from the exterior, denote by $r_j$ the entry in the right
extremity, and by $r_j'$ the bottom extremity. Note
that 
\begin{equation}
r_1>r_2>\dots>r_h\ge m\ge r_h'>r_{h-1}'>\dots>r_1',
\end{equation}
and that the length of the hook is $r_j-r_j'+1.$
Starting with the most interior hook, for every hook $j,$ there are
two cases:

\begin{enumerate}
\item if $r_j'\le 1/4,$ add to $\lambda$, $4r_j+2$ and $-4r_j'+2$ (they have
  opposite defect);
\item if $r_j'> 1/4,$ add to $\lambda$, $4r_j+2$, and remove $4r_j'-2$
  (they have same defect).
\end{enumerate}

The end partition $\lambda$ is $\mathcal O_\sigma.$ We summarize the
obvious properties of this construction.

\begin{claim}
The above procedure $\sigma\mapsto \mathcal O_\sigma$ is well-defined, and gives a distinguished orbit
in $X_{\ell,d}.$ Moreover, two different partitions give different
elements of $X_{\ell,d}.$
\end{claim}

\begin{example}\label{ex:43321}
Let us consider the example $n=13,$ $m=9/4$, and the partition
$\sigma=(4,3,3,2,1).$ Then $d=-2$ and $\ell=62.$ We view the partition
as:
\begin{figure}[h]\label{fig:43322}
\begin{large}
$$
\begin{tableau}
:.{\frac 94}.{\frac {13}4}.{\frac{17}4}.{\frac{21}4} \\
:.{\frac 54}.{\frac 94}.{\frac{13}4} \\
:.{\frac 14}.{\frac 54}.{\frac 94} \\
:.{-\frac 34}.{\frac 14} \\
:.{-\frac 74}\\
\end{tableau}$$
\end{large}
\caption{Partition $(4,3,3,2,1)$ for $\overline{\mathbb H}_{13,\frac
    94}$.}
\end{figure}

By the algorithm, we start with $\lambda_c=(3,7).$ There are three
hooks. The most interior
hook has $r_3=r_3'=9/4,$ so we add $11$ and remove $7$, and get
$\lambda=(3,11).$ Next $r_2=13/4,$ $r_2'=1/4$, so we add $15$ and $1$,
and get $\lambda=(1,3,11,15).$ Finally, $r_1=21/4$ and $r_1'=-7/4,$ so
we add $23$ and $9.$ Therefore $\mathcal O_\sigma=(1,3,9,11,15,23)$,
which is in $X_{62,-2}.$
\end{example}

Behind the reasoning for this algorithm is the fact that the middle
element of the nilpotent $\mathcal O_\sigma$ is obtained from the
central character $\bar s_\sigma$ and the middle element for the cuspidal
part.

\medskip

\noindent{\bf Partitions, distinguished orbits, and $W$-representations.}
Let us recall the conjecture of \cite{Sl}, and show that it is
equivalent to the \cite{LS} algorithm presented above.

\begin{algorithm}\label{al:sl}
Start with $\sigma$ a partition of $n$ viewed as before. We form a
bipartition $S_m(\sigma)=\{(\gamma)(\delta)\}$ of $n$ as follows. Begin by
setting $\gamma=\delta=\emptyset.$ Then find the largest in absolute
value entry in $\sigma$. (This is necessarily one of the extremities
of the first hook.) Remove all the boxes to the left of it in the same
row (including it), or all boxes above it in the same column
(including it). Let $x$ be the number of boxes removed. If they were
in the same row, append $x$ to $\gamma,$ if they were in the same
column, append $x$ to $\delta.$ Repeat the process until there are no
boxes left, or until there is a single box left. In the latter case,
if the entry in the single box left is positive, append $1$ to
$\gamma,$ if it is negative, append $1$ to $\delta.$ 
\end{algorithm}

\begin{example}
In Example \ref{ex:43321}, first we remove the boxes to the left of
$21/4,$ append $4$ to $\gamma,$ then the boxes to the left of $13/4$,
append $3$ to $\gamma,$ then the boxes to the left of $9/4,$ append $3$
to $\gamma,$ then the remaining boxes above $-7/4$, append $2$ to
$\delta,$ finally, the remaining box  $1/4$, so append $1$ to
$\gamma.$ So the bipartition $S_{9/4}(4,3,3,2,1)$ equals
$\{(4,3,3,1)(2)\}.$
\end{example}

\begin{claim}
For every $\sigma$ a partition of $n,$ and $4m\equiv 1,3$ mod $4$, the
$W_n$-representations (or rather, bipartitions of $n$) $S_m(\sigma)$
and $\rho(\mathcal O_\sigma)$ coincide. 
\end{claim}

In other words, the algorithm
 for the Springer correspondence of Lusztig-Spaltenstein coincides
 with the algorithm of Slooten in the case when $4m\equiv
 1,3$ mod $4.$

\begin{proof} We prove this statement by induction on $n$, the size of
  the Young tableau. The base $n=2$ is straightforward. Let $\sigma$
  be a partition of $n>2$ viewed as a Young tableau, and $\mathcal
  O_\sigma=(\lambda_1<\dots<\lambda_k)$ the orbit
  constructed before. Since this is a distinguished orbit, only cases
  (b) and (c) of Algorithm \ref{al:ls} enter. The largest entry
  in absolute value $\mathsf{max}$ is given by one of the extremities $r_1$ or
  $-r_1'$ (if $r_1'<0$) of the first
  hook. It corresponds to $\lambda_k$: $\lambda_k=4\mathsf{max}+2.$
  There are two cases.

a) Assume $\mathsf{max}=r_1.$ Then $d(\lambda_k)=-d(4m).$ Let $e'$ be the number of boxes on the
first row in $\sigma,$ so $e'=r_1-m+1.$ In \ref{al:ds}, one forms $r=[\frac
{\lambda_k}4]-d(\lambda_k) (d(\mathcal
O_\sigma)-d(\lambda_k))=[r_1+\frac 12]+d(4m)(-d(4m))([m+\frac
14]-1)=[e'+m-1+\frac 12]-[m+\frac 14]+1=e'+[m+\frac 12]-[m+\frac
14]=e'.$ So $e=e'.$ In \ref{al:sl}, $e'$ is placed in the left side
of the bipartition. We check that in \ref{al:ls}, $r$ is also placed
in the left side of the bipartition. There are two subcases: if
$d(4m)=1,$ then $d(\lambda_k)=-1,$ and so, in \ref{al:ls} (c),
$d(\mu)=-[m+\frac 14]+1\le 1$, so we are in the cases (c2,c3); if $d(4m)=-1,$ so $m\ge \frac 34$,
$d(\lambda_k)=1,$ and so, in \ref{al:ls} (b), $d(\mu)=[m+\frac
14]-1\ge 0,$ so we are in cases (b1,b2).

Now let $\tau$ be the Young tableau obtained
after removing the first row. The entry in the left upper corner is
$m-1,$ which is positive unless $m=1/4,3/4.$ If $m>1,$ we can regard
$\tau$ as a partition of $n-e$ and with $m-1$ instead of $m$. It is immediate that the corresponding
$\mathcal O_\tau$ is the same as $\mu$ in \ref{al:ls}. By induction
if $Sl(\tau)=\{(\gamma)(\delta)\}$, then
$\rho(\mu)=\{(\gamma)(\delta)\},$ and we are done. 

So consider the cases $m=1/4$ or $m=3/4.$ Then $m-1<0.$ Let $\bar\tau$
be the Young tableau which is obtained from $\tau$ by
  first taking the transpose tableau and then multiplying all the
  entries by $(-1)$. The
left upper corner of $\bar\tau$ has entry $1-m>0,$ so we can regard
$\bar\tau$ as a partition of $n-e$ for $1-m$ (not $m$), and associate $\mathcal
O_{\bar\tau}.$ Note that if $Sl(\tau)=\{(\delta)(\gamma)\},$ then
$Sl(\bar\tau)=\{(\gamma)(\delta)\}.$ The only observation left to
make is that $\mathcal O_{\bar\tau}=\mu$, where
$\mu=(\lambda_1,\dots,\lambda_{k-1})$. This follows easily from the
algorithm for $\mathcal O_{\sigma}$ and $\mathcal O_{\bar\tau}.$
b) Assume $\mathsf{max}=-q_1'>0.$ Then $d(\lambda_k)=d(4m).$ By the
same argument as in case a), one shows that $r'$, the number of boxes
in the first column, equals $r$ from \ref{al:ls}. In \ref{al:sl}, $r'$ is placed in the right side
of the bipartition. We check that in \ref{al:ls}, $e$ is also placed
in the right side of the bipartition. There are two subcases: if
$d(4m)=-1,$ so that $m\ge \frac 34$, then $d(\lambda_k)=-1,$ and so, in \ref{al:ls} (c),
$d(\mu)=[m+\frac 14]+1> 1$, so we are in the cases (c1); if $d(4m)=1,$ 
$d(\lambda_k)=1,$ and so, in \ref{al:ls} (b), $d(\mu)=-[m+\frac
14]-1< 0,$ so we are in case (b3).

If $\tau$ is the Young tableau obtained after removing the first
column, the entry in the upper left corner is $m+1>0,$ so we can
regard $\tau$ as a partition for $n-e$ and with $m+1$ instead of $m$. It is immediate that the corresponding
$\mathcal O_\tau$ is the same as $\mu$ in \ref{al:ls}. By induction
if $Sl(\tau)=\{(\gamma)(\delta)\}$, then
$\rho(\mu)=\{(\gamma)(\delta)\},$ and this concludes this case.

\end{proof}

Since we showed these algorithms yield the same $W$-representation,
let us denote it by $\rho(\sigma).$ There are two particular cases
worth mentioning. 

Assume that
$(\sigma,m)$ are such that the Springer
correspondence algorithm picks only rows, and so
$\rho(\sigma)=\{(a_1,\dots,a_k)(\emptyset)\}$, or only columns, and so
$\rho(\sigma)=\{(\emptyset)(b_1,\dots,b_k)\}.$ These are the cases we
referred as positive, respectively negative, ladders, and so by the
previous results, the discrete series representation with central
character $s_\sigma$ is irreducible as an $W$-module, and equals
$\rho(\sigma)\otimes\mathsf{sgn}.$ {(The tensoring with $\mathsf{sgn}$
is a normalization, so that the Steinberg module is $\mathsf{sgn}$ as
a $W$-representation.)}

The second particular case is that for discrete series which contain
the $\mathsf{sgn}$ $W$-representation. By the previous results,
if the hook extremities of $\sigma$ are 
$r_1>r_2>\dots>r_h\ge m\ge r_h'>r_{h-1}'>\dots>r_1',$ then in order
for $s_\sigma$ to contain the $\mathsf{sgn}$, they must satisfy:
$$r_1>-r_1'>r_2>-r_2'>\dots>r_h>-r_h'.$$
This means that $\rho(\sigma)$ is obtained as follows: remove the
first row, then the first column in the remaining tableau, then remove
the first row remaining, then the first column etc.

We remark that in these two cases, the $W$-representation attached to
$\sigma$ by the exotic Springer correspondence coincides with the
$W$-representation attached to $\sigma$ by the algorithms of
Lusztig-Spaltenstein and Slooten.

\subsection{$W$-independence of tempered modules}
Using the geometric realization and results of Lusztig for the graded
Hecke algebras arising from cuspidal local systems, one was able to
prove in \cite{Ci} a certain independence result for tempered modules
with positive real central characters. This is a generalization of the similar
result of Barbasch-Moy for Hecke algebras with equal parameters, and
it is a Hecke algebra analogue of Vogan's lowest $K$-types.

Retain the notation from \S \ref{sec:cuspidal}. We formulate this result in the setting of the graded Hecke algebra
$\overline{\mathbb H}_{n,m}$ from \S \ref{sec:cuspidal}, with
$4m\equiv 1,3$ (mod $4$). For every tempered module $\pi$ with positive real
central character, which by
Theorem \ref{t:spintemp}, corresponds to an orbit $\mathcal O_\pi\in X_{\ell,d},$  let
$\rho(\pi)$ be the generalized Springer correspondence
$W$-representation attached to $\mathcal O_\pi.$ The following result
follows from the fact that any other $W$-type appearing in the
restriction $\pi|_W$ is attached in the generalized Springer
correspondence to an orbit larger than $\mathcal O_\pi$ in the closure ordering.

\begin{proposition}[cf. \cite{Ci}]\label{p:linW} 
\begin{enumerate}
\item There is a bijection $\pi\mapsto
  \rho(\pi)$ between tempered modules $\overline{\mathbb H}_{n,m}$,
  $4m\equiv 1,3$ (mod $4$) with positive real central character and
  $\mathsf{Irrep}~ W.$ 
\item The set of positive real tempered $\overline{\mathbb H}$-modules viewed in $R(W)$ is linearly
  independent. Moreover, in the ordering coming from the closure
  ordering in $X_{\ell,d}$, the change of basis matrix to $\mathsf{Irrep}~ W$ is
  uni-triangular.
\end{enumerate}
\end{proposition}

Theorem \ref{eDL} allows us to extend this result to all generic
real positive
parameters $\vec q=(-1,q^{m},q)$, for the affine Hecke algebra
of type $C_n$. (Here we use
implicitly the
correspondence between the affine Hecke algebra and the graded Hecke
algebra for positive real central characters.) 

\begin{corollary}\label{c:lin} The set of tempered $\mathbb H_a$-modules for
  generic positive real $a$ is $W$-independent in $R(W).$ Moreover, in the
  ordering coming from the generalized Springer correspondence, the
  change of basis matrix to $\mathsf{Irrep} W$ is uni-triangular.  
\end{corollary}

\begin{proof}
Let $m$ be in the open interval $(\frac k2,\frac {k+1}2),$ for some
integer $k\ge 0$ and $a$ be the corresponding generic parameter. Let
$\mathsf{MP}(m)$ be the set of (exotic) marked partitions which
parameterizes the set of tempered $\mathbb H_a$-modules. Set
$m_0=\frac{2k+1}2,$ and fix $\tau\in \mathsf{MP}(m_0).$ The results
in this paper imply that the set $\mathsf{MP}(m)$ is the same for all
$m\in (\frac k2,\frac {k+1}2).$ Moreover, if we denote by
$\mathsf{temp}_m(\tau)$ the tempered module parameterized by $\tau$ at
the parameter $m,$ then
\begin{equation}\label{eq:rigid}
\mathsf{temp}_m(\tau)\cong \mathsf{temp}_{m'}(\tau) \text{ as $W$-modules},
\end{equation}
for any $m,m'\in (\frac k2,\frac {k+1}2).$ In particular,
$\mathsf{temp}_m(\tau)\equiv \mathsf{temp}_{m_0}(\tau),$ for all
$\tau\in \mathsf{MP}(m)=\mathsf{MP}(m_0).$ Then the claim follows from
Proposition \ref{p:linW}.
 \end{proof}

\begin{remark} In the preprint \cite{So}, the first part of Corollary
  \ref{c:lin}, the $W$-independence (but not the uni-triangularity),
  is proved independently  for arbitrary graded
  Hecke algebras by homological methods.
\end{remark}

\begin{remark} One can ask naturally if a similar uni-triangular
  correspondence as in Corollary \ref{c:lin} holds if one considers
  instead the exotic Springer correspondence (see
    \cite{K2} for an explicit algorithm). This is not the case
  however: in general, the map assigning to a tempered $\mathbb
  H_s$-module its exotic Springer representation is not one-to-one, as
  one can see in the example $n=4,$ and $0<m<1/2$ for the partitions of
  $n$, $\sigma_1=(1,1,2)$ and $\sigma_2=(1,1,1,1)$. The exotic Springer map assigns
  the $W_4$-representation $\{ (1^4)(0) \}$ to both
  $\mathsf{ds}(\sigma_1)$ and $\mathsf{ds}(\sigma_2)$, while the
  generalized Springer map assigns $\{(1^3)(1)\}$ to
  $\mathsf{ds}(\sigma_1)$ and $\{(1^4)(0)\}$ to
  $\mathsf{ds}(\sigma_2)$. 
\end{remark}

In particular, the ``lowest $W$-type'' correspondence of Corollary \ref{c:lin} shows that the construction of Theorem \ref{indADt} exhausts all tempered modules in the real positive generic range.

\begin{corollary}\label{c:exh} Every tempered $\mathbb H_a$-module for generic positive real $a$ is obtained by induction as in Theorem \ref{indADt}.
\end{corollary}

\begin{proof}
It is sufficient to show that Theorem \ref{indADt} produces
$\# \mathsf{Irrep} W_n$ distinct tempered modules. Let $\mathcal P(k)$ denote
the number of partitions of $k$, and $\mathcal P_2(k)$ denote the
number of bipartitions of $k,$ i.e., $\mathcal P_2(k)=\# \mathsf{Irrep} W_k.$ For every $1\le n_1\le n,$ a
tempered $\mathbb H_a$-module is constructed from a tempered $GL(n_1)$
module and a discrete series of $Sp(2n_2)$, where $n_2=n-n_1.$ There
are $\mathcal P(n_1)$ tempered modules of $GL(n_1)$ and $\mathcal
P(n_2)$ discrete series of $Sp(2n_2).$ Therefore we get
$\sum_{n_1=1}^n \mathcal P(n_1)\mathcal P(n-n_1)=\mathcal P_2(n)$
tempered $\mathbb H_a$-modules. 
These
are all distinct $\mathbb H_a$-modules since they are nonisomorphic as
$W_n$-modules.
\end{proof}

\subsection{One $W$-type discrete series} 
 We show that the only tempered $\mathbb H_{n,m}$-modules with real
 positive generic parameter which are
irreducible as $W$-modules are the $(\pm)$-ladder
representations (see \S \ref{sec:ladder}). Any tempered module which is not a discrete
series is obtained by parabolic unitary induction from a discrete
series module of a proper parabolic Hecke subalgebra. Therefore, no such
module could be $W$-irreducible, so we can restrict to the case of
discrete series, and we can restrict to the equivalent setting of
$\overline {\mathbb H}_{n,m}$-modules.

Let $\overline {\mathbb H}_{n} ^{\mathsf A}$ be the one-parameter graded Hecke
  algebra for $GL(n),$ viewed as a subalgebra of $\overline{\mathbb
    H}_{n,m}.$ We have that $\overline {\mathbb H}_{n} ^{\mathsf A}$ is
    generated by $\{\epsilon_1,\dots,\epsilon_n\}$ and
    $\{t_{i,i+1}:1\le i\le n-1\}$, where $t_{ij}$ denotes the generator in   corresponding to the
  reflection $s_{\epsilon_i-\epsilon_j}$. The following lemma is
  well-known and easy to prove by direct computation.

\begin{lemma}\label{l:evalA} There is a surjective algebra map $\phi:\overline {\mathbb
    H}_{A_{n-1}}\mapsto \mathbb C\mathfrak S_n,$ given on generators by
\begin{align}
\phi(t_{i,i+1})&= s_{i,i+1},\\\notag
\phi(\epsilon_j)&=s_{j,j+1}+s_{j,j+2}+\dotsb+s_{j,n}.
\end{align}
\end{lemma}

Note that $\phi$ allows us to lift any irreducible
$\mathfrak S_n$-representation to an irreducible
$\overline {\mathbb H}_{n} ^{\mathsf A}$-module. For $\sigma$ a partition of
  $n,$ let $\phi^*(\sigma)$ denote the irreducible $\overline{\mathbb
    H}_{A_{n-1}}$-module obtained in this way from lifting
  $\sigma\otimes\mathsf{sgn}.$ 

A simple modification of $\phi$ lifts any irreducible
$\mathfrak S_n$-representation to an irreducible $\overline{\mathbb
    H}_{n,m}$-module. The following statement can be viewed as a particular
case of the construction in \cite{BM}.

\begin{lemma}\label{l:evalB} Let $\eta\in\{+1,-1\}$ be given and let
  $\sigma$ be a fixed partition of $n$. The
  assignment 
\begin{align}\notag
t_{i,i+1}&\mapsto \phi^*(\sigma)(t_{i,i+1}),\qquad \ 1\le i\le n-1,\\\notag
\epsilon_i&\mapsto \eta m\mathsf{Id}+\phi^*(\sigma)(\epsilon_i),\ 1\le i\le n-1,\\\notag
 t_n&\mapsto \eta \mathsf{Id},\\
 \epsilon_n&\mapsto \eta m\mathsf{Id},
\end{align}
gives an irreducible $\overline{\mathbb H}_{n,m}$-module, $\pi(\sigma,\eta).$
\end{lemma}

\begin{proof} By Lemma \ref{l:evalA}, we only need to check that the
  Hecke relations 
\begin{align}
t_n\cdot\epsilon_n=-\epsilon_nt_n+2m,\\\notag
t_{n-1}\cdot \epsilon_n=\epsilon_{n-1}t_{n-1}-1,
\end{align}
are satisfied for this assignment. This is straightforward.
\end{proof}

Note that $\pi(\sigma,+)$ equals $\{(\sigma)
(\emptyset)\} \otimes\mathsf{sgn}$ as a $W_n$-representations, while
$\pi(\sigma,-)$ equals $\{(\sigma)(\emptyset) \}$,
in the bipartition notation of $W_n$-representations from \S
\ref{sec:cuspidal}. We show that these are precisely the
$(\pm)$-ladder representations from \S \ref{sec:ladder}.

\begin{theorem}\label{t:irrW} Let $s_\sigma$ be a distinguished central
  character for ${\mathbb H}_{n,m}$, where $m$ is generic. 
Then  $\mathsf{ds}(s_\sigma)$ is irreducible as a
  $W_n$-representation if and only if it is a $(\pm)$-ladder representation.   
\end{theorem}

\begin{proof}We prove the claims in the equivalent setting of $\overline{\mathbb H}_{n,m}.$ In one direction, let us assume that $\mathsf{ds}(s_\sigma)$ is a positive
  ladder. The proof for the other case is analogous. We wish to show that $\mathsf{ds}(s_\sigma)=\pi(\sigma,+).$
  A direct proof of this fact would be to compute the central
  character of $\pi(\sigma,+)$ and show that it is $s_\sigma.$
  We give an indirect proof.
  In the bijection of Corollary \ref{c:lin}, 
  $\mathsf{ds}(s_\sigma)$ (contains and) corresponds to the
  $W_n$-representation $\rho(\mathsf{ds}(s_\sigma))=\{(\sigma)
(\emptyset)\}\otimes\mathsf{sgn}.$ By the Lusztig classification \cite{L3},
  $\pi(\sigma,+)$ is the unique irreducible quotient of a standard
  module $M=M_{a_\sigma,\mathcal O_\sigma}.$ There is a continuous
  deformation of $s_\sigma\longrightarrow s_0$, and
  $M\longrightarrow M_0$ such that $M_0$
  is a tempered module at the semisimple element 
  and $M_0|_{W_n}=M|_{W_n}.$ Moreover, the
  tempered module $M_0$ must contain $\rho(\mathsf{ds}(s_\sigma))$,
  and in Proposition \ref{p:linW},
  $\rho(M_0)=\rho(\mathsf{ds}(s_\sigma)).$ But this implies that
  $M_0=\mathsf{ds}(s_\sigma).$ Since this is a discrete series, we
  then have $\pi(\sigma,+)=M=M_0=\mathsf{ds}(s_\sigma).$

To verify the converse claim, recall that in \cite{BM}, one determined which $W_n$-representations
  can be extended to hermitian graded Hecke algebra modules. When the
  Hecke algebra is $\overline{\mathbb H}_{n,m},$ the only cases are
  $W$-representations of the form $\{ (\gamma)(\emptyset) \}$,
  $\{ (\emptyset)(\delta) \},$ or {$\{ (d^{k})
  (f^{l}) \}$}, when $k-d=l-f+m.$

 Note that, using 
  Algorithm \ref{al:sl}, it is immediate that there is no discrete
  series $\mathsf{ds}(s_\sigma)$ such that {$\rho(\mathsf{ds}(s_\sigma))=\{(d^k)(f^l)\}$}, for $k>1$ and $l>1.$ We
  check the case $k=l=1,$ $d>0,f>0.$ In order to have
  $\rho(\mathsf{ds}(s_\sigma))=\{ (d)(f) \},$ by Algorithm \ref{al:sl},
  we must have $\sigma$ a one hook partition, with the largest two
  entry values at the two extremities of the hook. If the largest entry
  is the right extremity of the hook, then by \ref{c:sgn},
  $\mathsf{ds}(s_\sigma)$ also contains the $\mathsf{sgn}$
  $W_n$-representation, so it is not $W_n$-irreducible. So it remains to consider the
  case when the largest entry is in the bottom extremity of the
  hook. In that case, by Example \ref{ex:onehook}, the exotic Springer
  correspondence attaches to $\mathsf{ds}(s_\sigma)$, the
  {$W_n$-representation} given by the bipartition $\{(\emptyset)
  (n)\}$. So again $\mathsf{ds}(s_\sigma)$ contains at least two irreducible $W_n$-representations.
\end{proof}

\begin{corollary}\label{c:4.6} Let  $s_\sigma$ be a distinguished
  central character for ${\mathbb H}_{n,m}$. 
\begin{enumerate}
\item If $m>n-1$, then we have
  $\mathsf{ds}(s_\sigma)|_{W_n}=\{(\sigma)(\emptyset)\}\otimes\mathsf{sgn}$.
\item If $m<-n+1$, then we have
  $\mathsf{ds}(s_\sigma)|_{W_n}=\{(\sigma)(\emptyset)\}$.
\end{enumerate}
\end{corollary}

\begin{proof}
By Corollary \ref{c:ladder},  $\mathsf{ds}(s_\sigma)$
is a positive ladder, if $m>n-1$, and it is a negative ladder, if
$m<-n+1$. The proof of Theorem \ref{t:irrW} implies that
$\mathsf{ds}(s_\sigma)=\pi(\sigma,+)$, when $m>n-1$ and
$\mathsf{ds}(s_\sigma)=\pi(\sigma,-)$, when $m<-n+1$. Then the claims
follow from the remark after Lemma \ref{l:evalB}.
\end{proof}

\subsection{Closure relation of orbits}\label{clrelation}

Fix a generic $m$. Let $\sigma$ be a partition of $n$. Attached to
$\rho ( \sigma )$ and $\mathsf{ds} ( \sigma )$, we have irreducible
$W$-modules $L _{\sigma}$ and $E _{\sigma}$, respectively. For an
irreducible $W$-module $K$, we denote by $\mathbb O ( K )$ the
nilpotent orbit of $\mathop{Spin} ( \ell )$ corresponding to $K$ via
(the inverse of) a generalized Springer correspondence (c.f. \S
\ref{sec:cuspidal}). Let $\mathcal C$ be the pair $( \mathbb O,
\mathcal L )$ of $\mathop{Spin} ( \ell )$-orbit of $\mathfrak{so} (
\ell )$ and the local system which contribute to the generalized
Springer correspondence. Let $\mathcal O ( K ) \subset \mathfrak N ^{a
  _0}$ be the $G$-orbit corresponding to $K$ via Theorem
\ref{t:1.15}. Let $\mathsf{pr} : \mathbb V \to V _2$ be the
$G$-equivariant projection map. 

\begin{theorem}\label{clel}
In the above setting:
\begin{enumerate}
\item We have $[ \mathsf{ds} ( \sigma ) : L _{\sigma} ] = 1 = [ \mathsf{ds} ( \sigma ) : E _{\sigma} ]$ as $W$-modules;
\item For each irreducible $W$-submodule $K$ of $\mathsf{ds} ( \sigma )$, we have $\mathcal O ( L _{\sigma} ) \subset \overline{\mathcal O ( K )}$ and $\mathbb O ( E _{\sigma} ) \subset \overline{\mathbb O ( K )}$. In particular, we have $\mathcal O ( L _{\sigma} ) \subset \overline{\mathcal O ( E _{\sigma} )}$ and $\mathbb O ( E _{\sigma} ) \subset \overline{\mathbb O ( L _{\sigma} )}$;
\item Let $E$ be an irreducible $W$-module such that $\mathsf{pr} ( \mathcal O ( E ) ) = \mathsf{pr} ( \mathcal O ( E _{\sigma}) )$. Then, we have $[ \mathsf{ds} ( \sigma ) : E ] {\le 1}$ as $W$-modules.
\end{enumerate}
\end{theorem}

\begin{proof} Recall that both of the constructions of Lusztig \cite{L3} and \cite{K} depend on the realization of $\mathbb H _{a _{\sigma}}$ in terms of the self-extension algebras of certain complexes. Let $\mathsf{IC} ( \mathbb O, \mathcal M )$ be the minimal extension of a local system $\mathcal M$ on $\mathbb O$. In Lusztig's case, we have (by \cite{L3}):
\begin{itemize}
\item[i)] There exists a semi-simple element $\mathbf a _{\sigma} \in \mathop{Spin} ( \ell ) \times \mathbb C ^{\times}$. Let us denote by $\mathcal C _{\sigma}$ the set of pairs $( \mathbb O, \mathcal L \MID _{\mathbb O} )$ obtained from $\mathcal C$ by taking (connected component of) $\mathbf a _{\sigma}$-fixed points;
\item[ii)] Define $G ^L := Z _{\mathop{Spin} ( \ell ) \times \mathbb C ^{\times}} ( \mathbf a _{\sigma} )$. Then, each element of $\mathcal C _{\sigma}$ is a single $G ^L$-orbit with a local system. The set $\mathcal C _{\sigma}$ is in bijection with $\mathsf{Irrep} \mathbb H _{a _{\sigma}}$;
\item[iii)] For each $( \mathbb O, \mathcal L \MID _{\mathbb O} ) \in \mathcal C _{\sigma}$, we define $\mathbb O ^{\sim} := \mathop{Spin} ( \ell ) \mathbb O \subset \mathfrak{so} ( \ell )$. It defines a $W$-representation $\rho ( \mathbb O )$ via a generalized Springer correspondence;
\item[iv)] The standard module $M ( \mathbb O )$ contains an irreducible $W$-module $K$ with multiplicity (as $W$-modules) equal to $\dim H ^{\bullet} _{\mathbb O ^{\sim}} ( \mathsf{IC} ( \mathbb O ( K ), \mathcal L ) )$;
\item[v)] The standard module $M ( \mathbb O )$ has a unique simple quotient $L ( \mathbb O )$ and we have $[ M ( \mathbb O ) : L ( \mathbb O' ) ] = \dim H ^{\bullet} _{\mathbb O} ( \mathsf{IC} ( \mathbb O', \mathcal L \MID _{\mathbb O'} ) )$.
\end{itemize}
Now the assertion $[ \mathsf{ds} ( \sigma ) : L _{\sigma} ] = 1$
follows by the combination of {\bf iv)} and {\bf v)}. The assertion $
[ \mathsf{ds} ( \sigma ) : E _{\sigma} ] = 1$ follows by the
construction of $L _{(a, X)}$ and Theorem \ref{eDL}. We have $\mathbb
O ( E _{\sigma} ) \subset \overline{\mathbb O ( L _{\sigma} )}$ by the
combination of {\bf iv)} and {\bf v)}. We have $\mathcal O ( L
_{\sigma} ) \subset \overline{\mathcal O ( E _{\sigma} )}$ by Corollary \ref{Green} and Theorem \ref{mult} applied to $a = a_0$. This proves {\bf 1)} and {\bf 2)}. We prove {\bf 3)}. Notice that $\mathsf{pr} ( \mathcal O ( E ) ) = \mathsf{pr} ( \mathcal O ( E _{\sigma}) )$ implies either
$\overline{\mathcal O ( E )} \cap \mathcal O ( E _{\sigma}) =
\emptyset$ or $\mathcal O ( E )$ is a (open dense subset of a) vector
bundle over $\mathcal O ( E _{\sigma} )$. By {\bf 2)}, it suffices
to consider the latter case. By Corollary \ref{Green} and Theorem \ref{mult} applied to $a = a_0$, we have $[ M _{(a
  _{\sigma}, X )}: L _{Y} ] = 1$ as $W$-modules for $Y \in \mathcal O ( E )$ as required.
\end{proof}

\begin{remark}
To use Theorem \ref{clel}, one needs to know the Weyl group representation attached to each orbit and the closure relations between orbits. These are contained in \cite{K2} and Achar-Henderson \cite{AH}, respectively.
\end{remark}

\end{document}